\documentclass[12pt]{amsart}

\usepackage[letterpaper,margin=.6in]{geometry}
\usepackage[colorlinks,linkcolor=blue,citecolor=black!75!red]{hyperref}

\usepackage{url}
\usepackage{color}

\usepackage[all]{xypic}

\usepackage{latexsym}

\usepackage{amssymb}
\usepackage{amsfonts}
\usepackage{amscd}
\usepackage{amsmath,amsthm}
\usepackage{verbatim}

\usepackage{tikz}
\usetikzlibrary{matrix,arrows,decorations.pathmorphing,fit,cd}
\tikzset{myboxgroup/.style={draw, densely dotted}} 

\usepackage{cleveref}

\newtheorem{lemma}{Lemma}[section]
\newtheorem{proposition}[lemma]{Proposition}
\newtheorem{theorem}[lemma]{Theorem}
\newtheorem{corollary}[lemma]{Corollary}

\newtheorem{convention}[lemma]{Convention}

{

}

\theoremstyle{definition}

\theoremstyle{remark}

\newtheorem{remark}[lemma]{Remark}

\makeatletter
\let\xx@thm\@thm
\AtBeginDocument{\let\@thm\xx@thm}
\makeatother

\crefname{section}{section}{sections}
\Crefname{section}{Section}{Sections}
\crefformat{section}{#2section~#1#3}
\Crefformat{section}{#2Section~#1#3}

\crefformat{subsection}{#2\S#1#3}
\Crefformat{subsection}{#2\S#1#3}
\crefrangeformat{subsection}{\S\S#3#1#4--#5#2#6}
\Crefrangeformat{subsection}{\S\S#3#1#4--#5#2#6}
\crefmultiformat{subsection}{\S\S#2#1#3}{ and~#2#1#3}{, #2#1#3}{ and~#2#1#3}
\Crefmultiformat{subsection}{\S\S#2#1#3}{ and~#2#1#3}{, #2#1#3}{ and~#2#1#3}
\crefrangemultiformat{subsection}{\S\S#3#1#4--#5#2#6}{ and~#3#1#4--#5#2#6}{, #3#1#4--#5#2#6}{ and~#3#1#4--#5#2#6}
\Crefrangemultiformat{subsection}{\S\S#3#1#4--#5#2#6}{ and~#3#1#4--#5#2#6}{, #3#1#4--#5#2#6}{ and~#3#1#4--#5#2#6}

\crefname{definition}{Definition}{Definitions}
\crefformat{definition}{#2Definition~#1#3}
\Crefformat{definition}{#2Definition~#1#3}

\crefname{definitionnodiamond}{Definition}{Definitions}
\crefformat{definitionnodiamond}{#2Definition~#1#3}
\Crefformat{definitionnodiamond}{#2Definition~#1#3}

\crefname{example}{Example}{Examples}
\crefformat{example}{#2Example~#1#3}
\Crefformat{example}{#2Example~#1#3}

\crefname{examplenodiamond}{Example}{Examples}
\crefformat{examplenodiamond}{#2Example~#1#3}
\Crefformat{examplenodiamond}{#2Example~#1#3}

\crefname{remark}{Remark}{Remarks}
\crefformat{remark}{#2Remark~#1#3}
\Crefformat{remark}{#2Remark~#1#3}

\crefname{remarknodiamond}{Remark}{Remarks}
\crefformat{remarknodiamond}{#2Remark~#1#3}
\Crefformat{remarknodiamond}{#2Remark~#1#3}

\crefname{convention}{Convention}{Conventions}
\crefformat{convention}{#2Convention~#1#3}
\Crefformat{convention}{#2Convention~#1#3}

\crefname{notation}{Notation}{Notations}
\crefformat{notation}{#2Notation~#1#3}
\Crefformat{notation}{#2Notation~#1#3}

\crefname{notationnodiamond}{Notation}{Notations}
\crefformat{notationnodiamond}{#2Notation~#1#3}
\Crefformat{notationnodiamond}{#2Notation~#1#3}

\crefname{lemma}{Lemma}{Lemmas}
\crefformat{lemma}{#2Lemma~#1#3}
\Crefformat{lemma}{#2Lemma~#1#3}

\crefname{proposition}{Proposition}{Propositions}
\crefformat{proposition}{#2Proposition~#1#3}
\Crefformat{proposition}{#2Proposition~#1#3}

\crefname{corollary}{Corollary}{Corollaries}
\crefformat{corollary}{#2Corollary~#1#3}
\Crefformat{corollary}{#2Corollary~#1#3}

\crefname{theorem}{Theorem}{Theorems}
\crefformat{theorem}{#2Theorem~#1#3}
\Crefformat{theorem}{#2Theorem~#1#3}

\crefname{assumption}{Assumption}{Assumptions}
\crefformat{assumption}{#2Assumption~#1#3}
\Crefformat{assumption}{#2Assumption~#1#3}

\crefname{enumi}{}{}
\crefformat{enumi}{(#2#1#3)}
\Crefformat{enumi}{(#2#1#3)}

\crefname{equation}{}{}
\crefformat{equation}{(#2#1#3)}
\Crefformat{equation}{(#2#1#3)}

\crefname{align}{}{}
\crefformat{align}{(#2#1#3)}
\Crefformat{align}{(#2#1#3)}

\crefname{proofstep}{Step}{Steps}
\crefformat{proofstep}{#2Step~#1#3}
\Crefformat{proofstep}{#2Step~#1#3}

\crefname{table}{Table}{Tables}
\crefformat{table}{#2Table~#1#3}
\Crefformat{table}{#2Table~#1#3}

\numberwithin{equation}{section}

\newenvironment{pf}{\begin{proof}}{\end{proof}}

\def\CC{{\mathbb C}}

\def\NN{{\mathbb N}}

\def\PP{{\mathbb P}}
\def\QQ{{\mathbb Q}}

\def\ZZ{{\mathbb Z}}

\newcommand{\bbZ}{\mathbb{Z}}
\newcommand{\bbC}{\mathbb{C}}
\newcommand{\bbP}{\mathbb{P}}

\def\0ol{{\bar 0}}
\def\1ol{{\bar 1}}
\def\2ol{{\bar 2}}
\def\ol2{{\bar 2}}
\def\3ol{{\bar 3}}
\def\4ol{{\bar 4}}
\def\5ol{{\bar 5}}
\def\6ol{{\bar 6}}
\def\7ol{{\bar 7}}
\def\8ol{{\bar 8}}
\def\9ol{{\bar 9}}

\def\bold0{{\bf 0}}
\def\bold1{{\bf 1}}
\def\bold2{{\bf 2}} 
\def\bold3{{\bf  3}}
\def\bold4{{\bf 4}}
\def\bold5{{\bf 5}}
\def\bold6{{\bf 6}}
\def\bold7{{\bf 7}}
\def\bold8{{\bf 8}}
\def\bold9{{\bf 9}}

\def\P2Skly{\PP^2_{Skly}}

\def\a{\alpha}
\def\b{\beta}
\def\c{\gamma}
\def\d{\delta}

\def\g{\gamma}

\def\s{\sigma}

\def\ve{\varepsilon}

\def\D{\Delta}
\def\G{\Gamma}
\def\L{\Lambda}

\def\fd{{\mathfrak d}}

\def\fsl{{\mathfrak s}{\mathfrak l}}

\def\sfk{{\sf k}}
\def\sfl{{\sf l}}
 
\def\sfn{{\sf n}}

\def\sfz{{\sf z}}

\def\cal{\mathcal}

\def\cA{{\cal A}}
\def\cB{{\cal B}}
\def\cC{{\cal C}}

\def\cE{{\cal E}}
\def\cF{{\cal F}}
\def\cG{{\cal G}}

\def\cI{{\cal I}}

\def\cK{{\cal K}}
\def\cL{{\cal L}}

\def\cM{{\cal M}}
\def\cN{{\cal N}}
\def\cO{{\cal O}}
\def\cP{{\cal P}}

\def\cR{{\cal R}}

\def\cU{{\cal U}}
\def\cV{{\cal V}}

\def\cZ{{\cal Z}}

\def\Fdim{\operatorname{\sf Fdim}}
\def\Gr{\operatorname{\sf Gr}}
\def\QGr{\operatorname{\sf QGr}}

\def\Proj{\operatorname{Proj}}
\def\Projnc{\operatorname{Proj}_{nc}}

\def\Hom{\operatorname {Hom}}
\def\Aut{\operatorname{Aut}}

\def\id{\operatorname{id}}

\def\coh{{\sf coh}}

\def\Qcoh{{\sf Qcoh}}
\def\rank{{\rm rank}}

\def\th{{\rm th}}

\def\dirlim{\mathop{\vtop{\baselineskip -100pt\lineskip -1pt\lineskiplimit 0pt
\setbox0\hbox{lim}\copy0\hbox to \wd0{\rightarrowfill}}}\limits}
\def\invlim{\mathop{\vtop{\baselineskip -100pt\lineskip -1pt\lineskiplimit 0pt
\setbox0\hbox{lim}\copy0\hbox to \wd0{\leftarrowfill}}}\limits}

\def\I11{{1 \kern -0.8pt \! \mbox{l}}}
\def\mumu{{\mu\kern-4.2pt\mu}}
\def\bfmu{{\mu\kern-4.2pt\mu}}
\def\2slash{\backslash \! \backslash}

\def\rk{\mathrm{rank}}
\def\pic{\mathrm{Pic}}
\def\NS{\mathrm{NS}}

\let\originalleft\left
\let\originalright\right
\renewcommand{\left}{\mathopen{}\mathclose\bgroup\originalleft}
\renewcommand{\right}{\aftergroup\egroup\originalright}

\newcommand\spr[1]{\cite[\href{https://stacks.math.columbia.edu/tag/#1}{Tag {#1}}]{stacks-project}}

\makeatletter
\def\l@subsection{\@tocline{2}{0pt}{2.75pc}{5pc}{}}
\makeatother

\begin{document}

\title[Maps from elliptic algebras to twisted homogeneous coordinate rings]{Maps from Feigin and Odesskii's elliptic algebras\\
to twisted homogeneous coordinate rings}

\author{Alex Chirvasitu, Ryo Kanda, and S. Paul Smith}

\address[Alex Chirvasitu]{Department of Mathematics, University at
  Buffalo, Buffalo, NY 14260-2900, USA.}
\email{achirvas@buffalo.edu}

\address[Ryo Kanda]{Department of Mathematics, Graduate School of Science, Osaka City University, 3-3-138, Sugimoto, Sumiyoshi, Osaka, 558-8585, Japan.}
\email{ryo.kanda.math@gmail.com}

\address[S. Paul Smith]{Department of Mathematics, Box 354350,
  University of Washington, Seattle, WA 98195, USA.}
\email{smith@math.washington.edu}




\begin{abstract}
The elliptic algebras in the title are connected graded $\mathbb{C}$-algebras, denoted $Q_{n,k}(E,\tau)$, depending on a pair of relatively prime integers $n>k\ge 1$, an elliptic curve $E$, and a point $\tau\in E$. This paper examines a canonical homomorphism from $Q_{n,k}(E,\tau)$ to the twisted homogeneous coordinate ring $B(X_{n/k},\sigma',\mathcal{L}'_{n/k})$ on the characteristic variety $X_{n/k}$ for $Q_{n,k}(E,\tau)$. When $X_{n/k}$ is isomorphic to $E^g$ or the symmetric power $S^gE$ we show the homomorphism $Q_{n,k}(E,\tau) \to B(X_{n/k},\sigma',\mathcal{L}'_{n/k})$ is surjective, that the relations for $B(X_{n/k},\sigma',\mathcal{L}'_{n/k})$ are generated in degrees $\le 3$, and the non-commutative scheme $\mathrm{Proj}_{nc}(Q_{n,k}(E,\tau))$ has a closed subvariety that is isomorphic to $E^g$ or $S^gE$, respectively. When $X_{n/k}=E^g$ and $\tau=0$, the results about $B(X_{n/k},\sigma',\mathcal{L}'_{n/k})$ show that the morphism $\Phi_{|\mathcal{L}_{n/k}|}:E^g \to \mathbb{P}^{n-1}$ embeds $E^g$ as a projectively normal subvariety that is a scheme-theoretic intersection of quadric and cubic hypersurfaces.
\end{abstract}

\maketitle

\tableofcontents{}

\section{Introduction}

For a fixed $n$ and $k$, the elliptic algebras $Q_{n,k}(E,\tau)$, defined by Feigin and Odesskii in 1989 \cite{FO89}, 
are non-commutative deformations of the polynomial ring on $n$ variables.
Twisted homogeneous coordinate rings are non-commutative analogues (often deformations) 
of homogeneous coordinate rings (more precisely, section rings) for projective algebraic varieties.
This paper uses the latter to study the former.

\subsection{The contents of this and other papers}
Always, $n$ and $k$ denote relatively prime integers,  $n>k\ge 1$, $E=\CC/\L$ is a complex elliptic curve, and $\tau\in E$ is a (closed) point. We sometimes regard $\tau$ as a translation automorphism $\tau:E \to E$.

This is one of several papers we are writing about the algebras $Q_{n,k}(E,\tau)$. The first of them, \cite{CKS1}, focused on their definition in terms of generators and relations and established some immediate consequences of that definition.  The second, \cite{CKS2}, examined its characteristic variety, $X_{n/k}$, a projective algebraic variety that controls a large part of the structure and representation theory of $Q_{n,k}(E,\tau)$. 
Feigin and Odesskii identified a distinguished ample invertible sheaf $\cL_{n/k}$ on $E^g$, the $g^{\rm th}$ power of $E$, 
where $g$ is the ``length'' of the negative continued fraction expression for $\frac{n}{k}$.
This sheaf is generated by its global sections, the space of which has dimension $n$, so the complete linear system $|\cL_{n/k}|$ 
provides a morphism 
$\Phi_{n/k}:E^g \to \PP^{n-1}$, the image of which is $X_{n/k}$, by definition. 
The main result in \cite{CKS2} is that $X_{n/k}$ is isomorphic to the quotient $E^g/\Sigma_{n/k}$ where 
$\Sigma_{n/k}$ is a certain finite group; furthermore, $E^g/\Sigma_{n/k}$ is a bundle over a power of 
$E$ with fibers that are products of projective spaces. 
The third of our papers, \cite{CKS5}, examines the structure of $E^g/\Sigma_{n/k}$ in more detail: 
an \'etale cover of it that is a product of projective spaces and a power of $E$; a distinguished automorphism of it and
 its \'etale cover that is induced by a translation automorphism $\s:E^g\to E^g$ that ``controls'' the non-commutativity of $Q_{n,k}(E,\tau)$.
Another, \cite{CKS4}, 
will show that $Q_{n,k}(E,\tau)$ has the same Hilbert series as the polynomial ring on $n$ variables when $\tau$ is not a torsion point of $E$.

This paper concerns homomorphisms from $Q_{n,k}(E,\tau)$ to 
non-commutative algebras $B(X,\s,\cL)$ defined in terms of a scheme 
$X$, an automorphism $\s:X \to X$, and an invertible $\cO_X$-module $\cL$.
The algebras $B(X,\s,\cL)$ are the ``twisted homogeneous coordinate rings'' in the title. They are non-commutative analogues of the section rings $\oplus_{i \ge 0} H^0(X,\cL^{\otimes i})$. 

Our main result is as follows (some of the notation  is explained later in this introduction).

\begin{theorem}
\label{thm.main.01}\leavevmode
\begin{enumerate}
  \item\label{main.thm.hom}
There are non-trivial graded $\CC$-algebra homomorphisms 
$$
\xymatrix{
Q_{n,k}(E,\tau) \ar[rr]^>>>>>>>>{\Psi_{n/k}} &&  B(X_{n/k},\s',\cL'_{n/k}) \ar[rr]^>>>>>>>>{\sim} &&  B(E^g,\s,\cL_{n/k})^{\Sigma_{n/k}}  
\; \subseteq \; B(E^g,\s,\cL_{n/k}).
}
$$
\item{}\label{main.thm.qgr}
The quotient categories $\QGr(B(X_{n/k},\s',\cL'_{n/k}))$ and $\QGr(B(E^g,\s,\cL_{n/k}))$ are equivalent to the categories
$\Qcoh(X_{n/k})$ and $\Qcoh(E^g)$, respectively.
\end{enumerate}
If all  the integers $n_1,\ldots,n_g$ in the negative continued fraction for $\frac{n}{k}$ are $\ge 3$ (resp., exactly one of $n_1$ and $n_g$
is $\ge 3$ and the other $n_i$'s are $2$), then
 \begin{enumerate}\setcounter{enumi}{2}
 \item\label{main.thm.charvar}
  $X_{n/k}$ is isomorphic to the $g^{\rm th}$ power  $E^g$ (resp., the $g^{\rm th}$ symmetric power, $S^gE$);
  \item\label{main.thm.surj}
  the homomorphism $Q_{n,k}(E,\tau) \to B(X_{n/k},\s',\cL'_{n/k})$ is surjective; 
  \newline
  equivalently, $B(X_{n/k},\s',\cL'_{n/k})$
  is generated by elements of degree one;
  \item\label{main.thm.rel.deg}
   the relations for $B(X_{n/k},\s',\cL'_{n/k})$ are generated in degrees $\le 3$;
   \item\label{main.thm.subvar}
  $X_{n/k}$ is a closed subvariety of the non-commutative scheme $\Projnc(Q_{n,k}(E,\tau))$, 
  i.e.,  there are functors
  $$
  i_*,i^!:\QGr(Q_{n,k}(E,\tau)) \, \longrightarrow \, \QGr(B(X_{n/k},\s',\cL'_{n/k}))
  $$
  and 
  $$
  i^*:\QGr(B(X_{n/k},\s',\cL'_{n/k}))  \, \longrightarrow \,   \QGr(Q_{n,k}(E,\tau))
  $$
forming an adjoint triple $i^* \dashv i_* \dashv i^!$, and  $i_*$ is a fully faithful functor whose essential image is closed under subquotients. 
\end{enumerate} 
\end{theorem}
\begin{proof}
	\cref{main.thm.hom} \cref{cor.map.to.B} and \cref{thm.thcr.qnk}\cref{thcr.qnk.act}.
	
	\cref{main.thm.qgr} \cref{thm.AV}, \cref{cor.sigma-ample.2}, \cref{2-char-var.prop}, and \cref{thm.thcr.qnk}.
	
	\cref{main.thm.charvar} \cref{sssec.sp.cases}.
	
	\cref{main.thm.surj}\cref{main.thm.rel.deg} \cref{thm.B(S^gE)}, \cref{pr.l'l''}, and \cref{0892340894}.
	
	\cref{main.thm.subvar} \cref{sect.first.g}.
\end{proof}

The perspective of non-commutative algebraic geometry is illuminating. The algebra $Q_{n,k}(E,\tau)$ is a 
homogeneous coordinate ring for a non-commutative analogue, $\Projnc(Q_{n,k}(E,\tau))$, of the projective space $\PP^{n-1}$.
The  homomorphism $Q_{n,k}(E,\tau) \to B(X_{n/k},\s',\cL'_{n/k})$ induces a ``map'' $X_{n/k} \to \Projnc(Q_{n,k}(E,\tau))$. 
When $X_{n/k}$ is $E^g$ or $S^gE$ this map is a ``closed immersion'', i.e., there are non-commutative 
analogues of the usual inverse and direct image functors that allow one to carry information from $\Qcoh(X_{n/k})$ to 
an analogous category of graded $Q_{n,k}(E,\tau)$-modules.

When $\tau=0$, \cref{thm.main.01}\cref{main.thm.rel.deg} shows that the image of $E^g$ in $\PP^{n-1}$ under $\Phi_{n/k}$
is a scheme-theoretic intersection of quadric and cubic hypersurfaces (we do not know if this follows from known results).
Thus, in a sense, the situation for $\tau\ne 0$ is 
exactly the same. That result also recovers the less well-known fact that the image of $\Phi_{|\cL|}:S^gE \to \PP^{(m-1)g}$,
where the class of $\cL$ in the N\'eron-Severi group of $S^gE$ is $D+(m-1)F$ with $m\geq 3$ (see \cref{sssect.NS.S^gE} for notation),
is a scheme-theoretic intersection of quadric and cubic hypersurfaces. Again, we note that the non-commutative case is
perfectly analogous to the classical case.

\subsection{Some of what is known about $Q_{n,k}(E,\tau)$}
\label{sect.good.props}

The algebras $Q_{n,k}(E,0)$ and $Q_{n,n-1}(E,\tau)$ are polynomial rings on $n$ variables (\cite[Props.~5.1 and 5.5]{CKS1}).
The algebras $Q_{n,1}(E,\tau)$ are commonly called Sklyanin algebras.

For a fixed $E$ and $n$, Odesskii and Feigin showed that the algebras $Q_{n,1}(E,\tau)$ provide a flat family of  deformations of 
the  polynomial ring on $n$ variables for all $\tau$ in a countable intersection of Zariski-open neighborhoods of $0$.  
Tate and Van den Bergh made a careful analysis of the algebras $Q_{n,1}(E,\tau)$ for all $\tau$ and all elliptic curves defined over an arbitrary field \cite{TvdB96}. Among other things, they showed that as $E$ and $\tau$ vary  the algebras $Q_{n,1}(E,\tau)$ 
form a flat family of deformations of the polynomial ring on $n$ variables; i.e., for all $\tau$, 
the dimensions of the homogeneous components of $Q_{n,1}(E,\tau)$ are the same as those of the polynomial ring on $n$ variables.

This paper concerns $Q_{n,k}(E,\tau)$ when $k>1$ and $\tau$ is arbitrary.

Tate and Van den Bergh showed that  $Q_{n,1}(E,\tau)$ has the following properties:
\begin{enumerate}
\item\label{tv.hilb.ser} 
It is a connected graded left and right noetherian algebra having the same Hilbert series as the polynomial ring on $n$ variables (with its standard grading).
\item\label{tv.domain} 
It has no zero divisors. 
\item\label{tv.koszul} 
It is a Koszul algebra.
\item\label{tv.fin.order}
It is a finite module over its center if and only if $\tau$ has finite order.\footnote{Tate and Van den Bergh proved that 
$Q_{n,k}(E,\tau)$ is finite over its center if $\tau$ has finite order. The converse follows from \cref{cor.commutative} and \cref{thm.main.01}\cref{main.thm.surj}.}
\item\label{tv.cm}
It is Cohen-Macaulay.
\item\label{tv.aus}
It has the Auslander property.
\item\label{tv.as.reg}
It  is an Artin-Schelter regular algebra \cite{AS87}. 
\end{enumerate}
Definitions of the last three properties can be found in \cite{Lev92}. 
We expect that every $Q_{n,k}(E,\tau)$ has these properties. In \cite{CKS4}, we show that $Q_{n,k}(E,\tau)$ has the same Hilbert series as the polynomial ring on $n$ variables and it is a Koszul algebra, provided that $\tau$ is not a torsion point.

\subsection{The category $\QGr(A)$ when $A=Q_{n,k}(E,\tau)$}
Let $\Bbbk$ be a field and $A$ a finitely generated connected graded $\Bbbk$-algebra. 
Let $\Gr(A)$ denote the category of $\ZZ$-graded left $A$-modules.
We write $\Fdim(A)$ for the full subcategory of $\Gr(A)$ consisting of those modules that are the sum of their finite dimensional 
submodules and define the quotient category
$$
\QGr(A) \; :=\; \frac{\Gr(A)}{\Fdim(A)}.
$$
If $A$ is a finitely generated commutative connected $\Bbbk$-algebra generated by its degree-one component, 
then $\QGr(A)$ is equivalent to the category of quasi-coherent sheaves on the projective scheme $\Proj(A)$. 
Even when $A$ is not commutative, the category $\QGr(A)$ often 
behaves like the category of  quasi-coherent sheaves on a projective scheme.

\subsubsection{}
\label{sect.first.g}
Suppose $\cL_{n/k}$ is very ample or, equivalently, all integers in the ``negative'' continued fraction for $\frac{n}{k}$ are $\ge 3$ (see \cref{sssec.neg.conti.frac})
 or, equivalently, the natural map $E^g \to X_{n/k}$ is an isomorphism. Then $B(X_{n/k},\s',\cL'_{n/k})=B(E^g,\s,\cL_{n/k})$
 and the homomorphism 
\begin{equation}
\label{eq:Psi.nk}
\Psi_{n/k}:Q_{n,k}(E,\tau) \; \longrightarrow \; B(E^g,\s,\cL_{n/k})
\end{equation}
in \cref{thm.main.01}
is surjective or, equivalently, $B(E^g,\s,\cL_{n/k})$ is generated by elements of degree one (\cref{ssec.prod.gen.deg.one}).
The sheaf  $\cL_{n/k}$ is $\s$-ample (see \cref{ssse.sample} and \cref{thm.thcr.qnk}\cref{thcr.qnk.ample}) so a result of Artin and Van den Bergh \cite{AV90} (see \cref{AVdB.thm})
tells us that  
$\QGr(B(E^g,\s,\cL_{n/k}))$ is equivalent to $\Qcoh(E^g)$. Combining this with the main result in \cite{SPS15} (see \cite[Thm.~1.2]{SPS15-corrigendum}) implies there are
 functors
\begin{equation}
\label{eq:tclosed.immersion}
\xymatrix{
\Qcoh(E^g) \ar[rr]_{i_*}  &&   \QGr(Q_{n,k}(E,\tau)) \ar@/_1.5pc/[ll]_{i^*,\,i^!}
}
\end{equation}
satisfying the properties in \cref{thm.main.01}\cref{main.thm.subvar}. The claim for the cases $\frac{n}{k}=[m,2,\ldots,2]$ and $[2,\ldots,2,m]$ ($m\geq 3$) follows from a similar argument using \cref{cor.ab}.

\subsection{The definition of $Q_{n,k}(E,\tau)$}

Fix a point $\eta \in \CC$ lying in the upper half-plane. Let $\L=\ZZ+\ZZ\eta$ and define $E=\CC/\L$.
Let $\Theta_n(\L)$ be the space of theta functions defined in \cite[\S2.1]{CKS1}, and let $\theta_0(z), \ldots, \theta_{n-1}(z)$ 
be the basis for $\Theta_n(\L)$  defined in \cite[Prop.~2.6]{CKS1}.
For all $\tau \in \CC-\frac{1}{n}\L$,  we define $Q_{n,k}(E,\tau)$ to be the free algebra $\CC\langle x_0,\ldots,x_{n-1}\rangle$ 
modulo the $n^2$ relations
\begin{equation}
\label{the-relns}
\sum_{r \in \ZZ_n} \frac{\theta_{j-i+r(k-1)}(0)}{\theta_{j-i-r}(-\tau)\theta_{kr}(\tau)}\, x_{j-r}x_{i+r}  \;=\; 0, \qquad (i,j) \in \ZZ_n^2.
\end{equation}

For the rest of this introduction we assume $\tau \notin \tfrac{1}{n}\Lambda$. This ensures that the denominators in \cref{the-relns} are non-zero. 
In \cite[Defn.~3.11]{CKS1}, we defined $Q_{n,k}(E,\tau)$ for all $\tau\in E$.

In \cite{CKS4} it is shown that the $n^2$ relations in \cref{the-relns} span an ${n \choose 2}$-dimensional space when $\tau$ is not a torsion point.

By  \cite[Prop.~3.22]{CKS1}, $Q_{n,k}(E,\tau) \cong Q_{n,k}(E,-\tau)=Q_{n,k}(E,\tau)^{\rm op}$.

Although the relations for $Q_{n.k}(E,\tau)$ seem to have no meaning at first sight, there are two perspectives that make them less mysterious. One involves $R$-matrices and the other involves an identity for theta functions on $g$ variables.

\subsubsection{}
The relations in \cref{the-relns} come from Belavin's elliptic solutions to the quantum Yang-Baxter equation.  Let $V$ be a $\CC$-vector space with basis $e_0,\ldots,e_{n-1}$. For each $z \in \CC$, let $R(z):V \otimes V \to V \otimes V$ be the linear operator
$$
R(z)(e_i \otimes e_j) \; :=\; \sum_{r \in \ZZ_n} \frac{\theta_{j-i+r(k-1)}(-z+\tau)}{\theta_{kr}(\tau)\theta_{j-i-r}(-z)}    \, e_{j-r}\otimes e_{i+r}.
$$
As conjectured by Belavin \cite{bel}, and later proved by Cherednik \cite{cher}, Chudnovsky and  Chudnovsky \cite{ch2}, and Tracy  \cite{tra}, when $k=1$,
 these operators satisfy the equation
$$
R(u)_{12}R(u+v)_{23}R(v)_{12} \;=\; R(v)_{23}R(u+v)_{12}R(u)_{23}.
$$
Clearly,
$$
Q_{n,k}(E,\tau) \;=\; \frac{TV}{({\rm im} \, R(\tau))} \, ,
$$
the right-hand side of which denotes the quotient of the tensor algebra on $V$ by the ideal generated by the image of $R(\tau)$.
In \cite{CKS4}, we  use the fact that $R(z)$ satisfies the quantum  Yang-Baxter equation to show that $Q_{n,k}(E,\tau)$ has the same 
Hilbert series as the polynomial ring on $n$ variables when $\tau$ is not a torsion point.

\subsubsection{}
The second ``explanation'' for the relations involves an $n$-dimensional space $\Theta_{n/k}(\L)$ of theta functions in $g$ variables where 
$g$ is the number in \cref{sect.first.g}. The Heisenberg group
$H_n$ of order $n^3$ acts in a natural way on $\Theta_{n/k}(\L)$ and there is a basis $w_0,\ldots,w_{n-1}$ for $\Theta_{n/k}(\L)$ that
transforms in a nice way with respect to the ``standard'' generators for $H_n$ (see \cite[\S 5.1.1]{CKS2}). 
There is an identity 
	\begin{equation}\label{eq.id.rel.w.intro}
		\sum_{r\in\ZZ_n}  \frac{\theta_{\beta-\alpha+r(k-1)}(0)}{\theta_{\beta-\alpha-r}(-\tau) \theta_{rk}(\tau)} 
		\, w_{\beta-r}(\sfz)\,w_{\alpha+r}(\s(\sfz)) 
		\;=\; 0.
	\end{equation}
in which $\sfz \in\CC^g$ and $\s$ is a certain automorphism of $\CC^{g}$ defined in \cref{sssect.sigma}. 
Compare \cref{eq.id.rel.w.intro} and \cref{the-relns}: if one identifies $x_\a$ with $w_\a$, then  \cref{eq.id.rel.w.intro} tells us that the relations
for $Q_{n,k}(E,\tau)$ vanish on the graph of $\s$. 

\subsubsection{The relations for $Q_{2k+1,k}(E,\tau)$}\label{sssec.geom.description}
This case, which includes the 3-dimensional Sklyanin algebra $Q_{3,1}(E,\tau)$,  is special.
Since $\frac{2k+1}{k}=[3,2,\ldots,2]$ where there are $k-1$ twos,  $X_{n/k} \cong S^kE$. 
The automorphism $\s':S^kE \to S^kE$ is $(\!(z_1,\ldots,z_k)\!) \mapsto (\!(z_1+\tau,\ldots,z_k+\tau)\!)$ (\cref{prop.action.of.sigma.on.S^gE}).
The degree-one component, $V$ say,  of $Q_{2k+1,k}(E,\tau)$ can be viewed as linear forms on $\PP^{n-1}=\PP^{2k}$ so 
$V^{\otimes 2}$ can be viewed as bilinear forms on $\PP^{2k} \times \PP^{2k}$. 

\begin{theorem}
[\cref{prop.Q.2k+1.k}]
If $\tau\in E$ is not a $2$-torsion point, then the quadratic relations for $Q_{2k+1,k}(E,\tau)$  are exactly those elements of $V^{\otimes 2}$ that vanish on the graph of
the automorphism $\s':S^kE \to S^kE$.
\end{theorem}

\subsection{Review of results about $Q_{n,1}(E,\tau)$}
\subsubsection{}
 The algebras $Q_{3,1}(E,\tau)$ first appeared in Artin and Schelter's classification of 3-dimensional regular algebras \cite{AS87}.
There, the algebras $Q_{3,1}(E,\tau)$ belonged to a slightly larger class of algebras $A_{a,b.c}$  parametrized by 
points $(a,b,c) \in \PP^2$ and  defined as $\CC\langle x_0,x_1,x_2\rangle/(r_0,r_1,r_2)$ 
where $r_i=a x_ix_{i+1}+bx_{i+1}x_i+cx_{i+2}^2$ (see \cite[(10.36) and 10.37(i)]{AS87}, and the remark on page 38 of \cite{ATV1} to the effect that the conjecture in \cite[10.37(i)]{AS87} is true). Artin-Tate-Van den Bergh showed that $A_{(a,b,c)}$ is a 3-dimensional regular algebra if and only if $(a,b,c) \in \PP^2-\{\text{12 points}\}$. To do that they introduced the notion of a twisted homogeneous coordinate ring \cite{ATV1} (Odesskii and Feigin discovered  this notion around the same time;  \cite[p.~7]{FO-Kiev} and \cite[p.~208]{FO89})
and showed  there is a surjective homomorphism
$$
\Psi:Q_{3,1}(E,\tau) \; \longrightarrow \; B(E,\tau,\cL_3)
$$
where $\cL_3$ is an invertible $\cO_E$-module of degree 3 and $\ker(\Psi)$ is generated by a degree-three central element, $\Omega$ say.\footnote{When $\tau=0$, the vanishing locus of $\Omega$ is  the curve $abc(x^3+y^3+z^3)=(a^3+b^3+c^3)xyz$, which is non-singular if and only if  $abc \ne 0$ and $(3abc)^3 \ne (a^3+b^3+c^3)^3$,
and $\Psi$ is the familiar map from the polynomial ring on 3 variables to the homogeneous coordinate ring of the image of $E$ under the morphism $E\to \PP(H^0(E,\cL_3)^*)$.}
Artin-Tate-Van den Bergh exploit this, and the fact that $\QGr(B(E,\tau,\cL_3))$ is equivalent to $\Qcoh(E)$, 
to show that $Q_{3,1}(E,\tau) $  has properties \cref{tv.hilb.ser}-\cref{tv.as.reg} in \S\ref{sect.good.props}.
One should think of $B(X,\tau,\cL_3)$ as a homogeneous coordinate ring of $E$, albeit a non-commutative, or twisted,
one. In a similar spirit, one should view $Q_{3,1}(E,\tau)$ as a non-commutative algebra that behaves as if it is the homogeneous 
coordinate ring of a non-commutative analogue of the projective plane $\PP^2$. The element $\Omega$ plays the role of a cubic equation
whose zero locus is $E$. 
\subsubsection{}
Similar results hold for $Q_{4,1}(E,\tau)$. In 1982, Sklyanin used Baxter's elliptic solution to the quantum Yang-Baxter equation 
to define a family of algebras $A_4(E,\tau)$ \cite{Skl82,Skl83}. As mentioned in \cite[p.~20]{ST94}, there is an isomorphism
$$
Q_{4,1}(E,\tau) \; \cong A_4(E,2\tau).
$$
Smith and Stafford \cite[\S 3]{SS92} showed  there is a surjective homomorphism  $A_{4}(E,\tau) \to B(E,\tau,\cL_4)$, 
and hence a surjective homomorphism
$$
\Psi:Q_{4,1}(E,\tau) \; \longrightarrow \; B(E,2\tau,\cL_4),
$$
where $\cL_4$ is an invertible $\cO_{E}$-module of degree $4$, and showed that the kernel of 
$\Psi$ is generated by a regular sequence consisting of two degree-two central elements, $\Omega_1$ 
and $\Omega_2$ say (\cite[Cor.~3.9 and Thm.~5.4]{SS92}). They used this, and the fact that $\QGr(B(E,\tau,\cL_4))$ is equivalent to $\Qcoh(E)$, to show that $Q_{4,1}(E,\tau) $ has properties \cref{tv.hilb.ser}-\cref{tv.as.reg} in \S\ref{sect.good.props}.
 One thinks of $Q_{4,1}(E,\tau)$ as if it is the homogeneous coordinate ring of a non-commutative analogue $\Projnc(Q_{4,1}(E,\tau))$ of the projective space $\PP^3$ and of $\Omega_1$ and $\Omega_2$ as if they are 
the defining equations of $E$ presented as the intersection of two ``non-commutative quadrics'' in $\Projnc(Q_{4,1}(E,\tau))$.
This theme is elaborated on in \cite{SVdB-NCQ}.

\subsubsection{}
As stated immediately after the proof of Theorem 3.1 in \cite{FO-Kiev}, for all $n \ge 3$ there is a surjective homomorphism 
$$
\Psi:Q_{n,1}(E,\tau) \;  \longrightarrow \;  B(E,(n-2)\tau,\cL_n)
$$
 with $\cL_n$ an invertible $\cO_{E}$-module of degree $n$; see \cite[\S4.1]{TvdB96}.  
 All degree-$n$ invertible $\cO_{E}$-modules are pullbacks of each other along suitable translation automorphisms so the isomorphism class of $B(E,(n-2)\tau,\cL_n)$ does not depend on the choice of $\cL_n$.
 When $n \ge 5$ it is difficult to use the surjectivity of $\Psi$ to obtain information about $Q_{n,1}(E,\tau)$ 
 because  $\ker(\Psi)$ is no longer generated by a 
regular sequence of central elements; this is analogous to the fact that the image of $E$ under the morphism 
$\Phi_{|\cL_n|}:E \subseteq \PP(H^0(E,\cL_n)^*)$ is a complete intersection if and only if $n=3,4$.\footnote{This is well-known. The case $n=3$ is 
trivial. When $n=4$, $E$ is an intersection of two quadrics (see, e.g., \cite[Exer.~IV.3.6]{Hart} or \cite[Ch.~III]{Hulek86}.
Since the degree-$n$ elliptic normal curve $E \subseteq \PP^{n-1}$ is not contained in any hyperplane, if it is a complete intersection, 
it would be a complete intersection of $n-2$ hypersurfaces of degree $\ge 2$ so would have degree $\ge 2^{n-2}$; however, if $n >4$, then $n<2^{n-2}$.} 

\subsection{The organization of this paper}

\Cref{sect.prelims} concerns twisted homogeneous coordinate rings. It records important results due to Artin-Van den Bergh and Keeler, and a few results that are not in the literature (but should be). Some of those are surely known to others.
\cref{cor.sigma-ample,cor.sigma-ample.2}, which appear to be new, give a criterion for $\s$-ampleness that is 
particularly useful for the types of twisted homogeneous coordinate rings that appear in the study of $Q_{n,k}(E,\tau)$. 

\Cref{sect.negative.contd.fracs} records some results and notation from our earlier papers about $Q_{n,k}(E,\tau)$ that are used in this paper.  We discuss maps from $Q_{n,k}(E,\tau)$ to twisted homogeneous coordinate rings in   \cref{ssect.Q.B}. The main results there 
are  \cref{thm.thcr.qnk,cor.map.to.B}. We also want to emphasize the  isomorphism 
$B(X_{n/k},\sigma',\cL'_{n/k}) \cong B(X_{n/k},(\sigma')^{-1},\cL'_{n/k})$ in \cref{th.-tau}. 
This, and the anti-isomorphism $B(X,\s^{-1},\cL)^{\mathrm{op}} \cong B(X,\s,\cL)$  in  \Cref{prop.B-op}, allow one to 
reconcile some sign differences that arise when one compares various papers.
These (anti-)isomorphisms, and the homomorphisms
in  \cref{ssect.Q.B}, are compatible with the observation in \cite[Prop.~3.22]{CKS1} that $Q_{n,k}(E,\tau) \cong Q_{n,k}(E,-\tau)=Q_{n,k}(E,\tau)^{\rm op}$.

Our questions about the degrees of minimal sets of generators and relations for $B(X_{n/k},\sigma',\cL'_{n/k})$ often reduce
to this: if $\cF$ and $\cG$ are locally free\footnote{All our locally free sheaves are coherent.} $\cO_X$-modules,  when is the natural map
\begin{equation}
\label{mult.map}
H^0(X,\cF) \otimes H^0(X,\cG) \to H^0(X,\cF \otimes \cG)
\end{equation}
surjective? This question is of broad interest in algebraic geometry and has been studied a great deal. 
We prove several new results of this form in the later sections of the paper. 
Most of those results are for varieties $X$ for which there is a surjective morphism $\pi:X \to E$. The proofs often reduce to the 
question of whether $H^0(E,\pi_*\cF) \otimes H^0(E,\pi_*\cG) \to H^0(E,\pi_*\cF \otimes \pi_*\cG)$ is surjective.
It usually turns out in the cases of interest to us that $\pi_*\cF$ and $\pi_*\cG$ are semistable locally free $\cO_E$-modules.
For this reason, \cref{sect.semistable} collects a number of standard results about semistable $\cO_E$-modules.
We also prove the following result that we found particularly useful. 

\begin{theorem}[\cref{le.otimes}]
Let $\cU$ and $\cV$ be semistable locally free coherent $\cO_E$-modules of slopes $\mu(\cU)$ and $\mu(\cV)$. If 
$\cU$ and $\cV$ are generated by their global sections and 
$$
\frac{1}{\mu(\cU)} \, + \, \frac{1}{\mu(\cV)} \; < \; 1,
$$
then the multiplication map $H^0(E,\cU) \otimes H^0(E,\cV) \to H^0(E,\cU \otimes \cV)$ is surjective.
\end{theorem}

Although our ultimate interest is the specific twisted homogeneous coordinate rings $B(X_{n/k},\sigma',\cL'_{n/k})$
we often prove results in greater generality. For example, \cref{le.aux} provides a result about the surjectivity of the map in \cref{mult.map} 
when $X$ is a projective space bundle over $E$. In  \cref{se.rel} we show that $B(S^gE,\sigma,\cL)$ is generated in
 degree one and has relations of degrees 2 and 3 for all $g \ge 2$, all translation automorphisms $\s:S^gE\to S^gE$, and all invertible $\cL$ whose 
class in the N\'eron-Severi group is $aD+bF$ with $a \ge 1$ and $b \ge 2$.\footnote{Given the basis 
$\{D,F\}$ for $\NS(S^gE)$ in \cref{sssect.NS.S^gE}, 
if $[\cL]=aD+bF$ with $a \ge 1$ and $b \ge 2$, then $\cL$ is ample and generated by its global sections.}
Likewise, \cref{se.rel.ep} shows that the relations for $B(E^g,\s,\cL_{n/k})$ are generated in degree $\le 3$ for all translation automorphisms $\s:E^g \to E^g$ when $\cL_{n/k}$ is very ample.

\subsection{Acknowledgements}

A.C. was partially supported by NSF grants DMS-1565226, DMS-1801011, and DMS-2001128.

R.K. was a JSPS Overseas Research Fellow, and supported by JSPS KAKENHI Grant Numbers JP16H06337, JP17K14164, and JP20K14288, Leading Initiative for Excellent Young Researchers, MEXT, Japan, and Osaka City University Advanced Mathematical Institute (MEXT Joint Usage/Research Center on Mathematics and Theoretical Physics JPMXP0619217849). R.K. would like to express his deep gratitude to Paul Smith for his hospitality as a host researcher during R.K.'s visit to the University of Washington.

S.P.S. thanks S\'andor Kov\'acs, Jack Lee, 
Max Lieblich, and  Bianca Viray for many helpful conversations.

\section{Twisted homogeneous coordinate rings}
\label{sect.prelims}

In this section we mostly work over an algebraically closed field $\Bbbk$. Always,  $E$ denotes an elliptic curve defined over $\Bbbk$ 
and $\PP^n$ denotes the projective space $\PP_\Bbbk^n$.

We always assume $\Bbbk=\CC$ when we discuss  $Q_{n,k}(E,\tau)$  because its definition  involves theta functions.


\subsection{Motivation: projective normality and defining relations for abelian varieties}
\label{ssect.proj.norm}
Nothing in this section is used later in the paper. Its purpose is to explain how the results about $B(E^g,\s,\cL_{n/k})$ (when all the $n_i$'s 
in the continued fraction for $\frac{n}{k}$ are $\ge 3$) in parts \cref{main.thm.surj} and \cref{main.thm.rel.deg} of \cref{thm.main.01} fit
into the theme of defining relations for abelian varieties: when $\s=\id$, $B(E^g,\s,\cL_{n/k})$  is the section ring 
$S(E^g,\cL_{n/k})$ so those results say that $\cL_{n/k}$ is normally generated and the image of $E^g$ under the embedding 
$\Phi_{|\cL_{n/k}|}:E^g \to \PP^{n-1}$ is a scheme-theoretic intersection of quadrics and cubics.

Let $X$ be a projective  algebraic variety, $\cL$ a very ample invertible $\cO_X$-module, and 
$\Phi_{|\cL|}:X \to \PP^r = \PP(H^0(X,\cL)^*)$ the associated embedding. We identify $X$ with its image in $\PP^r$,
and denote by $I_X$ the largest graded ideal in $\Bbbk[x_0,\ldots,x_r]$ vanishing on $X$, and write $S(X)$ for the 
homogeneous coordinate ring $\Bbbk[x_0,\ldots,x_r]/I_X$.
The following statements are equivalent:
\begin{enumerate}
\item 
the restriction map
$$
\Bbbk[x_0,\ldots,x_r] \, =\, \bigoplus_{k=0}^\infty  H^0(\PP^r,\cO_{\PP^r}(k)) \, \longrightarrow \, S(X,\cL) 
\, := \, \bigoplus_{k=0}^\infty   H^0(X,\cL^{\otimes k})
$$ 
is surjective;  
\item 
$S(X)$ is integrally closed; 
\item 
the map $S(X) \to  S(X,\cL)$ is an isomorphism;
\item
$S(X,\cL)$ is generated by its degree-one component.  
\end{enumerate}
If one, hence all, of these conditions holds we say that $\cL$ is {\sf normally generated} and that the subvariety $X \subseteq \PP^r$
is {\sf projectively normal}. 
A fundamental problem in algebraic geometry is to decide when this happens and, when it does, to determine  the 
degrees of a minimal set of relations for $X \subseteq \PP^r$. 

Let $X$ be a complex abelian variety. The theorem in the introduction to \cite{PP04} provides a short history of what is known about $S(X,\cL)$. Those results have the following flavor: if $\cL$ is a sufficiently high power of an ample invertible sheaf, perhaps with an additional 
hypothesis about its base locus or global generation, then $S(X,\cL)$ is generated by its degree-one component and the kernel
of the map $S(X) \to  S(X,\cL)$ is generated by elements of degree two, and perhaps degree three. 
Most of those results are subsumed by \cite[Thm.~6.1]{PP04}: if $\cM$ is an ample invertible sheaf on an abelian variety $X$, 
then $\cM^{\otimes 3}$ is very ample (Lefschetz's Theorem \cite[Thm.~2.11]{Kempf-book}), and 
$\cM^{\otimes 3}$ is normally generated (Koizumi's Theorem \cite[Cor.~4.7]{Koi76}), and the kernel of the map $S(X) \to
S(X,\cM^{\otimes 3})$ is generated by elements of degrees 2 and 3 (Mumford's Theorem \cite[Thm.~(7), p.~168]{PP04}).

The twisted homogeneous coordinate rings $B(X,\s,\cL)$ defined in \cref{sect.twhcr} below are non-commutative analogues of
$S(X,\cL)$ ($B(X,\id_X,\cL) = S(X,\cL)$) and the same questions about $B(X,\s,\cL)$ are of interest: is it generated in degree one
and what are the minimal degrees of a generating set of relations for it. 
For example, the question of whether the map $\Psi_{n/k}$ in \cref{thm.main.01} is surjective is equivalent to the question of 
whether its codomain $B(X_{n/k},\s',\cL'_{n/k})$ is generated by its degree-one component. 
None of the results referred to in the previous paragraph shows that $S(E^g,\cL_{n/k})$ is generated by its degree-one component so
one cannot expect to prove that $B(E^g,\s,\cL_{n/k})$ is generated by its degree-one component by tweaking the 
commutative arguments.  We therefore 
develop some new methods that yield fairly complete results about $B(X_{n/k},\s',\cL_{n/k}')$ when $X_{n/k}$ is 
$E^g$ and $S^gE$.

\subsection{Notation}
\label{sect.notation}
We adopt the notation laid out at \cite[p.~23]{ST94}. For convenience we recall it.

Let $X$ be a scheme over a field $\Bbbk$ and let  $f:\cF \to \cG$ be a homomorphism of $\cO_X$-modules.
Let $\nu$ be a $\Bbbk$-automorphism of $X$.

If $p \in X$ we write $p^\nu$ for $\nu(p)$. We extend this to Weil divisors in the obvious way. For example, if  $D=\sum n_p (p)$
is a divisor on a curve, then $D^\nu := \sum n_p(p^\nu)$. 

 We write $\cF^\nu$ for $\nu^*\cF=(\nu^{-1})_*\cF$. 
Thus if $D$ is a Weil divisor, $\cO_X(D)^\nu=\cO_X(D^{\nu^{-1}})$. We write $f^\nu$ for
$\nu^*(f):\cF^\nu \to \cG^\nu$. There is a $\Bbbk$-linear isomorphism 
$$
H^0(X,\cF) \, \longrightarrow \, H^0(X,\cF^\nu) \;=\; H^0(X,\cO_X \otimes_{\nu^{-1}\cO_X} \nu^{-1}\cF)
$$
given by $s \mapsto s^\nu :=1 \otimes s$. Notice that $s(p^\nu)=0$ if and only if $s^\nu(p)=0$. Notice too that the natural isomorphism 
$\Hom_X(\cO_X,\cF) \stackrel{\sim}{\longrightarrow} H^0(X,\cF)$, $f \mapsto f(1)$, satisfies $f^\nu \mapsto f(1)^\nu$.

There is a canonical map 
$H^0(X,\cF) \otimes H^0(X,\cG) \to H^0(X,\cF \otimes \cG)$ which we call {\sf multiplication}. 
If $s \in H^0(X,\cF)$ and $t \in H^0(X,\cG)$ we write $s*t$ for the image of $s \otimes t$ under the multiplication map.
If $\nu$ is a $\Bbbk$-linear automorphism of $X$, then $(s *t)^\nu=s^\nu * t^\nu$ because $\nu^*$ distributes across tensor products.

If  $X$ is an abelian variety and $x \in X$, we write $T_x:X \to X$ for the map $T_x(y)=x+y$. We call $T_x$ a {\sf 
translation automorphism}. If $\cF$ is an $\cO_X$-module we call  $T_x^*\cF$  a {\sf translate of $\cF$}. 

Let $X$ and $Y$ be projective $\Bbbk$-varieties. 
The N\'eron-Severi group of $X$ is the group  $\NS(X):={\rm Pic}(X)/{\rm Pic}^0(X)$. 
This is a finitely generated abelian group.  If $\s:X \to Y$ is a morphism, the inverse image functor induces a group homomorphism 
$\s^*:{\rm Pic}(Y) \to {\rm Pic}(X)$ that descends to a homomorphism $\NS(Y) \to \NS(X)$, and hence to $\NS(Y)_{\bbC} \to \NS(X)_{\bbC}$. The last homomorphism can be represented by a 
matrix with entries in $\ZZ$.  If $\s \in \Aut(X)$, we call $\s^*$ {\sf quasi-unipotent} if all its eigenvalues in $\CC$ are roots of unity.

\subsection{The twisted homogeneous coordinate rings $B(X,\s,\cL)$}
\label{sect.twhcr}
The rings $B(X,\s,\cL)$ we are about to define were introduced by Artin-Tate-Van den Bergh  in \cite {ATV1}, and independently
by Feigin-Odesskii in \cite{FO-Kiev} and \cite{FO89}, as a device to understand  graded algebras that map to them.
 That understanding is obtained through \cref{thm.AV} and \cref{cor.closed.subspace} below.

\begin{proposition}\label{pr.triples} 
  Let $\Bbbk$ be a field.  There is a contravariant functor $(X,\s,\cL) \rightsquigarrow B(X,\s,\cL)$ from the category of triples consisting of a $\Bbbk$-scheme $X$, a $\Bbbk$-automorphism $\s:X \to X$, and an invertible $\cO_X$-module $\cL$, to the category of graded $\Bbbk$-algebras.
\end{proposition} 

This needs some explanation.

A morphism of triples is a pair $(f,u):(X,\s,\cL) \to (X',\s',\cL')$ consisting of a $\Bbbk$-morphism $f:X \to X'$
 such that $f\s=\s'f$ and a homomorphism $u:f^*\cL' \to \cL$. 
 
As a graded vector space the $\Bbbk$-algebra $B(X,\s,\cL)$ is
$$
\bigoplus_{i=0}^\infty H^0(X,\cL_i)
$$
where 
$$
\cL_i \; :=\;  \cL \otimes \cL^\sigma  \otimes\cdots\otimes\cL^{ \sigma^{i-1}}.
$$
The product $x\cdot y$ of $x \in H^0(X,\cL_i)$ and $y \in H^0(X,\cL_j)$ is $x*y^{\s^i}$, i.e.,  the image of 
$x \otimes y^{\s^i}$ under the natural multiplication map
$$
H^0(X,\cL_i)  \otimes H^0(X, \cL_j^{\sigma^i})  \, \longrightarrow \,  H^0(X,\cL_i \otimes \cL_j^{\sigma^i}).
$$
We call $B(X,\s,\cL)$ a {\sf twisted homogeneous coordinate ring}. The terminology  is motivated and justified by \cref{thm.AV} below. 

\subsubsection{Isomorphisms and anti-isomorphisms}

The next two results are probably known to the experts.

\begin{proposition}
\label{prop.B-op}
$B(X,\s^{-1},\cL)^{\mathrm{op}} \cong B(X,\s,\cL)$ where $(\,\cdot\,)^{\mathrm{op}}$ denotes the opposite ring.
\end{proposition}
\begin{proof}
  We write $B':=B(X,\s^{-1},\cL)$ and $B:=B(X,\s,\cL)$ and denote the multiplication maps by $\mu:B \otimes B \to B$ and $\mu':B' \otimes B' \to B'$.  We will prove the result by defining a degree-preserving $\Bbbk$-linear isomorphism $\varphi:B \to B'$ with the property that $\varphi\circ \mu (x\otimes y) = \mu' (\varphi(y) \otimes \varphi(x))$ for all homogeneous $x$ and $y$ in $B$.

  Let $\cL_n'=\cL \otimes \cL^{\s^{-1}} \otimes \ldots \otimes \cL^{\s^{-n+1}}$.  The degree-$n$ component of $B'$ is $B_n':=H^0(X,\cL_n')$.

  Since $\cL_n'=(\cL_n)^{\s^{-n+1}}$, there is a $\Bbbk$-linear isomorphism $\varphi_n:H^0(X,\cL_n) \to H^0(X,\cL_n')$ given by $\varphi(x)=x^{\s^{-n+1}}$. If $x \in H^0(X,\cL_m)$ and $y \in H^0(X,\cL_n)$, then
$$
\varphi\circ \mu (x\otimes y) \;=\; \left(x*y^{\s^m} \right)^{\s^{-m-n+1}}.
$$
 On the other hand, 
 $$
\mu' (\varphi(y)\otimes \varphi(x)) \;=\;    y^{\s^{-n+1}} *\big(x^{\s^{-m+1}}\big)^{\s^{-n}}.
$$
The result now follows from the fact that inverse image commutes with tensor product (see \S\ref{sect.notation}).
\end{proof}

\begin{proposition}\label{pr.tw-by-mu}
  Let $(X,\s,\cL)$ be a triple. If $\mu:X \to X$ is a $\Bbbk$-automorphism, there is an isomorphism
  \begin{equation*}
    B(X,\sigma,\cL) \, \stackrel{\sim}{\longrightarrow} \, B(X,\mu^{-1}\sigma\mu,\mu^*\cL)
  \end{equation*}
sending $s\in H^0(X,\cL) = B(X,\sigma,\cL)_1$ to $\mu^*(s)=s^\mu \in H^0(X,\mu^*\cL)$. 
\end{proposition}
\begin{proof}
This is an immediate consequence of \Cref{pr.triples} because 
$(\mu,\id):   (X,\mu^{-1}\sigma\mu,\mu^*\cL) \to (X,\sigma,\cL)$ is an isomorphism of triples.
\end{proof}

\subsubsection{$\s$-ampleness}\label{ssse.sample} 
When $X$ is noetherian we say $\cL$ is {\sf $\s$-ample} if for every coherent $\cO_X$-module $\cF$,
$$
H^q(X,\cF  \otimes  \cL_i) \; = \; 0
$$ 
for all $q\ge 1$ and all $i \gg 0$.  When $\s$ is the identity this becomes the traditional definition of ampleness.

\subsubsection{The Artin-Van den Bergh Theorem and  the functor $\Gamma_*:\Qcoh(X) \to \Gr(B(X,\s,\cL))$}
\label{AVdB.thm}
\leavevmode

Following  \cite{AV90} and \cite{AZ94},  we define the auto-equivalence  $s:\Qcoh(X) \to \Qcoh(X)$ 
by the formula  
$$
s \; :=\; \cL \otimes \s^*(\,\cdot\,), \qquad i.e., \;\; s(\cM)=\cL \otimes \s^*\cM = \cL \otimes \cM^\s.
$$
We write $s^0$ for the identity functor on $\Qcoh(X)$. Now define the graded vector space
$$
M \;:=\; \bigoplus_{n \in \ZZ} H^0(X,s^n(\cM)) \;=\; \bigoplus_{n \in \ZZ}  \Hom_{\cO_X}(\cO_X,s^n(\cM))  \;=\; \bigoplus_{n \in \ZZ}  H^0(X,\cL_n \otimes \cM^{\s^n}).
$$
Let $b\in B(X,\s,\cL)_i$ and $m \in M_j=H^0(X,\cL_j \otimes \cM^{\s^j})$.
Since $m$ is a homomorphism
$\cO_X \to 
\cL_j \otimes \cM^{\s^j}$,
$$
s^i(m): s^i(\cO_X)\,=\, \cL_i \, \longrightarrow \, s^i(\cL_j \otimes \cM^{\s^j}) \,=\,  \cL_i \otimes (\cL_j \otimes \cM^{\s^j})^{\s^i} \,=\, 
\cL_{i+j} \otimes \cM^{\s^{i+j}}.
$$
Since $b$ is a homomorphism $\cO_X \to \cL_i$ we may define  
\begin{equation}
\label{eq.B-mod}
b \cdot m \; :=\; s^i(m) \circ b.
\end{equation} 
This formula gives $M$ the structure of a graded left $B(X,\s,\cL)$-module. We define $\Gamma_*$ by
$$
\Gamma_*\cM \; := \;  \bigoplus_{n \in \ZZ}  H^0(X,\cL_n \otimes \cM^{\s^n})
$$ 
with this graded module structure. Sometimes we abuse notation and  write $\G_*$ for the composition
$$
\xymatrix{
\Qcoh(X) \ar[rr]^-{\G_*} && \Gr\big(B(X,\s,\cL)\big) \ar[rr] && \QGr\big(B(X,\s,\cL)\big).
}
$$

\begin{theorem}
[Artin-Van den Bergh]
\cite[Thms.~1.3 and 1.4]{AV90}
\label{thm.AV}
If $X$ is a projective $\Bbbk$-scheme and $\cL$ is $\s$-ample, then $B(X,\s,\cL)$ is a finitely generated left noetherian $\Bbbk$-algebra and the functor $\G_*$ provides an equivalence of categories
\begin{equation}
\label{equiv-cats}
 \Qcoh(X) \; \equiv \; \QGr\big(B(X,\s,\cL)\big).
\end{equation}
\end{theorem}

By \cite[Cor.~5.1]{Keeler-Criteria}, $\cL$ is $\s$-ample if and only if it is $\s^{-1}$-ample. Thus, by \cref{prop.B-op}, in the context of 
\cref{thm.AV}, $B(X,\sigma,\cL)$ is right noetherian too \cite[Cor.~5.3]{Keeler-Criteria}.

\begin{theorem}
[Keeler]
\cite[Thms.~1.2 and 1.4]{Keeler-Criteria}
\label{thm.sigma.ample}
Let $\s$ be an automorphism of  a projective scheme $X$ over an algebraically closed field $\Bbbk$. 
The following conditions are equivalent:
\begin{enumerate}
\item 
there is a $\s$-ample invertible $\cO_{X}$-module;
\item 
every ample invertible $\cO_{X}$-module is $\s$-ample;
\item 
the action of $\s^*$ on $\NS(X)_\CC$ is quasi-unipotent. 
\end{enumerate}
Furthermore, if one of those conditions holds, then
\begin{enumerate}\setcounter{enumi}{3}
\item
the GK-dimension of $B(X,\s,\cL)$ is an integer for every ample $\cL$ and
\item
$B(X,\s,\cL)$ is right and left  noetherian for every ample $\cL$.
\end{enumerate}
\end{theorem}

The next result applies to $X=E^g$ and all translation automorphisms $\s:E^g \to E^g$.

\begin{corollary}
\label{cor.sigma-ample}
Let $X$ be a projective scheme over an algebraically closed field $\Bbbk$ and $G$ an algebraic group over $\Bbbk$ that acts on $X$. 
If $\s \in G$, then every ample invertible $\cO_{X}$-module is $\s$-ample.
\end{corollary}
\begin{proof}
By \Cref{thm.sigma.ample} it suffices to show that $\sigma^*$ is quasi-unipotent. 

The Picard functor is representable by a scheme $\pic(X)$ \cite[II.15]{mur} which is acted upon by $G$. Since $G$ has finitely many connected components (as all algebraic groups do) some power of $\s$, say $\s^r$, belongs to the connected component $G^0$ 
that contains the identity. The action of $G^0$ sends each connected component of $\pic(X)$ to itself and hence acts trivially on
$\NS(X) = \pic(X)/\pic^0(X)$. In particular, $(\s^*)^r$ acts trivially on $\NS(X)_\CC$ so the action of $\s^*$ on $\NS(X)_\CC$ is 
quasi-unipotent.
\end{proof}

\begin{corollary}
\label{cor.sigma-ample.2}
Let  $\Sigma$ be a finite group acting as group automorphisms of an abelian variety $A$ over an algebraically closed field $\Bbbk$.
If $\s:A \to A$ is translation by a point that is fixed by $\Sigma$, then $\s$ descends to an 
automorphism $\s'$ of $A/\Sigma$ having the property that every ample invertible module over $A/\Sigma$ is $\s'$-ample.  
\end{corollary}
\begin{proof}
The set $A^\Sigma$ consisting of points fixed by $\Sigma$ is an algebraic subgroup of $A$. It acts on $A$ by translation automorphisms
and each such automorphism descends to an automorphism of $A/\Sigma$. Since $A/\Sigma$ is a projective variety
the result follows from \cref{cor.sigma-ample} with $G=A^\Sigma$. 
\end{proof}

\subsection{Using the rings $B(X,\s,\cL)$}

If $J$ is a graded ideal in a finitely generated $\NN$-graded algebra $A$ over a field $\Bbbk$, 
 the three natural functors between the categories $\Gr(A)$ and $\Gr(A/J)$ induce functors
$$
\xymatrix{
\QGr(A/J) \ar@/_1pc/[rr]_{i_*} && \ar@/_1pc/[ll]_{i^*, \, i^!} \QGr(A)
}
$$
between the quotient categories such that $i_*$ is a fully faithful embedding whose essential image is closed under subobjects and quotients, 
$i^*$ is left adjoint to $i_*$, and $i^!$ is right adjoint to $i_*$ (see \cite{VdB-blowup} and \cite{SPS15-corrigendum}). 
The functors $i^*$ and $i_*$ behave like the inverse and direct image functors associated to a closed immersion of one scheme in another. Thus, the next result says, in effect, that the non-commutative 
scheme with homogeneous coordinate ring $A$ has a closed subscheme isomorphic to $X$. In this paper, we will show this happens
when $A$ is $Q_{n,k}(E,\tau)$ and $X$ is its characteristic variety provided that characteristic variety is a product or symmetric product of copies of $E$.

\begin{corollary}
\label{cor.closed.subspace}
Let $A$ be an $\NN$-graded $\Bbbk$-algebra. 
Assume the hypotheses in \cref{thm.AV} hold.
If there is a surjective homomorphism $A \to B(X,\s,\cL)$, then there are functors 
$$
\xymatrix{
\Qcoh(X) \ar@/_1pc/[rr]_{i_*} && \ar@/_1pc/[ll]_{i^*, \, i^!} \QGr(A)
}
$$
 in which $i_*$ is a fully faithful functor whose essential image is closed under subobjects and quotients, 
$i^*$ is left adjoint to $i_*$, and $i^*$ is right adjoint to $i_*$.
\end{corollary}

In \cite[\S3.17, Prop.~3.20]{ATV1}, Artin-Tate-Van den Bergh describe a procedure that associates to a 
fairly general graded $\Bbbk$-algebra $A$ a canonical algebra $B$ and a canonical homomorphism of 
graded algebras $A \to B$. 
In general, $B$ might not be of the form $B(X,\s,\cL)$.
But in a number of important situations it is. 

The next result, which uses ideas in \cite{ATV1}  and \cite[p.~8]{FO-Kiev},  will be applied to $A=Q_{n,k}(E,\tau)$. In it  
we view elements of $V^{\otimes 2}$ as forms of bi-degree $(1,1)$ on the product of projective spaces $\PP(V^*) \times \PP(V^*)$.

\begin{proposition}
\label{prop.map.to.B}
Let $TV$ denote the tensor algebra on a finite dimensional $\Bbbk$-vector space $V$ and let $A=TV/(R)$ be the quotient by the ideal generated by a subspace $R \subseteq V^{\otimes 2}$. 
Let $X$ be a $\Bbbk$-scheme, $\s \in \Aut_\Bbbk(X)$, $\G_\s$ the graph of $\s$, $f:X \to \PP(V^*)$ a morphism, and let $\cL=f^* \cO_{\PP(V^*)}(1)$. 
If $R$ vanishes on $(f \times f)(\G_\s)$, 
then the canonical linear map $V \to H^0(X,\cL)$ extends to a $\Bbbk$-algebra homomorphism $\varphi:A \to B(X,\s,\cL)$.
\end{proposition}
\begin{pf} 
Since $V=H^0(\PP(V^*),\cO_{\PP(V^*)}(1))$, there is a canonical linear map $V \to H^0(X,\cL)=B(X,\s,\cL)_1$. 
This map extends in a unique
way to a homomorphism $\varphi:TV \to B(X,\s,\cL)$.  
An element $w \in V \otimes V=H^0(X\times X,\cL \boxtimes \cL)$ is in the kernel of $\varphi$  
if and only if it vanishes on $(f \times f)(\G_\s)$. Since elements of $R$ vanish on this graph by hypothesis, 
$R \subseteq \ker(\varphi)$. The result follows.
\end{pf}

\begin{proposition}
\label{prop.functoriality}
Let $Y$ be a projective $\Bbbk$-scheme, $\s$ a $\Bbbk$-automorphism  of $Y$, and $\cL$ a base-point free 
invertible $\cO_Y$-module. Let $\Phi:Y \to \PP(H^0(Y,\cL)^*)$ be the morphism associated to the complete linear system 
$|\cL|$ and let $X=\Phi(Y)$.
There is a factorization $\Phi=i \circ f$ where $i:X \to \PP(H^0(Y,\cL)^*)$ is the inclusion and $f:Y \to X$
is obtained by restricting the codomain of $\Phi$. 
Let $\cL'=\cO(1)|_X$.  If $\s':X \to X$ is an automorphism  such that $f\s=\s' f$, 
then the canonical map $H^0(X,\cL') \to H^0(Y,\cL) $ extends to a homomorphism of graded rings
$$
\varphi:B(X,\s',\cL') \; \longrightarrow \;  B(Y,\s,\cL) 
$$
\end{proposition}
\begin{proof}
  Since $\cL$ is generated by its global sections, $\cL\cong \Phi^*\cO(1)$.  Since $\Phi^*\cO(1)=f^*i^* \cO(1)=f^*\cL'$, there is an isomorphism $u:f^*\cL' \to \cL$.  We therefore obtain a morphism of triples $(f,u):(Y,\s,\cL) \to (X,\s',\cL')$ and hence, by functoriality of the $B$-construction, a homomorphism $\varphi$ as claimed.
\end{proof}

In \cref{cor.map.to.B}, we apply \cref{prop.map.to.B,prop.functoriality} to obtain homomorphisms
$$
Q_{n,k}(E,\tau) \, \stackrel{\Psi}{\longrightarrow} \,  B(X_{n/k},\s',\cL'_{n/k}) \, \longrightarrow \,  B(E^g,\s,\cL_{n/k}).
$$
The following questions then become 
relevant:
\begin{itemize}
\item
is $\cL'_{n/k}$ a $\s'$-ample sheaf ?
  \item 
is $B(X_{n/k},\s',\cL'_{n/k})$ generated in degree one; i.e., is $\Psi$ surjective ?
  \item 
are the relations for $B(X_{n/k},\s',\cL'_{n/k})$ generated in degrees 2 and 3 ?
\end{itemize}
We answer these in the affirmative when $X_{n/k}$ is $E^g$ and $S^gE$.

\subsection{Generators and relations for $B(X,\s,\cL)$}
\label{sect.ker.mult}

We assume that $\Bbbk$ is an algebraically closed field, $X$ is a projective $\Bbbk$-scheme, and $\sigma:X\to X$ is a $\Bbbk$-automorphism.

In this section we use the notation
\begin{align*}
\cL &  \; :=\; \text{an ample invertible $\cO_X$-module generated by its global sections},
\\
\cM  \; = \; \cM_m & \; :=\;   \sigma^*\cL\otimes\cdots\otimes (\sigma^*)^m \cL,
\\
\cN   \; = \; \cN_m  & \; :=\;  (\sigma^{m+1})^*\cL,
\\
\cK \;=\; \cK_m  & \; :=\;  \ker\big(   H^0(X,\cN_m) \otimes \cO_X \twoheadrightarrow   \cN_m\big),
\\
\cG \;=\; \cG_m & \;:=\;  \cM_m \otimes \cK_m\; =\;  \ker\big(\cM_m \otimes   H^0(X,\cN_m)   \twoheadrightarrow  \cM_m \otimes \cN_m\big),
\\
R(\cM,\cN) &\; := \;  \ker \big(H^0(X,\cM)\otimes H^0(X,\cN)\to H^0(X,\cM\otimes \cN)\big). 
\end{align*}
The notation $R(\cM,\cN)$ is taken from \cite{mum-q,SS92}.

There is a natural map
 \begin{equation}
\label{eq.H0L.RMN}
H^0(X,\cL) \otimes R(\cM_m,\cN_m) \to R(\cL \otimes \cM_m,\cN_m) 
\end{equation}
that fits into the following commutative diagram with exact rows:
\begin{equation}
\label{diag1}
\xymatrix{
0 \longrightarrow  H^0(\cL) \otimes R(\cM_m,\cN_m) \ar[d] \ar[r] & H^0(\cL) \otimes  H^0(\cM_m) \otimes  H^0(\cN_m)  \ar[r] \ar[d]& H^0(\cL) \otimes  H^0(\cM_m \otimes  \cN_m) \ar[d]
\\
0 \longrightarrow   R(\cL \otimes \cM_m,\cN_m)  \ar[r] & H^0(\cL \otimes  \cM_m) \otimes  H^0(\cN_m)  \ar[r] & H^0(\cL \otimes  \cM_m \otimes  \cN_m). 
 }
\end{equation}

The next result is a small extension of \cite[Lem.~3.7]{SS92}.

\begin{lemma}
\label{lem.SS92}
Let $X$ be a connected projective $\Bbbk$-scheme; i.e., $H^0(X,\cO_X)=\Bbbk$, $\sigma\colon X\to X$ a $\Bbbk$-automorphism, and  
$\cL$ an ample invertible $\cO_X$-module generated by its global sections.
Suppose $B=B(X,\s,\cL)$ is generated as a $\Bbbk$-algebra by $B_1$. Write $B=T(B_1)/J$ where $T(B_1)$ is the tensor algebra on $B_1$. Let $J_i=J \cap B_1^{\otimes i}$. 
The ideal $J$ is generated by $J_2+\cdots+J_\ell$ if and only if the map in (\ref{eq.H0L.RMN})
is surjective for all $m \ge \ell-1$.
\end{lemma}
\begin{pf}
The degree-$r$ component of $T=T(B_{1})$ is  $T_r=B_1^{\otimes r}$. Clearly, $J$ is generated by $J_2+\cdots+J_\ell$ if and only if 
$J_{r+1}=T_1J_r+J_rT_1$ for all $r \ge \ell$, i.e., if and only if 
$$
\frac{J_{r+1}}{J_rT_1} \; = \; \frac{T_1J_r+J_rT_1}{J_rT_1}
$$
for all $r \ge \ell$. 

We will now reformulate this, but first, to be consistent with the definitions of $\cL$, $\cM$, $\cN$, we set $m=r-1$ so that 
$J_{r+1}=J_{m+2}$.

There is a commutative diagram
$$
\xymatrix{
0 \ar[r] & T_1J_mT_1 \ar[r]   \ar[d]_\a & T_1J_{m+1} \ar[r]^>>>>>>\d  \ar[d]_\b  & T_1J_{m+1}/ T_1J_mT_1 \ar[r]  \ar[d]_\c  & 0
\\
0 \ar[r]    & J_{m+1}T_1 \ar[r] & J_{m+2} \ar[r]_>>>>>>\ve &  J_{m+2}/ J_{m+1}T_1 \ar[r] & 0
}
$$  
in which $\a$ and $\b$ are the natural inclusions, the rows are exact, and $\c$ is the unique linear map such that $\ve\b=\c\d$.
Clearly 
$$
{\rm im}(\c) \; = \; {\rm im}(\c\d) \;=\; {\rm im}(\ve\b) \;=\; \frac{T_1J_{m+1}+J_{m+1}T_1}{J_{m+1}T_1}.
$$
Thus, $\c$ is surjective if and only if $J_{m+2}=T_1J_{m+1}+J_{m+1}T_1$. 

Hence $J$ is generated by $J_2+\cdots+J_\ell$ if and only if $\c$ is surjective for all 
$m \ge \ell-1$. Since $T_1=B_1=H^0(\cL)$, the right-hand square in the commutative diagram just prior to this lemma 
 is canonically isomorphic to the diagram
\begin{equation*}
\xymatrix{
	T_{1}\otimes(T_{m}/J_{m})\otimes T_{1}\ar[d]\ar[r]^-{\lambda} & T_{1}\otimes(T_{m+1}/J_{m+1})\ar[d] \\
	(T_{m+1}/J_{m+1})\otimes T_{1}\ar[r]_-{\mu} & T_{m+2}/J_{m+2}
}.
\end{equation*}
Thus the map in (\ref{eq.H0L.RMN}) is surjective if and only if the induced map $\ker\lambda\to\ker\mu$ is surjective. Here we have
\begin{align*}
	\ker(\lambda)
& \ = \;  \ker\left(\frac{T_1 \otimes T_{m} \otimes T_1}{T_1 \otimes J_{m} \otimes T_1} \, \longrightarrow \,   \frac{T_1 \otimes T_{m+1}}{T_1 \otimes J_{m+1}} \right)
\\
& \ = \;  \ker\left(\frac{T_{m+2}  }{T_1J_{m} T_1} \, \longrightarrow \,   \frac{T_{m+2}}{T_1J_{m+1}} \right)
\\
& \ = \; \frac{T_1J_{m+1}  }{T_1J_{m} T_1}
\end{align*}
and
\begin{equation*}
	\ker(\mu) \ = \; \ker\left(\frac{T_{m+2}  }{J_{m+1} T_1} \, \longrightarrow \,   \frac{T_{m+2}}{J_{m+2}} \right)
	\ = \; \frac{J_{m+2}  }{J_{m+1} T_1},
\end{equation*}
and these equalities identify the map $\ker\lambda\to\ker\mu$ with $\gamma$. This completes the proof.
\end{pf}

\begin{lemma}
\label{lem.B.deg.relns}
If   $B(X,\sigma,\cL)$ is generated  in degree one, then its relations 
are generated in degree $\le \ell$ if and only if the multiplication map 
 \begin{equation}
H^0(X,\cL) \otimes H^0(X,\cG_m)\, \longrightarrow \, H^0(X,\cL \otimes \cG_m)
  \end{equation}
is onto for all $m \ge \ell - 1$.
\end{lemma}
\begin{proof}
Fix an integer $m$. By \cref{lem.SS92}, it suffices to show that the map 
$$
H^0(X,\cL) \otimes R(\cM,\cN) \, \longrightarrow \, R(\cL \otimes \cM,\cN)
$$ is onto if and only if  the map 
$H^0(X,\cL) \otimes H^0(X,\cG)\, \longrightarrow \, H^0(X,\cL \otimes \cG)$ is onto.

There are exact sequences $0 \to \cG \to  \cM  \otimes H^0(X,\cN) \to  \cM \otimes \cN  \to 0$  and
$$
  0 \to \cL \otimes  \cG \to  \cL \otimes  \cM  \otimes H^0(X,\cN) \to \cL \otimes   \cM \otimes \cN  \to 0
  $$
  and therefore exact sequences 
  $$
  0 \to H^0(X,\cG) \to H^0(X, \cM)  \otimes H^0(X,\cN) \to  H^0(X,\cM \otimes \cN)  
$$
and
$$
  0 \to H^0( X,\cL \otimes  \cG) \to  H^0(X,\cL \otimes  \cM)  \otimes H^0(\cN) \to H^0(X,\cL \otimes   \cM \otimes \cN).
  $$
Thus, there are canonical isomorphisms $R(\cM,\cN)\cong H^0(X,\cG)$ and $R(\cL \otimes \cM,\cN) \cong H^0(X,\cL \otimes \cG)$;
it follows that (\ref{eq.H0L.RMN}) is onto if and only if  
the multiplication map $H^0(X,\cL) \otimes H^0(X,\cG)\, \longrightarrow \, H^0(X,\cL \otimes \cG)$ is onto.
\end{proof}

 \subsection{Point modules for $B(X,\s,\cL)$}
 \label{sect.pt.mods}
The ``simplest part''  of the representation theory of a non-commutative algebra consists of its 
1-dimensional modules.  The ``simplest part'' of the graded representation theory of a connected graded 
$\Bbbk$-algebra, $A$ say, consists of its point modules: a {\sf point module} for $A$  is a cyclic 
 graded left $A$-module $M=M_0 \oplus  M_1\oplus \cdots$ such that $\dim_\Bbbk(M_i)=1$ for all $i \ge 0$ 
 (see \cite[p.~208]{FO89} and \cite{ATV2}).

The next result was known to Artin-Tate-Van den Bergh \cite{ATV1} and to Feigin-Odesskii \cite{FO89} sometime in the late 1980's
but it wasn't recorded explicitly.

\begin{proposition}
\label{prop.pt.mods.1}
Let $\cO_p$ be the skyscraper sheaf at a closed point $p \in X$. If $\cL$ is generated by its global sections, then
$$
M_p \; := \; (\Gamma_*\cO_p)_{\ge 0} \;=\; \bigoplus_{n =0}^\infty  H^0(X,\cO^{\s^n}_p).
$$
is a point module for $B(X,\s,\cL)$ and
\begin{equation}
\label{eq.pt.mod.isom.for.B}
(M_p)_{\ge 1}(1) \cong M_{\s^{-1} p}.
\end{equation}
An element $b \in B(X,\s,\cL)_1$ annihilates the degree-$n$ component of $M_p$ if and only if $b(\s^{-n}p)=0$.
\end{proposition}
\begin{proof}
By definition,
$$
\Gamma_*\cO_p \;=\; \bigoplus_{n \in \ZZ} H^0(X,\cL_n \otimes \cO^{\s^n}_p) \; \cong \; \bigoplus_{n \in \ZZ} H^0(X,\cO_{\s^{-n}p}).
$$
A section $b \in B(X,\s,\cL)_1=H^0(X,\cL)$ annihilates the degree-$n$ component of $\Gamma_*\cO_p$ if and only if $b(\sigma^{-n}p)=0$.  Since $\cL$ is generated by its global sections, for each $x \in X$ there is some $b \in B(X,\s,\cL)_1$ such that $b(x) \ne 0$. It follows that $M_p$ is generated by its degree-zero component as a $B(X,\s,\cL)$-module. Since $\dim_\Bbbk((M_p)_i)=1$ for all $i \ge 0$, $M_p$ is a point module for $B(X,\s,\cL)$.

The degree $i$ component of $M_{\s^{-1} p}$ is $H^0(X,\cO_{\s^{-i}(\s^{-1} p)})=H^0(X,\cO_{\s^{-i-1} p)}) = (M_p)_{i+1}$. It follows that $(M_p)_{\ge 1}(1) \cong M_{\s^{-1} p}$.
\end{proof}

\subsubsection{Remark}

When $X$ is projective and $B(X,\id,\cL)$ is finitely generated, each point module for $B(X,\id,\cL)$ is isomorphic in $\QGr(B(X,\id,\cL))$ to one of 
the $M_p$'s in \cref{prop.pt.mods.1}. We will now prove this claim. 

First, \spr{01Q0} implies that the image of the canonical morphism
\begin{equation}\label{eq:2}
	f:\,X \, \longrightarrow \,  \Proj\big(B(X,\id,\cL)\big)
\end{equation}
is dense; since $X$ is projective the image is closed so the morphism is onto. The morphism $f$ has the property that $f^{-1}(D_{+}(s))=X_{s}$ for all homogeneous $s\in B(X,\id,\cL)_{+}$, and this implies that $f(p)=\operatorname{Ann}(M_{p})$ for all closed points
$p\in X$, where the annihilator $\operatorname{Ann}(M_{p})$ inside $B(X,\id,\cL)$ is a homogeneous ideal that is maximal among those not containing $B(X,\id,\cL)_{+}$, and hence is regarded as a point of $\Proj(B(X,\id,\cL))$.

Now let $N$ be a point module for $B(X,\id,\cL)$. Since $B(X,\id,\cL)$ is finitely generated, $N$ admits a subquotient isomorphic to $B(X,\id,\cL)/\mathfrak{m}$ shifted by some degree $d\in\ZZ$, where $\mathfrak{m}$ is a point of $\Proj(B(X,\id,\cL))$ regarded as a homogeneous ideal of $B(X,\id,\cL)$. The surjectivity of \Cref{eq:2} implies that $\mathfrak{m}=f(p)$ for some $p\in X$, whence $\mathfrak{m}=\operatorname{Ann}(M_{p})$. Therefore $M_{p}$ is isomorphic to $B(X,\id,\cL)/\mathfrak{m}$ and its degree shift by $d$ is a subquotient of $N$. Since both $M_{p}$ and $N$ are point modules, $d\leq 0$ and $M_{p}(d)\cong (M_{\sigma^{-d}(p)})_{\geq -d}$ is a submodule of $N$. It follows that $N$ is isomorphic to $M_{\sigma^{-d}(p)}$ in $\QGr(B(X,\id,\cL))$.

\section{The algebras $Q_{n,k}(E,\tau)$}
\label{sect.Q}

As always, $n$ and $k$ are relatively prime integers such that  $n>k\ge 1$.

\subsection{Some notation and results from \cite{CKS2}}
 \label{sect.negative.contd.fracs}
 
\subsubsection{}
Fix $\eta \in \CC$ lying in the upper half-plane. Let $\L=\ZZ+\ZZ\eta$ and $E=\CC/\L$. 
We will usually view $E$ as an elliptic curve. 
If $r$ is a positive integer we write $E[r]$ for the $r$-torsion subgroup of $E$. 
It equals $\frac{1}{r}\L/\L$ so is isomorphic to $\ZZ_r \times \ZZ_r$.

We fix a point $\tau \in \CC$ and use the symbol $\tau$ to denote the image of $\tau$ in $E$ and the translation automorphisms
$\CC \to \CC$ and $E \to E$ given by the formula $z \mapsto z+\tau$. The meaning of $\tau$ will always be clear from the context.

\subsubsection{}
At different times we give the degree-one component of $Q_{n,k}(E,\tau)$ different interpretations as:
\begin{enumerate}
\item
an anonymous vector space $V$ with  basis $\{x_\a \; | \; \a \in \ZZ_n\}$;
  \item 
a space  $\Theta_n(\L)$ of theta functions in one variable  with  basis $\{\theta_\a(z) \; | \; \a \in \ZZ_n\}$;
  \item 
a space  $\Theta_{n/k}(\L)$ of theta functions in $g$ variables with  basis $\{w_\a(\sfz) \; | \; \a \in \ZZ_n\}$;
  \item 
$H^0(E^g,\cL_{n/k})$ where $\cL_{n/k}$ is the invertible $\cO_{E^g}$-module defined in \cref{defn.Lnk} below;
  \item 
$H^0(X_{n/k},\cL'_{n/k})$ where $X_{n/k}$ and $\cL'_{n/k}$ are defined  below.
\end{enumerate}
See \cite[\S5.3]{CKS2} for the relations between these interpretations.

\subsubsection{Negative continued fractions}\label{sssec.neg.conti.frac}
If $a,b,\ldots, c$ are integers $\ge 2$ we write
$$
[a,b,\ldots,c]   \;:=\; a-\frac{1}{b-\frac{1}{\ddots \, -\frac{1}{c}}} \, .
$$
 There is a unique integer $g\ge 1$ and a unique 
sequence of integers $n_{1},\ldots,n_g$, all $\ge 2$, such that 
$$
	\frac{n}{k}\;=\; [n_1,\ldots,n_g].
$$

\subsubsection{The translation automorphism $\s:E^g \to E^g$}
\label{sssect.sigma}
As in \cite[\S2.4]{CKS2}, we define
\begin{equation*}
	d(n_{1},\ldots,n_{g})\;:=\;\det
	\begin{pmatrix}
		n_{1}&-1&&&\\
		-1&n_{2}&-1&&\\
		&-1&\ddots&\ddots&\\
		&&\ddots&n_{g-1}&-1\\
		&&&-1&n_{g}
	\end{pmatrix}
\end{equation*}
Using this notation, we define $k_i$ and $l_i$ for $i=0,\ldots,g+1$ by
\begin{equation*}
	k_{i}\;:=\;d(n_{i+1},\ldots,n_{g})\quad\text{and}\quad l_{i}\;:=\;d(n_{i-1},\ldots,n_{1})
\end{equation*}
with conventions $k_{g}=1$, $k_{g+1}=0$, $l_{0}=0$, and $l_{1}=1$. By \cite[Prop.~2.6]{CKS2}, we have the formulas
$$
\frac{k_{i-1}}{k_i} \;=\; [n_i,\ldots,n_g]
\qquad \text{and} \qquad 
\frac{l_{i+1}}{l_i} \;=\; [n_{i},\ldots,n_1]
$$
for $i=1,\ldots,g$, and
\begin{equation}\label{eq.nk.sp}
	k_{0}\;=\;l_{g+1}\;=\;d(n_{1},\ldots,n_{g})\;=\;n,\qquad
	k_{1}\;=\;k,\qquad
	l_{g}\;=\;k',
\end{equation}
where $n>k'\geq 1$ and $kk'\equiv 1$ ($\mathrm{mod}$ $n$). In \cite[\S2.4.1]{CKS2}, we also observed
\begin{equation}\label{eq.kl.ind}
	k_{i}n_{i}\;=\;k_{i-1}+k_{i+1}\quad\text{and}\quad l_{i}n_{i}\;=\;l_{i-1}+l_{i+1}
\end{equation}
for $i=1,\ldots,g$. We will use these in \cref{sect.applic}.

For $i=1,\ldots,g$,  we  define $ \tau_i :=  (k_i+l_i-n)\tau$ and define the automorphism $\s:\CC^g \to \CC^g$ by
$$ 
  \s(z_1,\ldots,z_g) \;:=\;  (z_1+\tau_1,\ldots,z_g+\tau_g).
$$
Because $(\CC^g,+)$ is an abelian group, all translation automorphisms  of it commute with one another. 
In particular,  $\s$ 
commutes with the translation action of $\L^g$ on $\CC^g$, 
and therefore induces an automorphism of $E^g=\CC^g/\L^g$  that we will also 
 denote by $\s$.

\subsubsection{The  group $\Sigma_{n/k} \subseteq \Aut(E^g)$}
\label{sect.Sigma.n/k}
 Let $\Sigma_{n/k} := \langle s_i \; | \; n_i=2 \rangle$ 
where $s_{i}:E^g \to E^{g}$ is the automorphism  
	\begin{equation*}
		s_{i}(z_{1},\ldots,z_{g})\; :=\; (z_{1},\ldots,z_{i-1},z_{i-1}-z_{i}+z_{i+1},z_{i+1},\ldots,z_{g})
	\end{equation*} 
with the convention that $z_{0}=0$ and $z_{g+1}=0$.

\subsubsection{The invertible sheaf $\cL_{n/k}$ and the characteristic variety $X_{n/k}$}
\label{sect.Xn/k}

Following Odesskii and Feigin \cite[\S3.3]{FO89}, we define an invertible sheaf $\cL_{n/k}$ on $E^g$ as follows.\footnote{In \cite[\S 3.1.3]{CKS2} we relate this definition to Odesskii and Feigin's original definition.}
Let $\cL=\cO_E((0))$ be the degree-one invertible $\cO_E$-module corresponding to the divisor $(0)$, and define
\begin{equation}
\label{defn.Lnk}
 \cL_{n/k} \; : =\; \big( \cL^{n_1} \boxtimes \cdots \boxtimes \cL^{n _g} \big) \otimes \left( \bigotimes_{j=1}^{g-1} {\rm pr}^*_{j,j+1}\cP\right)
\end{equation}
where $\cP$ is the Poincar\'e bundle $(\cL^{-1} \boxtimes \cL^{-1}) (\D)$ on $E \times E$ and
 ${\rm pr}_{j,j+1} :E^g \to E \times E$ is the projection $(z_1,\ldots,z_g) \mapsto (z_j,z_{j+1})$  and $\D=\{(z,z) \; | \; z \in E\}$.

Thus $\cL_{n/k}=\cO_{E^g}(D_{n/k})$ where 
\begin{equation}\label{eq:d} 
D_{n/k}  \; :=\; \sum_{i=1}^g E^{i-1} \times D_i \times E^{g-i} \, + \, \sum_{j=1}^{g-1} \D_{j,j+1}
\end{equation}
and $ \D_{j,j+1} = {\rm pr}^*_{j,j+1} \D$ and $D_{i}:=(n_{i}-2+\delta_{i,1}+\delta_{i,g})(0)$. If $g\geq 2$, then
\begin{equation}
\label{defn.Di}
D_i \; = \; 
\begin{cases}
(n_i-1)(0) & \text{if $i \in \{1,g\}$}
\\
(n_i-2)(0) & \text{if $2 \le i \le g-1$.}
\end{cases}
\end{equation}

A \textsf{standard divisor} of \textsf{type} $(n_{1},\ldots,n_{g})$ is a divisor of the form
\begin{equation*}
	D_{\fd_{i},z_{j}}\;:=\;\sum_{i=1}^{g}E^{i-1}\times\fd_{i}\times E^{g-i}+\sum_{i=1}^{g-1}\Delta^{z_{j}}_{j,j+1}
\end{equation*}
where $\fd_{i}$ ($1\leq i\leq g$) are effective divisors on $E$ of respective degrees $(n_{i}-2+\delta_{i,1}+\delta_{i,g})$, $z_{j}\in E$ ($1\leq j\leq g-1$) are points, $\Delta^{z_{j}}_{j,j+1}=\operatorname{pr}^{*}_{j,j+1}\Delta^{z_{j}}$, and
\begin{equation*}
	\Delta^{z_{j}}\;:=\;\{(z,z+z_{j})\;|\;z\in E\}\;\subseteq\;E^{2}.
\end{equation*}

\begin{proposition}
\label{2-char-var.prop}
\cite[\S\S 3 and 4]{CKS2}
If $D$ is a standard divisor of type $(n_{1},\ldots,n_{g})$ with $n_{i}\geq 2$ for all $i$, then   $\cO_{E^{g}}(D)$ has the following properties:
\begin{enumerate}
\item \label{item:1}
it is base-point free or, equivalently, generated by its global sections;
\item \label{item:2}
it is ample;
\item \label{item:3}
it is very ample if and only if $n_i \ge 3$ for all $i$;
\item \label{item:4}
  $\dim_\CC\left( H^0(E^g,\cO_{E^{g}}(D))\right)=n$;
\item \label{item:5}
  $H^q(E^g,\cO_{E^{g}}(D)) =0$ for all $q \ge 1$.
\end{enumerate}
In particular, $\cL_{n/k}=\cO_{E^{g}}(D_{n/k})$ satisfies these properties.
\end{proposition}

Since  $\cL_{n/k}$ is base-point free, the complete linear system $|D_{n/k}|$ determines a morphism 
$$
\Phi_{n/k}:E^g \, \longrightarrow \, \PP^{n-1} \,=\, \PP\left(H^0(E^g,\cL_{n/k})^*\right).
$$  
The {\sf characteristic variety} for $Q_{n,k}(E,\tau)$ is   $X_{n/k} :=  \text{the image of }\, \Phi_{n/k}$.

\subsubsection{Special cases}\label{sssec.sp.cases}
The following examples illustrate some of the possibilities.
\begin{enumerate}
   \item 
If $n\geq 3$ and $k=1$, then $X_{n/k}=E$, $\s$ is translation by $(2-n)\tau$, and $\cL_{n/k}$ is an invertible $\cO_{E}$-module of degree $n$.
  \item 
  If $[n_1,\ldots,n _g]=[2,\ldots,2]$, then $g=n-1=k$, $\Phi_{n/k}:E^g \to \PP^{n-1}$ is surjective (\cite[\S 4.6.2]{CKS2}), 
  $Q_{n,n-1}(E,\tau)$ is a polynomial ring on $n$ variables, and $\QGr\big(Q_{n,n-1}(E,\tau)\big)=\Qcoh(\PP^{n-1})$ (see the footnote in \cref{sect.good.props}). 
   \item 
  If  $n_i \ge 3$ for all $i$, then $X_{n/k} \cong E^g$ is an isomorphism, and conversely (\cite[\S 4.6.1]{CKS2}). 
  \item
  If $f_0=f_1=1$ and $f_{i+1}=f_i+f_{i-1}$ and $(n,k)=(f_{2g+1},f_{2g-1})$, then $\frac{n}{k}=[3,\ldots,3]$ and $X_{n/k} \cong E^g$.
\item
If  $m \ge 3$ and $[n_1,\ldots,n _g]$ is either $[m,2,\ldots,2,2]$ or $[2,2,\ldots,2,m]$,  then $X_{n/k} \cong S^gE$, and conversely (\cite[Cor.~4.24]{CKS2}). 
\item
$X_{(2k+1)/k} \cong S^gE$ since $\frac{2k+1}{k}=[3,2,\ldots,2]$. See \cref{sssec.geom.description} for the significance of this case.
\item
$X_{n^2/n-1} \cong S^{n-1}E$ when $n \ge 2$, because $\frac{n^2}{n-1}=[n+2,2,\ldots,2]= [n+2,2^{n-2}]$ (see \cref{prop.7.1}\cref{item.nk.mfirst}).
The algebras $Q_{n^2,n-1}(E,\tau)$ were studied by Cherednik in \cite{Ch86-R-matrix}. They are, in a sense, homogenized 
elliptic versions of the quantized enveloping algebras $U_q(\fsl_n)$. Or, conversely, the $U_q(\fsl_n)$'s are ``degenerations'' of $Q_{n^2,n-1}(E,\tau)$.
A detailed examination of this degeneration process for $n=2$ is carried out in \cite{CSW}.
\end{enumerate}

\subsection{Twisted homogeneous coordinate rings related to $Q_{n,k}(E,\tau)$}
\label{ssect.Q.B}
Let
$$
\cL'_{n/k} \; :=\; \cO_{\PP^{n-1}}(1)\big\vert_{X_{n/k}}.
$$
In \cref{cor.map.to.B} we obtain a graded $\CC$-algebra homomorphism
$Q_{n/k}(E,\tau) \to B(X_{n/k},\s',\cL'_{n/k})$ that is an isomorphism in degree one.

\begin{theorem} \label{thm.thcr.qnk}
  Let $f:E^g \to X_{n/k}$ be the co-restriction of the morphism $\Phi_{n/k}$.
 \begin{enumerate}
 \item\label{thcr.qnk.quot} The map $f:E^g \to X_{n/k}$ is a quotient morphism for the action of $\Sigma_{n/k}$ on $E^g$.
\item\label{thcr.qnk.aut} There is a unique automorphism $\s'$ of $X_{n/k}$ such that $\Phi_{n/k} \circ \s = \s' \circ \Phi_{n/k}$.
 \item\label{thcr.qnk.tri.mor} There is a morphism of triples $(f,u):(E^g,\s,\cL_{n/k}) \to (X_{n/k},\s',\cL_{n/k}')$.
\item\label{thcr.qnk.hom} There is a homomorphism of graded algebras $B(X_{n/k},\s',\cL_{n/k}') \longrightarrow B(E^g,\s,\cL_{n/k})$.
\item\label{thcr.qnk.act} The group $\Sigma_{n/k}$ acts as automorphisms of $B(E^g,\s,\cL_{n/k})$, and the co-restriction 
of the homomorphism in \cref{thcr.qnk.hom} is an isomorphism
\begin{equation}
\label{eq:homom} 
B(X_{n/k},\s',\cL_{n/k}')  \stackrel{\sim}{\longrightarrow} B(E^g,\s,\cL_{n/k})^{\Sigma_{n/k}}. 
\end{equation}
\item\label{thcr.qnk.ample}
Every ample invertible sheaf on $X_{n/k}$, in particular $\cL_{n/k}'$, is $\s'$-ample.
\item\label{thcr.qnk.isom}
There are isomorphisms $\cL_{n/k}'   \cong (f_*\cL_{n/k})^{\Sigma_{n/k}}$ and $\cL_{n/k} \cong f^*\cL_{n/k}'$.
\end{enumerate}
 \end{theorem}
 \begin{proof}
\cref{thcr.qnk.quot}
See \cite[Cor.~4.19]{CKS2}.
 
\cref{thcr.qnk.aut}
 This follows from \cref{thcr.qnk.quot} and  \cite[Prop.~2.10]{CKS2}.
 
\cref{thcr.qnk.tri.mor}
This follows from the definitions of $X_{n/k}$ and $\cL_{n/k}'$.
More explicitly, if $\iota:X_{n/k} \to \PP(H^0(E^g,\cL_{n/k})^*)$ is the inclusion morphism, then $\Phi_{n/k} =\iota \circ f$ and 
$\cL_{n/k} \cong \Phi_{n/k}^{*} \cO_{\PP^{n-1}}(1) =  f^*\iota^* \cO_{\PP^{n-1}}(1) = f^*\cL_{n/k}'$ so we take $u$ to be 
the canonical isomorphism $f^*\cL_{n/k}' \to \cL_{n/k}$.

\cref{thcr.qnk.hom}
 This follows from \cref{thcr.qnk.tri.mor} and \cref{pr.triples}.
 
\cref{thcr.qnk.act}
 By \cite[Prop.~4.11]{CKS2}, $\cL_{n/k}$ is  a $\Sigma_{n/k}$-equivariant sheaf on $E^g$;  i.e., 
there are isomorphisms $t_\g:\cL_{n/k} \to \g^*\cL_{n/k}$, $\g \in \Sigma_{n/k}$,
such that  $t_{\a\b}=\b^*(t_\a) \circ t_\b$ for all $\a,\b \in \Sigma_{n/k}$. Since the action of $\s$  commutes with that of $\Sigma_{n/k}$,
each pair $(\g,t_\g^{-1})$ is an automorphism of the triple $(E^g,\s,\cL_{n/k})$ and therefore induces (by 
 functoriality) a right action of $\Sigma_{n/k}$ as automorphisms of $B(E^g,\s,\cL_{n/k})$.\footnote{The triple $(E^g,\s,\cL_{n/k})$ is a 
 $\Sigma_{n/k}$-triple in the terminology of \cite[p.~27]{ST94}. In \cite{ST94} the group acts  freely but the 
 terminology extends to the present situation.} 
 
For brevity we write $\cL=\cL_{n/k}$ and $\cL'=\cL_{n/k}'$. 
 
Comparing the degree-$(m+1)$ components  in (\ref{eq:homom}), we must show that the natural map
$$
H^0(E^g/\Sigma_{n/k},\cL' \otimes \s'^*\cL' \otimes \cdots \otimes (\s'^m)^*\cL') 
\; \longrightarrow \;
H^0(E^g,\cL \otimes \s^*\cL \otimes \cdots \otimes (\s^m)^*\cL)^{\Sigma_{n/k}}
$$
is an isomorphism for all $m \ge 0$.

For simplicity we assume $m=2$; all other cases are essentially the same. We will show that the natural map
\begin{equation}
\label{eq:m=2.case}
H^0(E^g/\Sigma_{n/k},\cL' \otimes \s'^*\cL') 
\; \longrightarrow \;
H^0(E^g,\cL \otimes \s^*\cL)^{\Sigma_{n/k}}
\end{equation}
is an isomorphism. 
Since $\s' \circ f=f \circ \s$, $f^*\circ (\s'^i)^*=(\s^i)^* \circ f^*$ for all $i$; hence, since $f^*$ commutes with $\otimes$
and $\cL\cong f^*\cL'$, 
$$
\cL \otimes \s^*\cL  \; \cong \; f^*(\cL' \otimes \s'^*\cL').
$$
Therefore, by the projection formula, $f_*(\cL \otimes \s^*\cL) \cong f_*\cO_{E^g} \otimes  \cL' \otimes \s'^*\cL'$. 
Since the action of $\Sigma_{n/k}$ on $\cL'$ and $\s'^*\cL'$ is trivial it follows that 
$$
\big(f_*(\cL \otimes \s^*\cL)\big)^{\Sigma_{n/k}} \; \cong \;\big(f_*\cO_{E^g}\big)^{\Sigma_{n/k}} \otimes  \cL' \otimes \s'^*\cL'  
\;=\;   \cL' \otimes \s'^*\cL'. 
$$
Hence
\begin{align*}
H^0(E^g,\cL \otimes \s^*\cL)^{\Sigma_{n/k}} &  \; \cong \; H^0(E^g/\Sigma_{n/k},f_*(\cL \otimes \s^*\cL ) )^{\Sigma_{n/k}}
\\
&  \; \cong \;  H^0\big(E^g/\Sigma_{n/k},(f_*(\cL \otimes \s^*\cL ) )^{\Sigma_{n/k}}\big)
\\
&  \; \cong \;  H^0\big(E^g/\Sigma_{n/k}, \cL' \otimes \s'^*\cL' ). 
\end{align*}
Thus the map in \cref{eq:m=2.case} is an isomorphism in degree two.

\cref{thcr.qnk.ample}
This is an immediate consequence of \cref{thcr.qnk.quot} and \cref{cor.sigma-ample.2}.

\cref{thcr.qnk.isom}
See Remark 4.10, Proposition 4.5, and the remarks at the beginning of \S4 in \cite{CKS2}.
\end{proof}

\begin{remark}\label{re.bdescent}
  Part \Cref{thcr.qnk.act} of \Cref{thm.thcr.qnk} is essentially \cite[Prop.~2.5]{ST94}. The difference is that in the latter the action of the group ($\Sigma_{n/k}$) is free. An examination of the proof, however, reveals that the only consequence of freeness needed there is the fact that
  \begin{equation*}
    \cL_{n/k}\cong f^*\cL'_{n/k} \cong f^*((f_*\cL_{n/k})^{\Sigma_{n/k}}). 
  \end{equation*}
  This, in turn, is a consequence of \cite[Prop.~4.5]{CKS2}.
\end{remark}
 
By \Cref{prop.B-op}, $B(X,\s^{-1},\cL)^{\mathrm{op}} \cong B(X,\s,\cL)$.  In the context of elliptic algebras more is true.

\begin{theorem}\label{th.-tau}
$B(X_{n/k},\sigma',\cL'_{n/k}) \cong B(X_{n/k},(\sigma')^{-1},\cL'_{n/k})$. 
\end{theorem}
\begin{proof}
  Let $[-1]:E^g \to E^g$ be the automorphism $\sfz \mapsto -\sfz$ and let $\mu:X_{n/k} \to X_{n/k}$ be the automorphism that is the descent of $[-1]$. Since $\s$ is a translation automorphism, $\s^{-1}=[-1]^{-1} \circ \s \circ [-1]$; hence $(\sigma')^{-1}=\mu^{-1}\sigma'\mu$.  We will complete the proof by applying \Cref{pr.tw-by-mu} to $(X_{n/k},\s',\cL_{n/k}')$ after showing that $\mu^*\cL'_{n/k}\cong \cL'_{n/k}$.

The action of $\ZZ_2=\{\id,[-1]\}$ on $E^g$ preserves the effective divisor $D_{n/k}$ as a subscheme. 
By \cite[Lem.~4.8]{CKS2}, this gives a $\ZZ_2$-equivariant structure on $\cL_{n/k}$. Setting $V=H^0(E^g,\cL_{n/k})$, we now have a $\ZZ_2$-action on $V$ inducing one on $\PP(V^*)$ together with a compatible $\ZZ_2$-equivariant structure on the twisting sheaf 
$\cO_{\PP(V^*)}(1)$ (see, e.g., \cite[Prop.~1.7]{mumf-git}).
  The morphism $\Phi_{n/k}:E^g \to \PP(V^*)$ is $\ZZ_2$-equivariant and the generator of $\ZZ_2$ acts on $X_{n/k}$ as $\mu$. 
  The equivariant structure on $\cO_{\PP(V^*)}(1)$ restricts to one on
  \begin{equation*}
    \cO_{\PP(V^*)}(1)|_{X_{n/k}}\cong \cL'_{n/k},
  \end{equation*}
  whence the desired isomorphism $\mu^*\cL'_{n/k}\cong\cL'_{n/k}$ that completes the proof.
\end{proof}

Let $\Theta_{n/k}(\L)$ be the space of theta functions in $g$ variables defined in \cite[\S2.7 and \S5.2]{CKS2}, and 
let $\{w_\a(\sfz) \; | \; \a \in \ZZ_n\}$ be the basis for $\Theta_{n/k}(\L)$ in \cite[\S5.1.1]{CKS2}.
We make the identifications $Q_{n,k}(E,\tau)_1 = H^0(E^g,\cL_{n/k}) = \Theta_{n/k}(\L)$ described in \cite[\S5.3]{CKS2}. The
identifications are such that $x_\a=w_\a(\sfz)$, and the morphism  $\Phi_{n/k}:E^g \to \PP(H^0(E^g,\cL_{n/k})^*)$
is given by $\Phi_{n/k}(\sfz)=(w_0(\sfz),\ldots,w_{n-1}(\sfz))$. 

\begin{proposition}
\cite[Cor.~5.9]{CKS2}
\label{prop.cor5.13.in.2}
The quadratic relations for $Q_{n,k}(E,\tau)$ vanish on the graph of the automorphism $\s':X_{n/k} \to X_{n/k}$.
\end{proposition}

\begin{corollary}
\label{cor.map.to.B}
There are $\CC$-algebra homomorphisms
\begin{equation}
\label{map.Q.to.B}
\begin{tikzcd}
	Q_{n,k}(E,\tau)\ar[r,"\Psi_{n/k}"] & B(X_{n/k},\s',\cL'_{n/k})\ar[r] & B(E^g,\s,\cL_{n/k})
\end{tikzcd}
\end{equation}
that  are isomorphisms in degree one.
\end{corollary}
\begin{proof}
Let $\varphi:Q_{n,k}(E,\tau)_1 \longrightarrow \Theta_{n/k}(\L)=H^0(E^g,\cL_{n/k})=H^0(X_{n/k},\cL'_{n/k})=B(X_{n/k},\s',\cL'_{n/k})_1$ 
be the vector space isomorphism defined by $\varphi(x_\a)=w_\a(\sfz)$.  
By \cref{prop.map.to.B}, $\varphi$ extends to the desired algebra homomorphism if the degree-two
relations for $Q_{n,k}(E,\tau)$ vanish on 
$$
 \big\{\big(\Phi_{n/k}(\sfz), \Phi_{n/k}(\s(\sfz))\big) \; \big\vert \; \sfz \in E^g\big\}.
$$
They do by \cref{prop.cor5.13.in.2}.
\end{proof}

\begin{corollary}
\label{cor.commutative}
Assume $k \ne n-1$. Suppose the homomorphism $Q_{n,k}(E,\tau) \to B(X_{n/k},\sigma',\cL_{n/k}')$ is surjective.
If $Q_{n,k}(E,\tau)$ is a finitely generated module over its center, then $\tau$ has finite order. 
\end{corollary}
\begin{proof}
Since the homomorphism is surjective, $B(X_{n/k},\sigma',\cL_{n/k}')$ is also finite over its center. 
Let $K$ denote the field of rational functions of $X_{n/k}$, and let $K[t^{\pm 1}; \sigma']$ 
denote the skew Laurent polynomial extension associated to the automorphism $\s'$ of 
$K$. By \cite[Prop.~2.1]{ST94}, $K[t^{\pm 1}; \sigma']$ is a localization of 
$B(X_{n/k},\sigma',\cL_{n/k}')$ so it is also finite over its center. 
It is well-known, and easy to show, that the fact that $K[t^{\pm 1}; \sigma']$ is finite over its center implies
$\sigma'$ has finite order as an automorphism of $K$ and 
hence as an automorphism of $X_{n/k}$. If that order is $b$, then $\sigma^b(p)$ and $p$ have the same
image in $X_{n/k}$ for all $p \in E^g$. Thus, if $p \in E^g$, then $\sigma^b(p)=\gamma\cdot p$ for some $\gamma \in \Sigma_{n/k}$.
If $a$ denotes the size of $\Sigma_{n/k}$, then 
$\sigma^{ab}(p)=p$ for all $p \in E^g$.

But $\sigma$ is translation by $(\sfk+\sfl-\sfn)\tau$ so, in particular, $(k_1+l_1-n)\tau$ has finite order. But $k_1=k$ and 
$l_1=1$ so $k_1+l_1-n \ne 0$. Hence $\tau$ has finite order. 
\end{proof}

It is stated at  \cite[p.~209, Rmk.~1]{FO89} and \cite[p.~1143]{Od-survey} that $Q_{n,n-1}(E,\tau)$ is a polynomial ring for all $E$ and $\tau$. We proved this in \cite[Prop.~5.5]{CKS1}. The proof is a direct calculation and also uses the fact that the space of relations for $Q_{n,n-1}(E,\tau)$ has dimension $\binom{n}{2}$.
The direct calculation part has an alternative proof using the twisted homogeneous coordinate ring.
 
 \begin{corollary} 
  $Q_{n,n-1}(E,\tau)$ is a polynomial ring on $n$ variables.
\end{corollary}  
\begin{proof} 
An induction argument shows that $\frac{n}{n-1}=[2,\ldots,2]$ where the number of 2's is $n-1$. Thus $g=n-1$.
 Hence $\Sigma_{n/(n-1)}=\Sigma_{g+1}$ and $X_{n/(n-1)}\cong \PP^{n-1}$. 
 \cref{cor.map.to.B} therefore provides a homomorphism $Q_{n,n-1}(E,\tau) \to B(\PP^{n-1},\s',\cO_{\PP^{n-1}}(1))$ that is 
 surjective in degree one. 
 
The numbers $k_i$ and $l_i$ defined in \cref{sssect.sigma} are 
 $(k_1,\ldots,k_{n-1})=(n-1,\ldots,2,1)$ and $(l_1,\ldots,l_{n-1})=(1,2,\ldots,n-1)$ so $k_i+l_i-n=0$ for
all $i=1,\ldots,n-1$. Hence  $\s$ and $\s'$ are the identity morphisms. In particular, 
$B(X_{n/k},\s',\cL'_{n/k})=B(\PP^{n-1},\id,\cO(1))=\bigoplus_{i\ge 0} H^0(\PP^{n-1},\cO(i))$. This is a polynomial ring on 
$n$ variables so the homomorphism $Q_{n,n-1}(E,\tau) \to B(\PP^{n-1},\s',\cO_{\PP^{n-1}}(1))$
 is surjective. It is also injective because the quadratic relations for both $Q_{n,n-1}(E,\tau)$ and $B(\PP^{n-1},\id,\cO(1))$ 
 span vector spaces of dimension ${n \choose 2}$.
\end{proof}

 \begin{proposition}
 There is a commutative diagram 
 $$
 \xymatrix{
 Q_{n,k}(E,\tau)^{\mathrm{op}}  \ar@{=}[d]  \ar[r]    & B(E^g,\s,\cL_{n/k})^{\mathrm{op}}   \ar[d]^{\varphi}
 \\
  Q_{n,k}(E,-\tau) \ar[r]    & B(E^g,\s^{-1},\cL_{n/k})
  }
  $$
 in which the horizontal arrows are given by \cref{map.Q.to.B} and $\varphi$ is the  isomorphism in \cref{prop.B-op}.
 \end{proposition}
 \begin{proof}
 By \cite[Prop.~3.22]{CKS1},  $Q_{n,k}(E,\tau)^{\mathrm{op}}=  Q_{n,k}(E,-\tau)$ because the space of relations for 
 $Q_{n,k}(E,\tau)^{\mathrm{op}}$ is the same subspace of $V^{\otimes 2}$ as the space of relations for $Q_{n,k}(E,-\tau)$.  
 Since $Q_{n,k}(E,\tau)$ is generated by its degree-one component, to show that the diagram commutes we need only check it
commutes in degree one. This is true because $\varphi$ is the identity map in degree one, and so are the horizontal maps.
  \end{proof}

\subsubsection{Remark}
The homomorphisms in \cref{cor.map.to.B} do not give all homomorphisms to twisted homogeneous coordinate rings.
For example, there are four surjective homomorphisms from $Q_{4,1}(E,\tau)$ to the polynomial ring in one variable corresponding to
the four isolated point modules. If we present $Q_{4,1}(E,\tau)$ as Sklyanin does, then those homomorphisms are obtained
by quotienting out three of the four generators for the algebra (\cite[Prop.~5.2]{LS93}).
Similarly, the remarks at the end of \cite[\S5.5]{CKS2} exhibit four surjective homomorphisms from $Q_{8,3}(E,\tau)$ to the 
polynomial ring on two variables.

\section{Semistable and locally free  $\cO_E$-modules}
\label{sect.semistable}

We need some standard results on semistable locally free sheaves on a smooth projective curve $C$. 

Loring Tu's paper \cite{tu} is a good source for these results when $C$ is the elliptic curve  $E$. 

In this and subsequent sections, a locally free sheaf always means a locally free \emph{coherent} sheaf, i.e., of finite type.

\subsection{Semistable $\cO_C$-modules}

Let $C$ be a smooth projective curve.

The {\sf slope} of a non-zero locally free    $\cO_C$-module $\cF$ is the number
\begin{equation*}
  \mu(\cF) \; := \; \frac{\mathrm{\deg(\cF)}}{\rk(\cF)}.
\end{equation*}
We say $\cF$ is
\begin{itemize}
  \item{}
   {\sf semistable}  if $\mu(\cF')  \le  \mu(\cF)$ for all non-zero $\cF' \subseteq \cF$ and 
  \item{}
{\sf stable}  if $\mu(\cF')  <  \mu(\cF)$ for all non-zero $\cF' \subsetneq \cF$.  
\end{itemize}

\begin{lemma}
\label{le.seesaw}
	If $0\to\cF'\to\cF\to\cF''\to 0$ is an exact sequence of non-zero locally free  $\cO_C$-modules, then either
	\begin{enumerate}
		\item\label{le.seesaw.st} $\mu(\cF')<\mu(\cF)<\mu(\cF'')$ or
		\item\label{le.seesaw.eq} $\mu(\cF')=\mu(\cF)=\mu(\cF'')$ or
		\item\label{le.seesaw.ne} $\mu(\cF')>\mu(\cF)>\mu(\cF'')$.
	\end{enumerate}
	In particular, $\min\{\mu(\cF'),\mu(\cF'')\}\leq\mu(\cF)\leq\max\{\mu(\cF'),\mu(\cF'')\}$.
	\begin{enumerate}\setcounter{enumi}{3}
		\item\label{le.seesaw.ss}
		If $\cF$ is semistable, then  \cref{le.seesaw.st} or \cref{le.seesaw.eq} holds.
		\item\label{le.seesaw.sb}
		 If $\cF$ is stable, then \cref{le.seesaw.st} holds. 
	\end{enumerate}
\end{lemma}

\begin{lemma} 
\label{lem.indec.summands.semistable}
All direct summands of a semistable $\cO_C$-module $\cF$ have the same slope as $\cF$.
\end{lemma}

\begin{lemma} 
\label{lem.slope.summands}
If   $\cA$ and $\cB$ are semistable and $0 \to \cA \to \oplus \cV_i \to \cB \to 0$ is an exact sequence of locally free $\cO_C$-modules, then 
$\mu(\cV_i) \ge {\rm min}\{\mu(\cA),\mu(\cB)\}$ for all $i$.
\end{lemma}
\begin{proof}
Fix $j$ and let $\pi: \oplus \cV_i \to \cV_j$ be the projection. Let $\cU_j$ be the image of $\cA$ in $\cV_j$. 
There is a commutative diagram
$$
\xymatrix{
0 \ar[r] &\cA \ar[d] \ar[r] &\oplus \cV_i  \ar[d]^\pi \ar[r] &\cB \ar[r] \ar[d] & 0
\\
0 \ar[r] &\cU_j \ar[r] &\cV_j \ar[r] &\cV_j/\cU_j \ar[r] & 0
}
$$
with exact rows. The left-most vertical arrow is an epimorphism by definition so, by the Snake Lemma, the right-most 
vertical arrow is also an epimorphism. Since $\cU_j$ and $\cV_j/\cU_j$ are quotients of the 
semistable sheaves $\cA$ and $\cB$, respectively, \cref{le.seesaw} tells us that $\mu(\cU_j) \ge \mu(\cA)$ and
$\mu(\cV_j/\cU_j) \ge \mu(\cB)$.  \cref{le.seesaw} also tells us that $\mu(\cV_j) \ge \min\{\mu(\cU_j),\mu(\cV_j/\cU_j)\}$.
The result follows. 
\end{proof}

\begin{lemma} 
If $\cU$ and $\cV$ are non-zero locally free $\cO_C$-modules, then 
$$
\mu(\cU \otimes \cV) \; =\; \mu(\cU) \, +\, \mu(\cV).
$$
\end{lemma}
\begin{proof}
This follows from the fact that $\deg(\cU \otimes \cV) = \deg(\cU)\rank(\cV) + \deg(\cV)\rank(\cU)$.
\end{proof}

\begin{lemma}\label{le.tu}
\cite[Appendix~A]{tu}  
\label{le.indec}
Let $\cF$ be a locally free $\cO_E$-module. If $\cF$ is indecomposable, then it is semistable and is stable if and only if 
its degree and rank are coprime.\footnote{Polishchuk uses the Harder-Narasimhan filtration to show indecomposability  implies semistability  \cite[Lem.~14.5]{Polishchuk-book}.}
\end{lemma}

\begin{lemma}\label{le.tens-ss}
 \cite[Thm.~2.5]{Mar81}
 \label{le.tens}
Assume ${\rm char}(\Bbbk)=0$. If $\cU$ and $\cV$ are semistable  locally free $\cO_C$-modules so is $\cU \otimes \cV$. 
\end{lemma}

\begin{proposition}
\label{prop.tens-ss}
  If ${\rm char}(\Bbbk)=0$, then  the tensor product of
  two indecomposable locally free $\cO_E$-modules is a direct sum of indecomposable sheaves with equal slopes.
\end{proposition}
\begin{proof}
Combine \cref{le.tu,le.tens-ss,lem.indec.summands.semistable}.
\end{proof}

\begin{lemma}
\label{le.semist.pos.deg}
Let $\cF$ be a semistable locally free $\cO_E$-module.
\begin{enumerate}
  \item\label{semist.pos.deg.pos}
  If  $\deg(\cF)>0$, then $\dim H^{0}(\cF)=\deg(\cF)$ and $\dim H^{1}(\cF)=0$. 
  \item\label{semist.pos.deg.neg}
  If  $\deg(\cF)<0$, then $H^0(\cF)=0$
  \item\label{semist.pos.deg.gl.sec}
If $\cF$ is non-zero and generated by its global sections, then $\deg(\cF)>0$.
\item\label{semist.pos.deg.slope}
If $\mu(\cF)>1$, then $\cF$  is generated by its global sections.
\end{enumerate}
\end{lemma}
\begin{proof}
\cref{semist.pos.deg.pos}
This is \cite[Lem.~17]{tu}.

\cref{semist.pos.deg.neg}
If $\cF$ has a non-zero section, then there is a non-zero map $\cO_E \to \cF$. The image of this map is isomorphic to $\cO_E$
so the semistability of $\cF$ implies $0=\mu(\cO_E) \le \mu(\cF)$ whence $0 \le \deg(\cF)$. 

\cref{semist.pos.deg.gl.sec}
This is an immediate consequence of \cref{semist.pos.deg.neg}.

\cref{semist.pos.deg.slope}
Let $p \in E$. Since $\cF(-(p))$ is semistable of slope $\mu(\cF)-1>0$ its degree is positive, whence $H^1(\cF(-(p)))=0$
by \cref{semist.pos.deg.pos}. Therefore, applying $\cF \otimes -$ to
the sequence $0 \to \cO(-(p)) \to \cO_E \to \cO_p \to 0$ and taking cohomology yields an exact sequence $0 \to H^0(\cF(-(p))) \to H^0(\cF) \to
H^0(\cF \otimes \cO_p) \to 0$. The fact that $H^0(\cF) \to H^0(\cF \otimes \cO_p)$ is onto for all $p$, 
together with Nakayama's lemma, tells us that $\cF$ is generated by its global sections. 
\end{proof}

 \subsection{Surjectivity of multiplication maps}

\begin{theorem}\label{le.otimes}
  Let $\cU$ and $\cV$ be semistable locally free $\cO_E$-modules generated by their global
  sections. 
  If
    \begin{equation}\label{eq:ineq}
    \frac 1{\mu(\cU)} \, + \, \frac 1{\mu(\cV)} \; < \; 1,
  \end{equation}
  then the canonical map   $H^0(E,\cU)\otimes H^0(E,\cV)\to H^0(E,\cU\otimes \cV)$ is onto.\footnote{Since  $\cU$ and $\cV$ are generated by their global sections, their degrees are $\ge 0$;  hypothesis  \Cref{eq:ineq} implies 
  their degrees (or, equivalently, their slopes) are, in fact, positive.}
\end{theorem}

\begin{proof}
Because they are semistable, $\cU$ and $\cV$ are direct   sums of indecomposable summands of slopes equal to $\mu(\cU)$ and
  $\mu(\cV)$ respectively, so it suffices to prove the result when $\cU$ and
  $\cV$ are indecomposable (and therefore semistable); 
   we therefore make this assumption in the rest of the proof. We also assume that $\cU \ne 0$ and $\cV \ne 0$.
  By \Cref{le.semist.pos.deg}\cref{semist.pos.deg.gl.sec} and \cref{semist.pos.deg.pos}, $\deg(\cU)>0$ and $H^1(\cU)=0$.

 Tensoring the exact sequence
  \begin{equation}\label{eq:res}
	\begin{tikzcd}
		0\ar[r] & \cK\ar[r,"\varepsilon"] & H^0(E,\cU)\otimes\cO\ar[r] & \cU\ar[r] & 0
	\end{tikzcd}
  \end{equation}
with $\cV$ and taking cohomology produces an exact sequence 
  \begin{equation*}
  0 \to H^0(E,\cK \otimes \cV) \to H^0(E,\cU)\otimes H^0(E,\cV)\to H^0(E,\cU\otimes \cV)\to H^1(E,\cK\otimes \cV)
  \end{equation*}
  We will prove the theorem by showing that $H^1(E,\cK\otimes \cV)=0$.

By \Cref{le.semist.pos.deg}\cref{semist.pos.deg.pos}, $H^1(E,\cK\otimes \cV)=0$ if $\cK\otimes \cV$ is  semistable of positive degree.
That is what we will prove: first we will show $\cK$ is indecomposable and hence semistable by   \Cref{le.tu}  which will,
by    \Cref{le.tens-ss}, imply that $\cK\otimes \cV$ is semistable, 
then we will show that its slope, and hence its degree, is positive.

To show $\cK$ is indecomposable we write it as a direct sum $\cK = \cK_1 \oplus \cdots \oplus \cK_n$ of non-zero indecomposable
  submodules. Applying the functor ${\mathcal Hom}(-,\cO_E)$ to  \Cref{eq:res} produces an
  exact sequence 
  $$
  0 \to \cU^\vee \to H^0(E,\cU)^* \otimes \cO \to \cK^\vee \to 0.
  $$
Hence $\cK^\vee$ is generated by global sections. It follows that every $\cK_i^\vee$ is also generated 
  by its global sections. Since $\cK_i^\vee$ is indecomposable it is semistable and therefore of positive degree by
  \cref{le.semist.pos.deg}\cref{semist.pos.deg.gl.sec}. It follows that the kernel of the natural map $H^0(E,\cK_i^\vee) \otimes \cO \to \cK_i^\vee$ 
  is non-zero for all $i$.
  The kernel of the natural map $H^0(E,\cK^\vee) \otimes \cO \to \cK^\vee$ is therefore a direct sum of (at least) $n$ non-zero 
  $\cO_E$-submodules. 
  
  That kernel has another description.
  By construction, the right-hand map in \Cref{eq:res}  induces an isomorphism on global sections so it follows from
  the long exact cohomology sequence associated to \Cref{eq:res}   that the sequence 
\begin{equation*}
	\begin{tikzcd}
		0\ar[r] & H^1(E,\cK)\ar[r,"{H^{1}(\varepsilon)}"] & H^1(E,H^0(E,\cU)\otimes\cO)\ar[r] & H^1(E,\cU) = 0
	\end{tikzcd}
\end{equation*}
is exact. Since $H^{1}(\varepsilon)$ is an isomorphism, $H^{0}(\varepsilon^{\vee})$ is also an isomorphism by Serre duality. Thus we obtain a commutative diagram
\begin{equation*}
	\begin{tikzcd}[column sep=large]
		H^{0}(E,(H^0(E,\cU)\otimes\cO)^{\vee})\otimes\cO\ar[d]\ar[r,"H^{0}(\varepsilon^{\vee})\otimes\cO"] & H^{0}(E,\cK^{\vee})\otimes\cO\ar[d] \\
		(H^0(E,\cU)\otimes\cO)^{\vee}\ar[r,"\varepsilon^{\vee}"'] & \cK^{\vee}
	\end{tikzcd}
\end{equation*}
where $H^{0}(\varepsilon^{\vee})\otimes\cO$ and the left vertical morphism are isomorphisms (since $(H^0(E,\cU)\otimes\cO)^{\vee}$ is free).

Since $\cU$ is indecomposable so is $\cU^\vee\cong\ker(\varepsilon^{\vee})$. The kernel of the canonical map $H^0(E,\cK^\vee)  \otimes \cO \to  \cK^\vee$ is therefore indecomposable too. However, that kernel is a direct sum of at least $n$ non-zero $\cO_E$-submodules, so $n=1$;
  i.e., $\cK$ is indecomposable, as   claimed.
  
To complete the proof we show that $\mu(\cK\otimes \cV)> 0.$ Recall that
  \begin{equation*}
    \mu(\cK\otimes \cV) \; = \;  \mu(\cK)+\mu(\cV).
  \end{equation*}
  By the definition of $\cK$ in (\ref{eq:res}), its degree is $-\deg(\cU)$ and its rank is
  \begin{equation*}
    \dim_\Bbbk H^0(E,\cU)-\rk(\cU) \;=\;  \deg(\cU)-\rk(\cU)
  \end{equation*}
  (the equality follows from \Cref{le.semist.pos.deg}). The target
  inequality $\mu(\cK)+\mu(\cV)> 0$ is thus equivalent to
  \begin{equation*}
    \mu(\cV) \; > \; -\, \mu(\cK) \;=\;  \frac{\deg(\cU)}{\deg(\cU)-\rk(\cU)} \;=\;\frac{1}{1-\frac{1}{\mu(\cU)}}.
     \end{equation*}
     The hypothesis in the statement can be written as
     $$
\frac 1{\mu(\cV)} \; < \; 1\,-\, \frac 1{\mu(\cU)} \, ,
$$
  and the both sides are positive. The proof is complete.
\end{proof}

\begin{corollary}
\label{prop.improved.prop.3.5} 
Let $\cU$ and $\cV$ be locally free $\cO_E$-modules generated by their global sections and suppose $\cU$ is semistable of slope $>2$. If
 $0 \to \cA \to  \cV \to \cB \to 0$ is an exact sequence in which $\cA$ and $\cB$ are semistable  locally free $\cO_E$-modules of slope $\ge 2$, then the multiplication map 
$$
H^0(E,\cU) \otimes H^0(E,\cV) \, \longrightarrow \, H^0(E,\cU \otimes  \cV)
$$
is onto.
 \end{corollary}
\begin{proof}
 Write $\cV=\oplus \cV_i$ as the sum of its indecomposable summands. 
 By \cref{lem.slope.summands}, each $\cV_i$ is locally free, semistable of slope $\ge 2$, and is generated by its global sections. Since $\mu(\cU)^{-1}+\mu(\cV_i)^{-1}<1$, the map $H^0(\cU) \otimes H^0(\cV_i) \to 
H^0(\cU \otimes  \cV_i)$ is onto. The conclusion follows from \Cref{le.otimes}. 
\end{proof}

\begin{corollary}
\label{thm.good.U.V}
If  $\cU$ and $\cV$ are semistable locally free $\cO_E$-modules of slope $> 2$, then the multiplication map 
$$
H^0(E,\cU) \otimes H^0(E,\cV) \; \longrightarrow \; H^0(E,\cU \otimes \cV)
$$
is onto.
\end{corollary}
\begin{proof}
By \cref{le.semist.pos.deg}\cref{semist.pos.deg.slope}, $\cU$ and $\cV$ are generated by global sections.
By hypothesis, the inequality in \Cref{eq:ineq}  holds. 
The result now follows from \Cref{le.otimes}.
\end{proof}

\subsection{Remarks}

\subsubsection{}
After proving \cref{prop.improved.prop.3.5,thm.good.U.V} we learned that those results were already known in greater generality:
they are consequences of Theorems 2.1 and 1 in \cite{Butler94}. Nevertheless, for elliptic curves over an algebraically closed field of characteristic zero, 
\cref{le.otimes} is not a consequence of the results in \cite{Butler94} and suggests that the following might be true: 
if $C$ is a smooth projective curve 
of genus $g$ and $\cU$ and $\cV$ are semistable locally free $\cO_C$-modules such that $\mu(\cU)^{-1}+\mu(\cV)^{-1} < g^{-1}$,
then the map $H^0(\cU) \otimes H^0(\cV) \to H^0(\cU \otimes \cV)$ is surjective.

\subsubsection{}
In order for the multiplication map $H^0(\cU) \otimes H^0(\cV) \to H^0(\cU \otimes \cV)$ in \Cref{le.otimes}
  to be onto it is necessary that $\mu(\cU)^{-1}+\mu(\cV)^{-1}$ is $\le 1$: if the multiplication map is onto, then 
    \begin{equation}\label{eq:ddd}
    \deg(\cU\otimes \cV)\; =\; \dim H^0(\cU\otimes \cV) \; \le \; \dim H^0(\cU)\dim H^0(\cV)\;=\; \deg(\cU)\deg(\cV);
  \end{equation}
since $    \deg(\cU\otimes\cV) = \deg(\cU)\rk(\cV)+\deg(\cV)\rk(\cU)$, 
dividing \Cref{eq:ddd} by $\deg(\cU)\deg(\cV)$ yields the inequality $\mu(\cU)^{-1}+\mu(\cV)^{-1}  \le 1$.

\subsubsection{}\label{re.to} 
There is a less elementary proof of the indecomposability of $\cK$ in the proof of \Cref{le.otimes}.
Given a coherent sheaf $\cE$ on $E$, let $T_\cE$ be the endofunctor of the bounded derived category 
${\sf D}^b(E)={\sf D}^b(\coh(E))$ that sends $\cF$ to the cone over
\begin{equation*}
  \mathrm{RHom}(\cE,\cF)\otimes^L \cE\to \cF. 
\end{equation*}
Applying this with $\cE=\cO$, $T_\cO(\cU)=\cK[1]$. But \cite[Prop.~2.10]{st} implies that $T_\cO$ is an autoequivalence 
 so the indecomposability of $\cK[1]$, and hence $\cK$, follows from that of $\cU$.

\subsubsection{}
\cref{le.semist.pos.deg}\cref{semist.pos.deg.slope} with a stronger assumption $\mu(\cF)\geq 2$ can be shown using the following result which might prove useful in other situations.

\begin{lemma}\label{le.filt}
	Every semistable locally free $\cO_E$-module $\cF$ has a filtration by invertible $\cO_E$-modules of degrees 
	$\geq \lfloor\mu(\cF)\rfloor$. 
\end{lemma}
\begin{proof}
	We can assume that $\cF$ is indecomposable. Let $r=\rk(\cF)$ and $d=\deg(\cF)$. We will prove the result by the induction on $r$. 
	The result is certainly true when  $r=1$.
	
  If $\mu(\cF)=0$, then the assertion follows from \cite[Thm.~5]{Atiyah}.
  
  If $0<\mu(\cF)<1$ (i.e.\ $0<d<r$), then \cite[Lem.~15]{Atiyah} implies
  that $\cF$ contains a rank-$d$ free subsheaf $\cI$ of $\cF$ such
  that $\cF/\cI$ is an indecomposable locally free sheaf of rank $r-d$ and degree
  $d$. The induction hypothesis shows that $\cF/\cI$ has a filtration by invertible sheaves of degree $\geq 0$. Thus $\cF$ also has a filtration by invertible sheaves of degree $\geq 0=\lfloor\mu(\cF)\rfloor$.
  
  If $\mu(\cF)<0$ or $1\leq\mu(\cF)$, take an invertible sheaf $\cM$ of degree $\lfloor\mu(\cF)\rfloor$. Then
  \begin{equation*}
  	\mu(\cF\otimes\cM^{-1})\;= \; \mu(\cF)-\mu(\cM) \; =\;  \mu(\cF)-\lfloor\mu(\cF)\rfloor.
  \end{equation*}
  The former two cases shows that $\cF\otimes\cM^{-1}$ admits a filtration by invertible sheaves of degrees $\geq 0$. 
  Tensoring $\cM$  with the filtration gives the desired filtration of $\cF$. This
  completes the induction.
  \end{proof}
  
  Let $\cF$ be a semistable locally free $\cO_E$-module of slope $\ge 2$.
By \cref{le.filt}, $\cF$ has a filtration
	\begin{equation*}
    	0=\cZ_{0}\subset\cZ_{1}\subset\cdots\subset\cZ_{r}=\cF
	\end{equation*}
	in which each $\cF_{i}:=\cZ_{i}/\cZ_{i-1}$ is an invertible $\cO_{E}$-module of degree $\geq 2$. All $H^{1}(\cF_{i})$ vanish, 
	so an induction argument shows that all $H^{1}(\cZ_{i})$ vanish. Thus we obtain a commutative diagram
	\begin{equation}\label{eq.diag.epi}
		\begin{tikzcd}
			0\ar[r] & H^{0}(\cZ_{i-1})\otimes\cO\ar[d]\ar[r] & H^{0}(\cZ_{i})\otimes\cO\ar[d]\ar[r] & H^{0}(\cF_{i})\otimes\cO\ar[d]\ar[r] & 0 \\
			0\ar[r] & \cZ_{i-1}\ar[r] & \cZ_{i}\ar[r] & \cF_{i}\ar[r] & 0
		\end{tikzcd}
	\end{equation}
	with exact rows. Since invertible sheaves of degree $\geq 2$ are generated by their global sections \cite[Cor.~IV.3.2]{Hart}, 
	it is shown inductively that all $\cZ_{i}$ are also generated by global sections using \cref{eq.diag.epi}. In particular, $\cF$ is generated by global sections.

\section{Twisted homogeneous coordinate rings of the form $B(S^gE,\s,\cL)$}
\label{se.S^gE} 

Let $\Sigma_d$ denote the symmetric group on $d$ letters. Let $\Sigma_d$ act on $E^d$ by having
 the transposition $(i,j)$ interchange the $i^{\th}$ and $j^{\th}$ coordinates, $z_i$ and $z_j$, of a point $(z_1,\ldots,z_d)\in E^d$. 
 The $d^{\th}$ symmetric power, $S^dE$, is defined to be the quotient variety $E^d/\Sigma_d$ with respect to this action.
We write 
$$
(\!(z_1,\ldots,z_d)\!)
$$
 for the image of $(z_1,\ldots,z_d)$ in $S^dE$.\footnote{Sometimes $X_{n/k}$ is 
 isomorphic to a symmetric power of $E$; under a careless identification between the two the morphism  
 $\Phi_{n/k}:E^g \to X_{n/k}$ might not correspond to  the natural map $E^g \to S^gE$; this is irrelevant in this section and the next but 
becomes relevant in \S\ref{sect.applic}.} As is well-known, the Abel-Jacobi map, i.e., the addition map 
$$
\pi \;=\; {\sf sum}:S^dE\to E, \qquad {\sf sum}(\!(z_1,\ldots,z_d)\!) = z_1+\cdots+z_d,
$$
presents $S^dE$, as  a $\PP^{d-1}$-bundle over $E$. We will say more about this in \S\ref{se.sym}.

For this reason we start this section with results about projective space bundles on $E$.

\subsection{Projective space bundles $\PP(\cE)$ on an elliptic curve $E$}
\label{subsec.proj}
We recall some standard results and notation for projective space bundles, for the most part following the material in  \cite[pp.~160--171]{Hart}.

We adopt the following notation in this subsection:
\begin{itemize}
\item
$\cE$ is a locally free $\cO_E$-module of rank $d$;
  \item 
  $S^m\cE$ is the $m^{\th}$ symmetric power of $\cE$ when $m \ge 0$;
  \item 
  $S(\cE)=\cO_E \oplus \cE \oplus S^2\cE \oplus \cdots$ is the symmetric algebra on $\cE$;
  \item 
  $X=\PP(\cE) :={\bf Proj}(S(\cE))$ is the associated $\PP^{d-1}$-bundle on $E$;
  \item
  $\cO_X(1)$, or simply $\cO(1)$, be the tautological $\cO_X$-module associated to $S(\cE)$, i.e., $\pi_*(\cO_X(1))=\cE$;
  \item
  $\pi:\PP(\cE) \to E$ is the structure morphism;
\end{itemize}
We call $X$
the {\sf projectivization} of $\cE$. 

It follows from the definition of $\cO_X(1)$ that there are canonical isomorphisms 
$$
\pi_*(\cO_X(m))  \; \cong \;  S^m\cE
$$ 
for all $m \ge 0$, cf., \cite[Prop.~II.7.11]{Hart}. We will make frequent use of this fact without further comment.

We will make frequent use of the following observation.

\begin{lemma}
\label{split-epi-0}
Assume ${\rm char}(\Bbbk)=0$. For all integers $m,n \ge 1$, the canonical maps $\cE^{\otimes n} \to S^n\cE$ and   
$S^m\cE\otimes S^n\cE\to S^{m+n}\cE$ are split epimorphisms. 
\end{lemma}
\begin{proof}
The map $\cE^{\otimes n} \to S^n\cE$ splits because  ${\rm char}(\Bbbk)=0$.
The composition  
  $$
  \cE^{\otimes(m+n)}\;=\; \cE^{\otimes m}\otimes \cE^{\otimes n} \; \longrightarrow \; S^m\cE\otimes S^n\cE  \; \longrightarrow \; S^{m+n}\cE
  $$
  is therefore a split epimorphism for all $m,n \ge 0$. 
  The map $S^m\cE\otimes S^n\cE\to S^{m+n}\cE$ is therefore a split epimorphism. 
\end{proof}

\subsubsection{Remark}
\label{re.splt}
When ${\rm char}(\Bbbk)=0$, the splitting of the map $S^m\cE\otimes S^n\cE\to S^{m+n}\cE$ can be defined universally by splitting the
  symmetrization morphism $S^m\otimes S^n\to S^{m+n}$ in the category of {\it polynomial endofunctors} on the category 
  $\coh(E)$ of coherent $\cO_E$-modules 
   (see, e.g., \cite[\S2.2]{sam-snow}   for a reminder on these).

  Concretely, we write $T_{m,n}:S^{m+n}\to S^m\otimes S^n$ for the
  natural transformation between polynomial functors which for
  symmetric powers of a vector space $V$ reads
  \begin{equation*}
    S^{m+n}V\ni v_1\cdots v_{m+n}\mapsto \frac {m!n!}{(m+n)!}\sum v_{a_1}\cdots v_{a_m}\otimes v_{b_1}\cdots v_{b_n}\in S^mV\otimes S^nV,
  \end{equation*}
  where the sum is over all decompositions of the set
  $\{1,\cdots,m+n\}$ as a disjoint union of $\{a_i\}$ and $\{b_j\}$.

\subsubsection{}
\label{sect.replace.cE}
Sometimes it is convenient to replace $\cE$ by another invertible $\cO_E$-module $\cE'$ such that $\PP(\cE) \cong \PP(\cE')$ as bundles over $E$. 
Exercise II.7.9(b) at \cite[p.~170]{Hart} and Lemma II.7.9 at \cite[p.~161]{Hart} address this matter: if $\cE$ and $\cE'$ are locally free 
$\cO_E$-modules and $\pi :\PP(\cE) \to E$ and $\pi' :\PP(\cE') \to E$ are the structure morphisms, then there is 
an isomorphism 
$$
\varphi:X=\PP(\cE)  \, \longrightarrow \, X'=\PP(\cE')
$$ 
such that $\pi=\pi' \varphi$  if and only if $\cE' \cong \cE \otimes \cL$ for some invertible $\cO_E$-module $\cL$. When this happens,
$$
\cO_{X}(1) \; \cong \; \varphi^*\cO_{X'}(1) \otimes \pi^*\cL.
$$
Replacing $\cE$ by $\cE'$ allows one to assume that $0 \le \deg(\cE) \le \rank(\cE)-1$; i.e., 
given any $\cE$ there is an invertible $\cO_E$-module $\cL$  such that $0 \le \deg(\cE\otimes \cL) \le \rank(\cE\otimes \cL)-1$.

\subsubsection{The N\'eron-Severi and Picard groups of $S^dE$}
\label{sssect.NS.S^gE}
\label{ssect.NS}
\label{se.rel.pic}
There is a split exact sequence 
\begin{equation*}
 0 \; \longrightarrow \; \pic(E)  \; \stackrel{\pi^*}{\longrightarrow} \;  \pic(X)  \; \longrightarrow \; \ZZ  \; \longrightarrow \;  0
\end{equation*}
with a splitting $\ZZ \to \pic(X)$ given by $1 \mapsto [\cO_X(1)]$ (see \cite[Exer.~II.7.9(a)]{Hart} and \cite[A11, p.~429]{Hart},
for example). 

The image  in  $\ZZ$ of $[\cL]\in \pic(X)$ is called the {\sf degree} of $\cL$ and is denoted by $\deg(\cL)$. 

The N\'eron-Severi group, ${\rm NS}(X)$, is isomorphic to $\ZZ\oplus \ZZ$ with basis 
$D:=[\cO_X(1)]$ and $F=[F_u]$ where $F_u:=\pi^{-1}(u)$ is the fiber over an arbitrary point $u \in E$. It is well known that
$$
F\cdot F\;=\; 0, \qquad F \cdot D^{d-1}\;=\; 1, \qquad D^d \;=\; \deg(\cE);
$$
see, e.g., \cite[Prop.~1.1(1)]{Gushel}. Let $D_u:=$ the image of $\{u\} \times E^{d-1}$ in $S^dE$; 
i.e., $D_u$ consists of the points $(\!(u,x_2,\ldots,x_d)\!)$ and is isomorphic to $S^{d-1}E$. 
When $X=S^dE$, $D=[\cO_X(1)]=[D_u]$.\footnote{The remarks after \cite[Lem.~1.3]{CaCi93} provide a nice 
proof of this equality that uses the 
Poincar\'e bundle on $E$. 
Here is another proof. Clearly $\cO_X(D_u)$ restricts to $\cO_{F_v}(1)$ on every fiber $F_{v}$.
The difference $\cR:=\cO_X(D_u)^{-1}\otimes \cO_X(1)$ is therefore trivial on each fiber and hence a pullback of a divisor on $E$. 
It follows that in $\NS(X)$, $D=[D_u]+tF$ for some integer $t$. To show that  $t=0$ it suffices to show that $D_u^d=1$. 
Intersections behave well in families so we can compute $D_u^d$ by taking the intersection of $D_{u_i}$ for $d$ distinct points $u_i$; 
this intersection is obviously  a singleton (and the intersections are transverse) so we conclude that $D_u^d=1$. Thus $t=0$ and 
$D=[D_u]$.
}

\begin{proposition}
\label{cor.picX}
If $\cL$ is an invertible $\cO_X$-module, there is an invertible $\cO_E$-module $\cL'$ such that
\begin{align*}
\cL & \; \cong \; \cO_X(a) \otimes \pi^*\cL' \qquad \text{and}
\\
[\cL] & \;=\;  aD+ (\deg(\cL'))F
\end{align*}
where $a=\deg(\cL)$.
\end{proposition}

\begin{proposition}
[Gushel \cite{Gushel}]
\label{prop.gushel}
Let $\cE$ be an indecomposable locally free $\cO_E$-module of rank $d$ such that $0 \le \deg(\cE)\le d-1$.\footnote{As remarked in \cref{sect.replace.cE}, if $\cE$ is any indecomposable locally free $\cO_E$-module of rank $d$, there is an indecomposable locally free $\cO_E$-module $\cE'$ such that $X=\PP(\cE)$ is isomorphic to $\PP(\cE')$.}
Let $\cL$ be an invertible $\cO_X$-module. Suppose $[\cL]=aD+bF$.
Then $\cL$ is
\begin{enumerate}
  \item 
  generated by its global sections if $a \ge 0$ and $b \ge 2$;
  \item 
  ample if $a \ge 1$ and $b \ge 1$;
  \item 
  very ample if  $a \ge 1$ and $b \ge 3$.
\end{enumerate}  
\end{proposition}
\begin{proof}
These statements are weak versions of Proposition 1.1(iv), Proposition 3.3(i), and Theorem 4.3 in \cite{Gushel}.
\end{proof}

\subsubsection{}

In order to analyze the push-forward $\pi_*\cL$ of an invertible $\cO_X$-module $\cL$, we  adapt 
\cite[Lems.~V.2.1 and V.2.4]{Hart} to the present setting.

\begin{lemma}\label{le.1}
  Let $\cL$ be an invertible $\cO_X$-module with $[\cL]=aD+bF\in {\rm NS}(X)$. If $a \ge 0$, 
  then 
  \begin{enumerate}
  \item\label{psb.pushforward.loc.free}
  $\pi_*\cL$ is a locally free $\cO_E$-module of  rank $\binom{a+d-1}{d-1}$;
  \item\label{psb.pushforward.der}
  $R^i\pi_*(\cL)=0$ for all  $i\ge 1$;
  \item\label{psb.pushforward.cohom}
  $    H^i(X,\cL) \cong H^i(E,\pi_*\cL)$  for all  $i\ge 0$.
\end{enumerate}
\end{lemma}
\begin{proof}
For all $u \in E$, the restriction $\cL_u$ of $\cL$ to $F_u$ is  isomorphic to $\cO_{\PP^{d-1}}(a)$.  

\cref{psb.pushforward.loc.free}
  The dimension of $H^0(F_u,\cL_u)$ is therefore $\binom{a+d-1}{d-1}$. Since this holds  for all $u \in E$  the conclusion follows
  from Grauert's Theorem \cite[Cor.~III.12.9]{Hart}, just as in the proof of \cite[Lem.~V.2.1]{Hart}.
 
\cref{psb.pushforward.der}
Since $a\ge 0$, $H^i(F_u,\cL_u)=0$ for all $i\ge 1$ and $u\in E$. Thus Grauert's Theorem shows that $R^i\pi_*(\cL) \cong H^i(F_u,\cL_u)$=0.
  
\cref{psb.pushforward.cohom}
 By \cref{psb.pushforward.der}, the spectral sequence
 $
    H^p(E,R^q\pi_*(\cL))\Rightarrow H^{p+q}(X,\cL)
 $
collapses. Thus, as in  \cite[Lem.~V.2.4]{Hart},   $ H^i(E,\pi_*\cL)     \cong  H^i(X,\cL)$  for all  $i\ge 0$.
\end{proof}

\subsection{Multiplication of sections of invertible sheaves on $\PP(\cE)$}\label{se.sct}

We continue to assume that  $X=\PP(\cE)$ and keep the notations in \S\ref{subsec.proj}.

\begin{lemma}
\label{split-epi}
Assume ${\rm char}(\Bbbk)=0$. 
Let $\cL_1 = \cO_X(a_1) \otimes \pi^*\cL_1' $  and $\cL_2 = \cO_X(a_2) \otimes \pi^*\cL_2'$ be invertible $\cO_X$-modules. If $a_1$ and $a_2$ are $\ge 0$, then 
 the natural map 
\begin{equation}
\label{eq.split-epi}
\pi_*\cL_1 \otimes \pi_*\cL_2 \, \longrightarrow \, \pi_*(\cL_1 \otimes \cL_2)
\end{equation}
is a split epimorphism.
\end{lemma}
\begin{proof}
By the Projection Formula \cite[p.~124]{Hart},
  \begin{equation*}
    \pi_*(\cL_i) \;=\; \pi_*\big(\cO(a_i)\otimes\pi^*\cL_i'\big) \;\cong \;  \pi_*(\cO(a_i))\otimes \cL_i'   \;\cong \;  S^{a_i}\cE  \otimes \cL_i' .
  \end{equation*}
Similarly,
$    \pi_*(\cL_1\otimes \cL_2)    \; \cong   \;   S^{a_1+a_2}\cE  \otimes \cL_1'  \otimes \cL_2'$.
    The morphism in (\ref{eq.split-epi}) can therefore  be written as 
  \begin{equation}
  \label{eq.arrow}
  S^{a_1} \cE  \otimes \cL_1' \otimes    S^{a_2} \cE \otimes \cL_2'   \, \longrightarrow   S^{a_1+a_2}\cE \otimes \cL_1'  \otimes \cL_2' .
  \end{equation}
It follows from \Cref{split-epi-0} that this is a  split epimorphism. 
\end{proof}

\begin{proposition}\label{pr.onto}\label{le.aux}
Assume ${\rm char}(\Bbbk)=0$. Let  $\cE$ be a semistable locally free $\cO_E$-module of positive degree and let $X=\PP(\cE)$. 
Let $\cL$ and $\cF$ be invertible $\cO_X$-modules whose
classes in ${\rm NS}(X)$ are $aD+bF$ and $sD+tF$, respectively.
  If $a,s\ge 1$ and $b,t\ge 2$, then  the multiplication maps
  \begin{enumerate}
  \item\label{psb.pf.mult.onto}
$H^0(E,\pi_*\cL)\otimes H^0(E,\pi_*\cF)  \, \longrightarrow \,  H^0(E,\pi_*\cL\otimes \pi_*\cF)$ and
  \item\label{psb.mult.onto}
$H^0(X,\cL)\otimes H^0(X,\cF)   \, \longrightarrow \,  H^0(X,\cL\otimes \cF)$
\end{enumerate}
are onto.
\end{proposition}
\begin{proof}
\cref{psb.pf.mult.onto}
   Let $\cU =\pi_*\cL$ and $\cV= \pi_*\cF$.
   
  By \cref{cor.picX}, there are invertible $\cO_E$-modules $\cL'$ and $\cF'$, of degrees $b$ and $t$ respectively, such that 
$ \cL\cong \cO_X(a)\otimes \pi^*(\cL')$ and $ \cF\cong \cO_X(s)\otimes \pi^*(\cF')$.
By the Projection Formula, $ \cU \cong \; \pi_*(\cO_X(a))\otimes \cL'$ and $ \cV \cong  \pi_*(\cO_X(s))\otimes \cF'$.

Since $\cE$ is semistable,   \Cref{le.tens} tells us that $\cE^{\otimes a}$ is semistable too;  its direct summand $S^{a}\cE$ 
is therefore semistable too. Similarly,  $S^s\cE$ is semistable.  
Since $\cL'$ and $\cF'$ are semistable by   \cref{le.indec}, the tensor products $\cU$ and $\cV$ are also semistable by  \Cref{le.tens}.

Since $\cE$ has positive degree, by assumption, and  $a,s\ge 1$, 
the summands $S^a\cE$ and $S^s\cE$  of the semistable sheaves $\cE^{\otimes a}$ and  $\cE^{\otimes s}$ have positive degree too. Hence
    \begin{equation*}
    \mu(\cU) \; = \; \mu(\pi_*(\cO_X(a))) + \mu(\cL') \; = \; \mu(S^a\cE) + \mu(\cL')  \; > \; \mu(\cL') \; \ge \; 2.
  \end{equation*}
Similarly, $\mu(\cV)>2$. The result now follows from \Cref{thm.good.U.V}.

\cref{psb.mult.onto}
  The canonical map $    H^0(\pi_*\cL)\otimes H^0(\pi_*\cF)\to H^0( \pi_*(\cL \otimes \cF))$ factors as
  $$
  H^0(E,\pi_*\cL)\otimes H^0(E,\pi_*\cF) \; \longrightarrow \; H^0(E,\pi_*\cL\otimes \pi_*\cF) \; \longrightarrow \;  H^0(E, \pi_*(\cL \otimes \cF)).
  $$
    We have just shown that the left-most map in this composition is onto; the other is also onto because the canonical map  
    $\pi_*\cL \otimes \pi_*\cF \to  \pi_*(\cL \otimes \cF)$ is a split epimorphism by \Cref{split-epi}. The composition is therefore onto.
The result now follows from the fact that there are functorial   isomorphisms $H^0(X,\cL)\cong H^0(E,\pi_*\cL)$,
$H^0(X,\cF)\cong H^0(E,\pi_*\cF)$, and $H^0(X,\cL\otimes \cF)\cong H^0(E,\pi_*(\cL\otimes \cF))$.  
\end{proof}

\subsection{Symmetric powers of $E$}\label{se.sym}

Let $X=S^dE$. 

Up to tensoring with a degree-0 invertible sheaf, there is a unique indecomposable locally free sheaf $\cE$ of rank $d$ and degree $1$ 
on $E$ such that 
$S^dE \cong \PP(\cE)$ \cite[p.~451]{Atiyah}. 
By \cite[Thm.~5, p.~432]{Atiyah}, one can construct such an $\cE$ iteratively: let $\cE_1$ be any invertible $\cO_E$-module of degree one
and for $r \ge 2$ let $\cE_{r}$ be the ``unique" non-trivial extension
 $0 \to \cO_E \to \cE_{r} \to \cE_{r-1} \to 0$.   Thus $X=S^dE=\PP(\cE_d)$ and $\pi_*(\cO_X(1)) \cong \cE_d$.
  Since $\cE_d$ is indecomposable it is 
 semistable (in fact, stable because its degree is 1) of positive degree so \cref{le.aux} applies. 
 
 Whenever we view $S^dE$ as a projective bundle $\PP(\cE)$ we will assume $\cE$ is $\cE_d$. 
 
 See \cite[p.~451]{Atiyah},  \cite{CaCi93}, and \cite{Polizzi} for more information about $S^dE$ as a projective space bundle.

\begin{theorem}
\label{cor.ab} 
Assume ${\rm char}(\Bbbk)=0$.
Let $\s:S^dE\to S^dE$ be a translation automorphism.
Let $\cL$ be an invertible sheaf on $S^dE$ whose N\'eron-Severi class is $[\cL]=aD+bF$. If $a\ge 1$   and $b\ge 2$,
then    $B(S^dE,\sigma,\cL)$ is generated  in degree one.
\end{theorem}
\begin{proof}
By definition, the degree-$n$ homogeneous component of $B(S^dE,\sigma,\cL)$ is 
  \begin{equation*}
  B_n \;=\;   H^0(S^dE,\cL\otimes\cdots\otimes (\sigma^*)^{n-1}\cL).
  \end{equation*}
The surjectivity of the multiplication map $B_1\otimes B_n\to B_{n+1}$ therefore follows from 
  \Cref{pr.onto} applied to $\cL$ and  $   \cF=\sigma^*\cL\otimes\cdots\otimes (\sigma^*)^n \cL$. 
\end{proof}

\section{Relations for $B(S^gE,\sigma,\cL$)}\label{se.rel}

The main result in this section,  \cref{pr.rel}, shows that rather mild hypotheses imply that 
the relations for $B(S^gE,\s,\cL)$ are generated in degrees 2 and 3.

 \subsection{Preparations}
The following hypotheses and notation apply throughout this section:
\begin{itemize}
  \item 
 $\Bbbk$ is an algebraically closed field of characteristic zero;
  \item
  $\pi:S^gE \to E$ is the map $\pi(\!(z_1,\ldots, z_g)\!) \;=\; z_1+\cdots + z_g$;
  \item 
  $\s$ is an arbitrary translation automorphism of $S^gE$;
  \item
   $\cL$ is an invertible $\cO_{S^gE}$-module  such that   $[\cL]=aD+bF$ where $a \ge 1$ and $b \ge 2$;
     \item 
  $X=S^gE$, $g \ge 2$.
\end{itemize}
By \cref{prop.gushel}, $\cL$ is ample and is generated by its global sections.

By \Cref{cor.ab}, $B:=B(S^gE,\s,\cL)$ is generated  as a $\Bbbk$-algebra  by its degree-one component, $B_1$; i.e., 
the canonical $\Bbbk$-algebra homomorphism  $\varphi:T(B_1) \to B$  from the tensor algebra on $B_1$ is surjective.

 \subsubsection{}
To prove that the ideal $J:=\ker(\varphi)$ is generated by $J_2+J_3$
 it suffices, by  \cref{lem.SS92},  to show that the canonical map 
$$
H^0(S^gE,\cL) \otimes R(\cM_m,\cN_m) \, \longrightarrow \, R(\cL \otimes \cM_m, \cN_m)
$$
is surjective  for all $m \ge 2$ where $\cM_m$ 
and $\cN_m$ are the sheaves defined at the beginning of \S\ref{sect.ker.mult}. 
We will prove more: if $\cL$, $\cM=\cM_m$, and $\cN=\cN_m$, have the properties in \cref{cv.lmn} below, then 
the canonical map 
\begin{equation}
\label{eq.R.LMN}
H^0(S^gE,\cL) \otimes R(\cM,\cN) \, \longrightarrow \, R(\cL \otimes \cM, \cN)
\end{equation}
is onto.

\begin{convention}\label{cv.lmn}
Let $E$ be an elliptic curve and let $X=S^gE$.
Let  $\cL$, $\cM$, and $\cN$, be invertible $\cO_X$-modules with N\'eron-Severi classes
\begin{align*}
[\cL]  &\;=\;  aD+bF \quad \text{where $a \ge 1$ and $b \ge 2$,}
\\
[\cM] & \;=\; pD+qF \quad \text{where $p\ge 2$ and $q\ge 4$.}
\\
[\cN] & \;=\; sD+tF \quad \text{where $s \ge 1$ and $t \ge 2$,}
\end{align*}
\end{convention}

\subsection{Surjectivity of multiplication maps}

By \cref{prop.gushel}, $\cN$ is generated by its global sections. Let 
 \begin{equation}
 \label{eq.K.NO.N}
  0 \to \cK \to H^0(S^gE,\cN) \otimes \cO_{S^gE} \to \cN \to 0
  \end{equation}
  be the associated exact sequence.
Let $\cG= \cM\otimes \cK$. Since $\cK$ is a locally free $\cO_{S^gE}$-module so is $\cG$. 
In this case, the argument in \cref{lem.B.deg.relns} showed that the map in (\ref{eq.R.LMN}) is onto if and only if the map 
\begin{equation}\label{eq:lg-main}
  H^0(S^gE,\cL)\otimes H^0(S^gE,\cG)\to H^0(S^gE,\cL\otimes \cG).
\end{equation}
is onto.

The $\cO_{S^gE}$-module $\cL$ satisfies the hypotheses and therefore the conclusions of \Cref{le.1}. Although $\cG$ and $\cL\otimes \cG$ need not be invertible, they have similar properties.

\begin{lemma}\label{le.Gr.G}\leavevmode
  \begin{enumerate}
  \item 
  $\pi_*\cG$ and $\pi_*(\cL\otimes\cG)$ are locally free $\cO_E$-modules;
  \item 
  $R^i\pi_*(\cG)=R^i\pi_*(\cL\otimes\cG)=0$ for all  $i\ge 1$;
  \item\label{le.Gr.G.hi} 
  $H^i(S^gE,\cG) \cong H^i(E,\pi_*\cG)$  and   $H^i(S^gE,\cL \otimes \cG) \cong H^i(E,\pi_*(\cL \otimes \cG))$ for all  $i\ge 0$.
  \end{enumerate}
\end{lemma}
\begin{proof}
	We prove the lemma for $\cG$. The other case is  similar. Since $\cM$ and $\cM\otimes\cN$  satisfy the hypotheses of \Cref{le.1}, $R^{i}\pi_{*}\cM=R^{i}\pi_{*} (\cM\otimes\cN)=0 $ for all $i\geq 1$. Hence the short exact sequence \Cref{eq.K.NO.N}  
	tensored with $\cM$ produces the exact sequence
	\begin{equation}
		\begin{tikzcd}
			0\ar[r] & \pi_{*}\cG\ar[r] & \pi_{*}(\cM\otimes H^{0}(S^gE,\cN))\ar[r,"f"] & \pi_{*}(\cM\otimes\cN)\ar[r] & R^{1}\pi_{*}(\cG)\ar[r] & 0
		\end{tikzcd}
	\end{equation}
	and shows $R^{i}\pi_{*}\cG=0$ for all $i\geq 2$. Applying $H^{0}(S^gE,-)$, the morphism $f$ induces a map
	\begin{equation*}
		H^{0}(S^gE,\cM)\otimes H^{0}(S^gE,\cN)\to H^{0}(S^gE,\cM\otimes\cN)
	\end{equation*}
	which is surjective by \cref{pr.onto}. Since $\pi_{*}(\cM\otimes\cN)$ is semistable and has slope $\geq 2$ as shown in the proof of \cref{le.aux}, it is generated by global sections by \cref{le.semist.pos.deg}\cref{semist.pos.deg.slope}. Therefore the morphism $f$ is surjective. This implies $R^{1}\pi_{*}(\cG)=0$.
	
Since $\pi_{*}(\cM\otimes H^{0}(\cN))$ is a locally free $\cO_E$-module, so is 	$\pi_{*}\cG$. The assertion \cref{le.Gr.G.hi} follows in the same way as \Cref{le.1}.
\end{proof}

\begin{lemma}
\label{lem.using.G}
The map  in \Cref{eq:lg-main} is surjective if and only if the map 
\begin{equation}\label{eq:lg-pi_*main}
H^0(E,\pi_*\cL) \otimes H^0(E,\pi_*\cG) \to H^0(E,\pi_*(\cL \otimes \cG))
\end{equation}
is surjective. 
\end{lemma}
\begin{proof}
This is an immediate consequence of  \cref{le.Gr.G}.
\end{proof}

In what follows, we use \cref{le.Gr.G}  without further comment.

By \Cref{split-epi}, the map $\pi_*\cL_{1}\otimes \pi_*\cL_{2}  \to \pi_*(\cL_{1}\otimes \cL_{2})$ is a split
epimorphism for suitable invertible $\cO_{S^gE}$-modules $\cL_i$. The next result is analogous.

\begin{lemma}\label{le.split}
  Under \Cref{cv.lmn}, the canonical map
  \begin{equation*}
    \pi_*\cL\otimes \pi_*\cG\to \pi_*(\cL\otimes \cG) 
  \end{equation*}
  is a split epimorphism and therefore induces a surjection on
  global sections over $E$.
\end{lemma}
\begin{proof}
There are invertible $\cO_E$-modules $\cL'$, $\cM'$, and $\cN'$, of degrees $b\ge 2$, $q\ge 4$, and $t\ge 2$, respectively,
such that
  \begin{equation*}
    \cL\cong \cO(a)\otimes \pi^*(\cL'),
    \qquad  \cM\cong \cO(p)\otimes \pi^*(\cM'),
    \quad \hbox{and} \quad \cN\cong \cO(s)\otimes \pi^*(\cN').
  \end{equation*}
Since $0 \to \cG \to \cM \otimes H^0(S^gE,\cN) \to \cM \otimes \cN \to 0$ is exact by definition of $\cG$, 
$\pi_*\cG$ is the kernel of    
  \begin{equation}
  \label{eq:pig}
  \pi_*\cO(p)\otimes \cM' \otimes  H^0(E,\pi_*\cO(s)\otimes \cN') \; \longrightarrow \;  \pi_*\cO(p+s)\otimes \cM'\otimes \cN'.
  \end{equation}
Similarly, $\pi_*(\cL\otimes \cG)$ is the kernel of
  \begin{equation}
  \label{eq:pilg}
  \pi_*\cO(a+p)  \otimes H^0(E,\pi_*\cO(s)\otimes \cN')  \; \longrightarrow \;  \pi_*\cO(a+p+s)\otimes \cN'
  \end{equation}
  tensored with $\cL'\otimes \cM'$. Tensoring \Cref{eq:pig} with
  $\pi_*\cL$ produces the tensor product of
  \begin{equation}
  \label{eq:piall}
 \pi_*\cO(a)\otimes \pi_*\cO(p)   \otimes  H^0(E,\pi_*\cO(s)\otimes \cN') \; \longrightarrow \;  \pi_*\cO(a)\otimes \pi_*\cO(p+s)\otimes \cN'
  \end{equation}
  with $\cL'\otimes \cM'$; in other words, tensoring \Cref{eq:pig} with   $\pi_*\cL$ produces 
    \begin{equation}
    \label{eq:pifinal}
 \pi_*\cL\otimes \pi_*\cM   \otimes  H^0(E,\pi_*\cN) \; \longrightarrow \;  \pi_*\cL \otimes \pi_*(\cM \otimes \cN).
  \end{equation}

  The symmetrization map
  $\pi_*\cO(a)\otimes \pi_*\cO(p)\to \pi_*\cO(a+p)$ between the left-hand terms of \Cref{eq:piall,eq:pilg} splits compatibly with the symmetrization map between the right-hand terms. Indeed, this amounts to noting that the diagram
\begin{equation*}
  \begin{tikzpicture}[auto,baseline=(current  bounding  box.center)]
    \path[anchor=base] (0,0) node (111) {$S^a\otimes S^p\otimes S^s$} +(8,0) node (3) {$S^{a+p+s}$} +(4,.5) node (12) {$S^a\otimes S^{p+s}$} +(4,-.5) node (21) {$S^{a+p}\otimes S^s$};
    \draw[->] (12) to [bend right=6] node[auto,pos=.5,swap] {$\scriptstyle \mathrm{id}\otimes T_{p,s}$} (111);
    \draw[->] (21) to [bend left=6] node[auto,pos=.5] {$\scriptstyle T_{a,p}\otimes\mathrm{id}$} (111);
    \draw[->] (3) to [bend right=6] node[auto,pos=.5,swap] {$\scriptstyle T_{a,p+s}$} (12);
    \draw[->] (3) to [bend left=6] node[auto,pos=.5] {$\scriptstyle T_{a+p,s}$} (21);        
  \end{tikzpicture}
\end{equation*}
of functors defined as in \Cref{re.splt} commutes. 
When applied to a finite dimensional vector space $V$ the diagram of functors expresses the coassociativity of the shuffle comultiplication on the symmetric algebra $SV$ (with this bialgebra structure $SV$ is the graded dual to the universal cocommutative bialgebra $B(V)$ defined in \cite[\S\S 12.2 and 12.3]{Swe69}). 

In conclusion \Cref{eq:pilg} can be realized as a direct summand of \Cref{eq:piall}. The kernel of \Cref{eq:pilg} tensored with $\cL'\otimes \cM'$, namely $\pi_*(\cL \otimes \cG)$, is therefore a direct summand of the kernel of \Cref{eq:piall} tensored with $\cL'\otimes \cM'$, and hence a direct summand of the kernel of \Cref{eq:pifinal}, which is $ \pi_*\cL\otimes \pi_*\cG$.
\end{proof}

By \cref{lem.using.G},  \Cref{eq:lg-main} is surjective if and only if \Cref{eq:lg-pi_*main} is. 
Now by \Cref{le.split},  the surjectivity of \Cref{eq:lg-pi_*main} is equivalent to the surjectivity of
\begin{equation*}
  H^0(E,\pi_*\cL)\otimes H^0(E,\pi_*\cG)   \; \longrightarrow \;  H^0(E,\pi_*\cL\otimes \pi_*\cG).
\end{equation*}
This has the same flavor as \Cref{le.otimes}, and we will use
that result to obtain the conclusion.

\begin{lemma}\label{le.l}
The locally free $\cO_E$-module  $\pi_*\cL$ is isomorphic to $\pi_*\cO(a)\otimes \cL'$, is semistable, 
generated by its   global sections, and has slope $a\mu(\cE)+b$.
\end{lemma}
\begin{proof}
  We have argued along the same lines above, several times:
  $\pi_*\cO_{S^gE}(1)\cong\cE$ is a semistable $\cO_E$-module so the
  summand $\pi_*\cO(a)$ of its tensor power $\pi_*\cO(1)^{\otimes a}$ is semistable of slope $a\mu(\cE)$.
   \Cref{le.tens-ss} now tells us that $\pi_*\cO(a)\otimes \cL'$ is also  semistable.

  The claim of global generation now follows
  from \cref{prop.gushel}.

  The slope of $\pi_*\cO(a)$ is $a\mu(\cE)$, so
  \begin{equation*}
 \mu(\pi_*\cL) \;= \;  \mu(\pi_*\cO(a)\otimes \cL') \;=\;    \mu(\pi_*\cO(a))+\mu(\cL') \; =\;  a\mu(\cE)+b,
  \end{equation*}
  as claimed.
\end{proof}

We now consider $\pi_*\cG$. As noted in the proof of \Cref{le.split}, $\pi_*\cG$ is the kernel of the composition
\begin{equation}\label{eq.comp}
  H^0(\pi_*\cO(s)\otimes \cN')\otimes \pi_*\cO(p)\otimes \cM' \; \to\;  \pi_*\cO(p)\otimes \pi_*\cO(s)\otimes \cM'\otimes \cN' 
  \; \to \; \pi_*\cO(p+s)\otimes \cM'\otimes \cN'
\end{equation}
where the first morphism is the canonical map
\begin{equation}\label{eq.can.epi.N}
	H^0(E,\pi_*\cO(s)\otimes \cN') \otimes \cO_E \to \pi_*\cO(s)\otimes \cN'
\end{equation}
tensored with $\pi_*\cO(p)\otimes\cM'$. The argument of \cref{le.l} shows that $\pi_{*}\cN\cong\pi_{*}\cO(s)\otimes\cN'$ is also generated by global sections, and hence \cref{eq.can.epi.N} is an epimorphism and so is the first morphism of \cref{eq.comp}. By \cref{split-epi}, the second morphism is also an epimorphism. Thus
there is an exact sequence 
$$
0 \to \ker(\a) \to\pi_*\cG \to \ker(\b) \to 0
$$ 
where $\a$ and $\b$ are the epimorphisms
\begin{equation}\label{eq:lps}
\a: H^0(E,\pi_*\cO(s)\otimes \cN')\otimes \pi_*\cO(p)\otimes \cM'   \, \longrightarrow  \,  \pi_*\cO(p)\otimes \pi_*\cO(s)\otimes \cM'\otimes \cN' 
\end{equation}
and
\begin{equation}\label{eq:rps}
\b: \pi_*\cO(p)\otimes \pi_*\cO(s)\otimes \cM'\otimes \cN' \, \longrightarrow  \, \pi_*\cO(p+s)\otimes \cM'\otimes \cN'.
\end{equation}
The morphism in \Cref{eq:rps} is a split epimorphism so $\ker(\b)$ is semistable of slope
\begin{equation}\label{eq:mukr}
  \mu(\ker(\b)) \; = \; \mu(\pi_*\cO(p)\otimes\pi_*\cO(s)\otimes \cM'\otimes \cN') = (p+s)\mu(\cE)+q+t. 
\end{equation}
On the other hand, $\ker(\alpha)$ is isomorphic to
$T_\cO(\pi_*\cO(s)\otimes \cN')[-1]\otimes \pi_*\cO(p)\otimes \cM'$
with $T_\cO$ as in \Cref{re.to}. We saw in passing in the proof of
\Cref{le.otimes} that $T_\cO(\pi_*\cO(s)\otimes \cN')[-1]$ is semistable and its slope slope $\mu$ 
satisfies
\begin{equation*}
  1\;=\;-\frac 1\mu + \frac 1{\mu(\pi_*\cO(s)\otimes \cN')}\;= \; -\frac 1\mu+\frac 1{s\mu(\cE)+t}.
\end{equation*}
Hence 
\begin{equation*}
  \mu\;=\;\frac{s\mu(\cE)+t}{1-s\mu(\cE)-t}. 
\end{equation*}
In conclusion, we obtain
\begin{equation}\label{eq:mukl}
  \mu(\ker(\a)) \;=\; \mu+\mu(\pi_*\cO(p)\otimes \cM') \;= \; \mu+p\mu(\cE)+q \;= \; \frac{s\mu(\cE)+t}{1-s\mu(\cE)-t} + p\mu(\cE)+q. 
\end{equation}
Clearly, this is smaller than \Cref{eq:mukr}.

\begin{lemma}\label{le.g}
The locally free $\cO_E$-module   $\pi_*\cG$ is a direct sum of semistable summands of slopes $\ge$
  \Cref{eq:mukl} and is generated by global sections.
\end{lemma}
\begin{proof}
  Since $\pi_*\cG$ is an extension of the semistable locally free sheaf $\ker(\b)$ by
  the semistable locally free sheaf $\ker(\a)$, its semistable summands have
  slopes
  \begin{equation*}
  \;  \ge \; \mu(\ker(\a)) \;= \; \min(\mu(\ker(\a)),\mu(\ker(\b))). 
  \end{equation*}
  Since $\mu(\cE)=\mu(\cE_{d})=\frac{1}{d}$, we have $s\mu(\cE)+t>3$ and $p\mu(\cE)+q>4$. Thus
	\begin{equation*}
		\mu(\ker(\b))\;\geq\;\mu(\ker(\a))  \;=\; \frac{1}{\frac{1}{s\mu(\cE)+t}-1}+p\mu(\cE)+q 
		\; \geq \; - 2 +p\mu(\cE)+ q\;>\; 2.
	\end{equation*}
	By \cref{le.semist.pos.deg}\cref{semist.pos.deg.slope}, $\ker(\a)$ and $\ker(\b)$ are generated by global sections. Since $H^{1}(\ker(\a))=0$ by \cref{le.semist.pos.deg}, the sheaves $\ker(\a)$, $\pi_*\cG$, and $\ker(\b)$ form an analogous diagram to \cref{eq.diag.epi}, which shows the global generation of $\pi_*\cG$.
\end{proof}

\begin{remark}\label{re.mu2}
  The above proof shows that $\mu(\ker(\a))>2$.
\end{remark}

\begin{lemma}\label{le.lglg}
  Under the assumptions of \Cref{cv.lmn}, the canonical map
  \begin{equation*}
    H^0(E,\pi_*\cL)\otimes H^0(E,\pi_*\cG) \; \longrightarrow \;  H^0(E,\pi_*\cL\otimes \pi_*\cG)
  \end{equation*}
  is onto. 
\end{lemma}
\begin{proof}
  This will follow from \Cref{le.l,le.g} and \Cref{le.otimes} applied
  to the semistable summands of $\pi_*\cL$ and $\pi_*\cG$ once we
 show that $\mu(\ker(\a))$ satisfies
  \begin{equation}\label{eq:mumu}
   \frac{1}{\mu(\ker(\a))}+\frac 1{\mu(\pi_*\cL)}  \;= \; \frac{1}{\mu(\ker(\a))}+\frac 1{a\mu(\cE)+b} \;< \; 1.
 \end{equation}
 This, however, is immediate from \Cref{re.mu2}, which shows that both
 summands on the left-hand side of \Cref{eq:mumu} are less than
 $\frac 12$.
\end{proof}

\begin{theorem}
\label{pr.rel}
Let $\s: S^gE\to S^gE$ be a translation automorphism.
Let $\cL$ be an invertible $\cO_{S^gE}$-module that is ample and generated by its global sections. 
If $[\cL]=aD+bF$ with $a \ge 1$ and $b \ge 2$, then $B(S^gE,\sigma,\cL)$ is generated in degree one and 
has relations of degrees 2 and 3.
\end{theorem}
\begin{proof} 
By \cref{cor.ab}, $B(S^gE,\sigma,\cL)$ is generated in degree one.

Let $\cM=\cM_m$ and $\cN=\cN_m$ be as defined at the beginning of \cref{sect.ker.mult}. 
Then $\cL$, $\cM_m$, and $\cN_m$ satisfy the assumptions in  \Cref{cv.lmn} for all $m \ge 2$.
 Let $\cG=\cM \otimes \cK$ where $\cK$ is the kernel in  the 
exact sequence
  $$
  0 \to \cK \to H^0(S^gE,\cN) \otimes \cO_{S^gE} \to \cN \to 0.
  $$
Since $B(S^gE,\sigma,\cL)$ is generated in degree one, to prove the theorem it suffices, by  \cref{lem.B.deg.relns}, to show that 
the multiplication map $H^0(S^gE,\cL) \otimes H^0(S^gE,\cG) \to H^0(S^gE,\cL \otimes \cG)$
is onto for all $m \ge 2$. By \cref{lem.using.G}, this is onto if and only if the map $  H^0(E,\pi_*\cL)\otimes H^0(E,\pi_*\cG)  \to
 H^0(E,\pi_*(\cL\otimes \cG))$ is onto. This map factors as  
  \begin{equation*}
    H^0(E,\pi_*\cL)\otimes H^0(E,\pi_*\cG)  \; \longrightarrow \;  H^0(E,\pi_*\cL\otimes \pi_*\cG)  \; \longrightarrow \;  H^0(E,\pi_*(\cL\otimes \cG))
  \end{equation*}
  and the surjectivity of each of the factors follows from \Cref{le.split,le.lglg}. This completes the proof.
\end{proof}

\section{The map $Q_{n,k}(E,\tau) \to B(S^gE,\s',\cL'_{n/k})$ when the characteristic variety is $S^gE$}
\label{sect.applic}

Now we use the results in \cref{se.S^gE,se.rel} when the characteristic variety for $Q_{n,k}(E,\tau)$  is $S^gE$ to show that 
$\Psi_{n/k}:Q_{n,k}(E,\tau) \to B(S^gE,\s',\cL'_{n/k})$ is surjective  and its kernel is generated in degree $\le 3$.

\subsection{Explicit description of $\s':X_{n/k} \to X_{n/k}$ when $X_{n/k} \cong S^gE$}
We write $[m,2^r]$  and $[2^r,m]$ for the continued fractions $[m,2,\ldots, 2]$ and $[2,\ldots,2,m]$, respectively,
when the number of $2$'s is $r$.

\begin{proposition}
\label{prop.7.1}
The characteristic variety $X_{n/k}$ is isomorphic to the $g^{\th}$ symmetric power 
$S^gE$ if and only if $\frac{n}{k}$ is equal to either $[m,2^{g-1}]$  or $[2^{g-1},m]$ for some $m \ge 3$.
In these cases, 
\begin{enumerate}
  \item\label{item.nk.mfirst}
$\frac{n}{k}=[m,2^{g-1}]$  if and only if $n=(m-1)g+1$ and $k=g$;
\item\label{item.nk.mlast}
$\frac{n}{k}=[2^{g-1},m]$ if and only if $n=(m-1)g+1$ and $k=(m-1)(g-1)+1$;
\item\label{item.nk.rho}
the morphism $\rho\colon E^{g}\to S^{g}E$ given by 
\begin{equation}
\label{eq.defn.rho}
	\rho(z_{1},\ldots,z_{g}) \;= \; 
\begin{cases}
(\!(z_{2}-z_{1},z_{3}-z_{2},\ldots,z_{g}-z_{g-1},-z_{g})\!) & \text{when $\frac{n}{k}=[m,2^{g-1}]$}
\\
(\!(-z_1,z_{1}-z_{2},z_{2}-z_{3},\ldots,z_{g-1}-z_{g})\!) & \text{when $\frac{n}{k}=[2^{g-1,m}]$}
\end{cases}
\end{equation}
is a quotient for the action of $\Sigma_{n/k}$ on $E^g$.
\end{enumerate}
\end{proposition}
\begin{proof}
Induction arguments on $g$ show that $\frac{(m-1)g+1}{g}=[m,2^{g-1}]$ and $\frac{(m-1)g+1}{(m-1)(g-1)+1}=[2^{g-1},m]$. It is 
easy to verify that $g$ and $(m-1)(g-1)+1$ are mutual inverses in $\ZZ_n$ when $n=(m-1)g+1$.  
By \cite[Cor.~4.24]{CKS2}, $E^g/\Sigma_{n/k}$ is isomorphic to $S^gE$ if and only if $\frac{n}{k}$ is equal to either 
$[m,2^{g-1}]$  or $[2^{g-1},m]$ for some $m \ge 3$; or, equivalently, if and only if $\Sigma_{n/k} \cong \Sigma_g$ (though the action
of $\Sigma_g$ on $E^g$ is not the ``natural'' one).  Part \cref{item.nk.rho} is  proved in \cite[Prop.~4.25]{CKS2}.
\end{proof}

\begin{proposition}
 \label{prop.action.of.sigma.on.S^gE}
Assume $g \ge 1$ and $m\ge 3$. Assume $\frac{n}{k}$ is either $[m,2^{g-1}]$ or $[2^{g-1},m]$.
Let $\rho:E^g \to S^gE$ be the corresponding quotient map  in \cref{eq.defn.rho}.  Let $\tau'=(m-2)\tau$.
There is a commutative diagram 
$$
\xymatrix{
E^g \ar[d]_\rho  \ar[r]^\s & E^g \ar[d]^\rho 
\\
S^gE \ar[r]_{\s'} & S^gE
}
$$
in which $\s:E^g \to E^g$ is the automorphism $\s(z_1,\ldots,z_g)= (z_1+\tau_1,\ldots,z_g+\tau_g)$  
defined in \cref{sssect.sigma}, and $\s':S^gE \to S^gE$ is the automorphism  $\s'(\!(z_1,\ldots,z_g)\!) \, = \,   (\!(z_1+\tau',\ldots,z_g+\tau')\!)$. 
\end{proposition}
\begin{proof}
In both cases,  $n=(m-1)g+1$.
By definition, $\tau_i = (k_i+l_i-n)\tau$ where $k_i$ and $l_i$ are the integers defined in \cref{sssect.sigma}. 
Since the characteristic variety $X_{n/k}$ is isomorphic to $E^g/\Sigma_{n/k}$, the existence and uniqueness of the 
automorphism $\s':S^gE \to S^gE$ making the diagram commute is established  in \cite[Prop.~2.10]{CKS2}.

(1)
If $\frac{n}{k}=[m,2^{g-1}]$, then \cref{eq.kl.ind} implies
\begin{equation*}
	mk_{1}\;=\;k_{0}+k_{2}\quad\text{and}\quad 2k_{i}\;=\;k_{i-1}+k_{i+1}\quad (2\leq i\leq g).
\end{equation*}
Since $k_{0}=n$ and $k_{1}=k$ as in \cref{eq.nk.sp},
\begin{equation*}
	k_{g+1}-k_{g}\;=\;k_{g}-k_{g-1}\;=\;\cdots\;=\;k_{2}-k_{1}\;=\;(m-1)k_{1}-k_{0}\;=\;(m-1)k-n
\end{equation*}
which is, by \cref{prop.7.1}\cref{item.nk.mfirst}, equal to $-1$.
Hence 
\begin{equation*}
	k_{i}\;=\;k_{1}+(i-1)(-1)=g-i+1
\end{equation*}
for $i=1,\ldots,g+1$. Since $l_{0}=0$, $l_{1}=1$, and $l_{i}$'s satisfy the same inductive formula as $k_{i}$'s,
\begin{equation*}
	l_{g+1}-l_{g}\;=\;l_{g}-l_{g-1}\;=\;\cdots\;=\;l_{2}-l_{1}\;=\;(m-1)l_{1}-l_{0}\;=\;m-1.
\end{equation*}
Thus
\begin{equation*}
	l_{i}\;=\;l_{1}+(i-1)(m-1)=(i-1)(m-1)+1.
\end{equation*}
It follows that $\tau_i = (k_i+l_i-n)\tau = (2-m)(g-i+1)\tau$ for $i=1,\ldots,g$. 
For all $i=1,\ldots,g-1$, $\tau_{i+1}-\tau_i=(m-2)\tau =\tau'$. Therefore
\begin{align*}
\rho\s(z_1,\ldots,z_g) & \;=\; \rho(z_1+\tau_1,\ldots,z_g+\tau_g)
\\
& \;=\; (\!(z_2+\tau_2-z_1-\tau_1, \ldots, z_g+\tau_g-z_{g-1}-\tau_{g-1}, -z_g-\tau_g )\!)
\\
& \;=\; (\!(z_2-z_1+\tau', \ldots, z_g -z_{g-1}+\tau', -z_g+\tau' )\!)
\\
& \;=\; \s'(\!(z_2-z_1, \ldots, z_g -z_{g-1}, -z_g)\!).
\end{align*}
Thus, $\rho\s=\s'\rho$. 

(2)
Suppose $\frac{n}{k}=[2^{g-1},m]$. Since
\begin{equation*}
	2k_{i}\;=\;k_{i-1}+k_{i+1}\quad (1\leq i\leq g-1)\quad\text{and}\quad mk_{g}\;=\;k_{g-1}+k_{g+1},
\end{equation*}
it follows that
\begin{equation*}
	k_{1}-k_{0}\;=\;k_{2}-k_{1}\;=\;\cdots\;=\;k_{g}-k_{g-1}\;=\;k_{g+1}-(m-1)k_{g}\;=\;-(m-1).
\end{equation*}
Hence
\begin{equation*}
	k_{i}\;=\;k-(i-1)(m-1)\;=\;(g-i)(m-1)+1
\end{equation*}
for $i=1,\ldots,g$. By \cref{eq.nk.sp} and the proof of \cref{prop.7.1}, $l_{g}=k'=g$. Hence
\begin{equation*}
	l_{1}-l_{0}\;=\;l_{2}-l_{1}\;=\;\cdots\;=\;l_{g}-l_{g-1}\;=\;l_{g+1}-(m-1)l_{g}\;=\;n-(m-1)k'\;=\;1
\end{equation*}
implies that $l_{i}=i$ for $i=1,\ldots,g$.
It follows that $\tau_i = (k_i+l_i-n)\tau = -(m-2)i\tau$ for $i=1,\ldots,g$ . 
For all $i=1,\ldots,g-1$, $\tau_{i}-\tau_{i+1}=(m-2)\tau$. 
Now,  
\begin{align*}
\rho\s(z_1,\ldots,z_g) & \;=\; \rho(z_1+\tau_1,\ldots,z_g+\tau_g)
\\
& \;=\; (\!(-z_1-\tau_1, z_1+\tau_1-z_2-\tau_2, \ldots, z_{g-1}+\tau_{g-1}-z_{g}-\tau_{g})\!)
\\
& \;=\; (\!(-z_1+\tau', z_1-z_2+\tau', \ldots,z_{g-1} -z_{g}+\tau' )\!).
\end{align*}
Thus, $\rho\s=\s'\rho$. 
\end{proof}

\subsection{The special case $Q_{5,2}(E,\tau)$}
When $(n,k)=(5,2)$, $\frac{n}{k}=[3,2]$, so $\s':S^2E \to S^2E$ is given by 
$$
\s'(\!(z_1,z_2)\!) \;=\; (\!(z_1+\tau,z_2+\tau)\!).
$$
Hence \cref{prop.action.of.sigma.on.S^gE}  agrees with a remark after Proposition 4.2 in the Kiev preprint \cite{FO-Kiev} which says there is a (surjective) homomorphism  
$Q_{5,2}(E,\tau) \to B(S^2E, \nu,\cL')$ where $\nu$ is the automorphism $(\!(z_1,z_2)\!) \mapsto (\!(z_1+\tau,z_2+\tau)\!)$. 
The next result shows there is also a (surjective) homomorphism  $Q_{5,2}(E,\tau) \to B(S^2E, \nu^{-1},\cL')$ where $\nu^{-1}$ 
is the automorphism $(\!(z_1,z_2)\!) \mapsto (\!(z_1-\tau,z_2-\tau)\!)$.  

\begin{proposition}
\label{prop.isom.thcr}
Let $m$ be an integer $\ge 3$ and assume $n=(m-1)g+1$ and $k=g$. Thus $X_{n/k} \cong S^gE$.
Let $\s:E^g \to E^g$ be the automorphism that is translation by $(\tau_1,\ldots,\tau_g)$. 
Let $\s_1',\s_2':S^gE \to S^gE$ be the automorphisms such that $\rho_i\s=\s_i'\rho_i$ where $\rho_1,\rho_2:E^g \to S^gE$ 
are the quotient morphisms
\begin{align*}
\rho_1(z_1,\ldots,z_g) & \;=\; (\!(z_2-z_1,\ldots,z_g-z_{g-1},-z_g)\!)
\\
\rho_2(z_1,\ldots,z_g) & \;=\; (\!(z_1-z_2,\ldots,z_{g-1}-z_{g},z_g)\!)
\end{align*}
for the action of $\Sigma_{n/k}$ on $E^g$. If $\cL_i' = (\rho_{i*} \cL_{n/k})^{\Sigma_{n/k}}$, then there is an isomorphism of triples $(f,u):(S^gE,\s_1',\cL_1') \to (S^gE,\s_2',\cL_2')$ where $f:S^gE \to S^gE$ is the automorphism
$$
f(\!(z_1,\ldots,z_g)\!) \;=\;  (\!(-z_1,\ldots,-z_g)\!)
$$
and hence an isomorphism $B(S^gE,\s_1',\cL_1') \cong B(S^gE,\s_2',\cL_2')$.
\end{proposition}
\begin{proof}
Since
\begin{equation*}
	f\sigma'_{1}\rho_{1}\;=\;f\rho_{1}\sigma\;=\;\rho_{2}\sigma\;=\;\sigma'_{2}\rho_{2}\;=\;\sigma'_{2}f\rho_{1}
\end{equation*}
and $\rho_{1}$ is surjective,
$f\s_1'=\s_2'f$. The proof will be complete once we show that $f^*\cL_2' \cong \cL_1'$.

Since $f$ is $\Sigma_{n/k}$-equivariant,
\begin{equation*}
	f^*\cL_2'
	\;=\;f^{*}\big((\rho_{2*}\cL_{n/k})^{\Sigma_{n/k}}\big)
	\;\cong\;f^{-1}_{*}\big((\rho_{2*}\cL_{n/k})^{\Sigma_{n/k}}\big)
	\;\cong\;(f^{-1}_{*}\rho_{2*}\cL_{n/k})^{\Sigma_{n/k}}
	\;=\;(\rho_{1*}\cL_{n/k})^{\Sigma_{n/k}}
	\;=\;\cL_1'.\qedhere
\end{equation*}
\end{proof}

\subsection{The map $\Psi_{n/k}:Q_{n,k}(E,\tau) \to B(S^gE,\s',\cL'_{n/k})$}
We continue to assume that $X_{n/k} \cong S^gE$, i.e., that $\frac{n}{k}$ is either $[m,2^{g-1}]$ or $[2^{g-1},m]$ where
 $m$ is an integer $\ge 3$.   
We identify $X_{n/k}$ with $S^gE$ via the quotient morphism $\rho:E^g \to S^gE$ in \cref{eq.defn.rho}.  
In other words, there is a closed immersion $i$ such that  $\Phi_{n/k}\colon E^g \to \PP^{n-1}$  factors as
$$
\xymatrix{
  E^g    \ar[r]^>>>>>{\rho} & S^gE  \ar[r]^{i} & \PP^{n-1}.
}
$$
Let $\cL'_{n/k} \;=\; i^{*}\cO_{\bbP^{n-1}}(1)$.

\begin{lemma}\label{lem.ns.pb.va}
Assume $\frac{n}{k}=[m,2^{g-1}]$. 
Then  $[\cL'_{n/k}]=D+(m-1)F$ in $\NS(S^gE)$. 
\end{lemma}
\begin{proof}
By definition, $D=[D_0]$ and $F=[F_0]$ where 
$F_0=\{(\!(z_1,z_2,\ldots,z_g)\!) \; | \; z_1+\cdots+z_g =0\}$ and $D_0 =\{(\!(0,z_2,\ldots,z_g)\!) \; | \; z_2,\ldots,z_g \in E\}$. 

	The N\'eron-Severi class of $\cL'_{n/k}$ is $aD+bF$ for some $a,b\in\bbZ$. 
	The divisor $\rho^*(aD_0+bF_0)$ equals
	\begin{equation*}
	D' \; :=\; 	a\big(\Delta_{1,2}+\cdots+\Delta_{g-1,g}  \,+ \, (E^{g-1}\times \{0\}) \big) \; + \;b\big(\{0\} \times E^{g-1} \big).
	\end{equation*}
	Since $\rho^{*}i^{*}\cO_{\bbP^{n-1}}(1)=\Phi_{n/k}^{*}\cO_{\bbP^{n-1}}(1)=\cO_{E^{g}}(D_{n/k})$, 
	the divisors $D_{n/k}$ and $D'$ are linearly equivalent and therefore give the same class in $\NS(E^g)$. 
	
	Fix a point $p \in E^{g-1}$ in general position and let $C_1$ and $C_g$ be the curves
	$E \times \{p\}$ and $\{p\} \times E$ on $S^gE$. Then $D' \cdot C_1 = a+b$ and $D_{n/k} \cdot C_1=(m-1)+1$ so $a+b=m$.
	Also, $D'\cdot C_2=2a$ and $D_{n/k} \cdot C_2=(n_g-1)+1=2$ so $2a=2$. Therefore $a=1$ and $b=m-1$.
\end{proof}

\begin{theorem}
\label{thm.B(S^gE)}
Assume $\frac{n}{k}=[m,2^{g-1}]$. 
For all translation automorphisms $\s:S^gE\to S^gE$, the algebra $B(S^gE,\sigma,\cL'_{n/k})$ is generated in degree one 
and has relations in degrees 2 and 3.
\end{theorem}
\begin{proof}
By \cref{lem.ns.pb.va}, $[\cL'_{n/k}]=D+(m-1)F$. Since $m \ge 3$,  \cref{pr.rel} applies. 
\end{proof}

\begin{corollary}\label{cor.12}
Fix an integer $m\ge 3$ and assume $\frac{n}{k}=[m,2^{g-1}]$. Let  $\s':S^gE \to S^gE$ be the translation automorphism
by  $(m-2)\tau$, i.e., the automorphism in \cref{prop.action.of.sigma.on.S^gE}. The homomorphism  
$\Psi_{n/k}:Q_{n,k}(E,\tau)\to B(S^gE,\s',\cL'_{n/k})$ is surjective and 
$\ker(\Psi_{n/k})$ is generated by elements of degree $\le 3$.  
\end{corollary}
\begin{proof}
This follows immediately from \cref{thm.B(S^gE)}.
\end{proof}

\subsection{The algebras $Q_{2k+1,k}(E,\tau)$}
Since $\frac{2k+1}{k}=[3,2^{k-1}]$, $X_{(2k+1)/k}=S^kE \subseteq \PP^{2k}$.

\begin{lemma}
\label{lem.2k+1.k}
If $\tau\in (E-E[2])\cup\{0\}$, then the kernel of the homomorphism $\Psi_{(2k+1)/k}:Q_{2k+1,k}(E,\tau) \to B(S^kE,\s',\cL'_{(2k+1)/k})$ is generated by 
$\frac{1}{6}k(k+1)(2k+1)$ elements of degree 3.
\end{lemma}
\begin{proof}
For brevity, we write $\Psi=\Psi_{(2k+1)/k}$ and $\cL'=\cL'_{(2k+1)/k}$.

Let $Q_j$ and $B_j$ denote the degree-$j$ components of $Q_{2k+1,k}(E,\tau) $ and $B(S^kE,\s',\cL')$,
respectively. 

To prove that $\ker(\Psi)$ is zero in degree two we must show that $\dim(Q_2)=\dim(B_2)$.
Since $\tau\in (E-E[2])\cup\{0\}$, we see in \cite[Thm.~5.10]{CKS4} that
$$
\dim(Q_2) \;=\;
\binom{2k+2}{2}\;=\;(2k+1)(k+1).
$$
We will use a special case of \cite[Thm.~1.17]{CaCi93}:  let $\cL$ be an invertible sheaf on $S^kE$ such that 
$[\cL]=aD+ bF$; if $a \ge 0$ and $a+kb>0$, then
$$
\dim H^0(S^kE,\cL) \; = \; \frac{(a+kb)}{k!} \, \prod_{i=1}^{k-1} (a+i)\, .
$$
Since $[\cL']=D+2F$, $[\cL' \otimes (\s')^*\cL']=2D+4F$. Hence
$$
\dim(B_2) \;=\; \dim H^0(S^kE,\cL' \otimes (\s')^*\cL') \; = \; \frac{(2+4k)}{k!} \, \prod_{i=1}^{k-1} (2+i) \;=\; (1+2k)(k+1).
$$
Thus, $\dim(B_2)=\dim(Q_2)$.
 
On the other hand, $\dim(Q_3)={{2k+3} \choose {3}}$ and $\dim(B_3)$ is 
$$
\dim H^0(S^kE,\cL' \otimes (\s')^*\cL'\otimes (\s')^{2*}\cL') \; = \; \frac{(3+6k)}{k!} \, \prod_{i=1}^{k-1} (3+i) \;=\;
\tfrac{1}{2} (1+2k)(k+1)(k+2)
$$
so $\dim(Q_3)-\dim(B_3)=\frac{1}{6}k(k+1)(2k+1)$.
\end{proof}

For example, the kernel of the map $Q_{5,2}(E,\tau) \to B(S^2E,\s',\cL')$ is generated by 5 elements of degree 3 when $\tau\in(E-E[2])\cup\{0\}$.  
When $\tau=0$ this recovers the well-known fact  that the image of the map $S^2E \to \PP^4$ is the intersection of 5 cubic hypersurfaces.
Feigin and Odesskii say that the subalgebra of $Q_{5,2}(E,\tau) $ generated by those 5 degree-3 elements is isomorphic to 
$Q_{5,1}(E,\tau)$ (\cite[p.~25]{FO-Kiev}). We do not know how to prove this.

\begin{proposition}
\label{prop.Q.2k+1.k}
Let $\PP^{2k}=\PP(V^*)$ and let $i:S^kE \to \PP^{2k}$ be the closed immersion given by the complete linear system $|\cL'|$
where $[\cL']=D+2F$. If $\tau\in (E-E[2])\cup\{0\}$, then the space of relations for $Q_{2k+1,k}(E,\tau)$ is the subspace of $V \otimes V$ vanishing on the graph of 
the automorphism $\s':S^kE\to S^kE$ that is translation by $\tau$. 
\end{proposition}
\begin{proof} 
As in \cref{lem.2k+1.k}, we write $\Psi=\Psi_{(2k+1)/k}$ and $\cL'=\cL'_{(2k+1)/k}$.

Because $\Psi$ is an isomorphism in degree two, $Q_{2k+1,k}(E,\tau)$ and  $B(S^kE,\s',\cL')$ have the same quadratic relations.  
Thus, the space of quadratic relations for $Q_{2k+1,k}(E,\tau)$ coincides with the kernel of the multiplication map 
$$
B_1 \otimes B_1 \, = \, H^0(S^kE \times S^kE,\cL' \boxtimes \cL')  \, \longrightarrow \, H^0(S^kE,\cL' \otimes (\s')^*\cL') \, = \, B_2.
$$ 
Let $X=S^kE$ and let $\G_{\s'} \subseteq X^2$ denote the graph of $\s'$.
If we apply the functor $(\cL' \boxtimes \cL') \otimes_{\cO_{X^2}} -$ to the exact sequence 
$0 \to \cO_{X^2}(-\G_{\s'}) \to \cO_{X^2} \to \cO_{\G_{\s'}} \to 0$ and take global sections it becomes clear that the above kernel is 
$H^0(X^2,(\cL' \boxtimes \cL')(-\G_{\s'}))$ which is the subspace of $H^0(S^kE \times S^kE,\cL' \boxtimes \cL')$ consisting of 
the sections that vanish on $\G_{\s'}$.

Since $\frac{2k+1}{k}=[3,2^{k-1}]$, \cref{prop.action.of.sigma.on.S^gE} tells us that $\s'(\!(x_1,\ldots,x_k)\!)=
(\!(x_1+\tau,\ldots,x_k+\tau)\!)$.
\end{proof}

\section{The  rings $B(E^g,\s, \cL_{n/k})$}
\label{se.ep}

In this and the next section we assume that the following equivalent conditions hold:
\begin{enumerate}
  \item 
  $X_{n/k}=E^g$;
  \item 
  $\cL_{n/k}$ is very ample;
  \item 
  all the $n_i$'s in $[n_1,\ldots,n_g]=\frac{n}{k}$ are $\ge 3$.
\end{enumerate}

Let $\s:E^g \to E^g$ be an arbitrary translation automorphism. 

In  this section we show that $ B(E^g,\s,\cL_{n/k})$  is generated in degree one.

In  \cref{se.rel.ep}, we show  that the ideal of relations for $B(E^g,\s,\cL_{n/k})$   is generated by elements of degree  $\le 3$.
Finally, we apply this to the particular $\s$ relevant to $Q_{n,k}(E,\tau)$.

\subsection{The main result in this section}\label{ssec.prod.gen.deg.one}
The fact that $B(E^g,\s,\cL_{n/k})$ is generated in degree one will follow from \cref{pr.l'l''}, the proof of which  occupies most of 
this section.

\begin{proposition}\label{pr.l'l''}
Suppose all the $n_i$'s in $[n_1,\ldots,n_g]=\frac{n}{k}$ are $\ge 3$.
  If $\cL'$ and $\cL''$ are tensor products of translates of $\cL_{n/k}$, then the multiplication map
  \begin{equation}\label{eq:l'l''}
    H^0(E^g,\cL')\otimes H^0(E^g,\cL'') \; \longrightarrow \;  H^0(E^g,\cL'\otimes\cL'')
  \end{equation}
 is onto.    
\end{proposition}

Our strategy for  proving \cref{pr.l'l''} is similar to that  used to prove \Cref{le.aux}\cref{psb.mult.onto}.

\subsection{Notation}
Most of the notation in this section and the next is the same as that in \cref{sect.Xn/k}
though we make some simplifications and introduce some new notation as follows:
\begin{enumerate}
  \item 
  $X=E^g$;
    \item
  $\pi:E^g\to E$ is the projection $\pi(z_1,\ldots,z _g)=z _g$;
    \item
  $\pi':E^g\to E^{g-1}$ is the projection $\pi(z_1,\ldots,z _g)=(z_{1},\ldots,z_{g-1})$;
 \item
 $\rho:E^{g-1} \to E$ is the projection $\pi(z_1,\ldots,z _{g-1})=z _{g-1}$;
  \item 
  $\cL=\cL_{n/k}$;
  \item
  $D=D_{n/k}$ defined in \cref{eq:d}; thus $\cL=\cO_{E^g}(D)$;
 \item
 $\s:E^g \to E^g$ is an arbitrary translation automorphism;
\end{enumerate}

\subsection{Preliminary results}

\begin{proposition}
\label{prop.grauert}
Let $k'$ be the unique integer such that $n>k'\ge 1$ and $kk'=1\ (\mathrm{mod}\ n)$.
  \begin{enumerate}
  \item\label{item.grauert.loc.free}
  $\pi_*\cL$ is a locally free $\cO_E$-module of  degree $n$ and rank $k'$.
  \item\label{item.grauert.slope}
$\mu(\pi_*\cL)=n/k'$. 
  \item\label{item.grauert.cohom}
  $H^1(E,\pi_*\cL)=0$. 
\end{enumerate}
\end{proposition}
\begin{proof} 
For all $z \in E$, the restriction, $\cL_z$, of $\cL$ to $X_z=\pi^{-1}(z)$
 is a standard divisor of type $(n_{1},\ldots,n_{g-1})$ (see \cref{sect.Xn/k}), which is very ample because all the $n_i$'s are $\ge 3$.
The dimension of $H^0(\cL_z)$ is $d(n_1,\ldots,n_{g-1})$ (see \cref{sssect.sigma})
and $H^q(\cL_{[n_1,\ldots,n_{g-1}]})=0$ for all $q \ge 1$.
Since the dimension of $H^0(X_z,\cL_z)$ is independent of $z$, 
Grauert's Theorem \cite[Cor.~III.12.9]{Hart}
tells us that $\pi_*\cL$ is locally free of rank $d(n_1,\ldots,n_{g-1})$. 
The higher cohomology groups of $\cL_z$ are zero so, by Grauert's Theorem again,
the higher cohomology groups of $\pi_* \cL $ are also zero.

Since $H^1(E,\pi_*\cL)=0$, the dimension of $H^0(E,\pi_* \cL)\cong H^{0}(X,\cL)$ equals the degree of $\pi_*\cL$. Hence
\begin{equation*}
  \mu(\pi_*\cL) \;=\; \frac{d(n_1,\ldots,n_g)}{d(n_1,\ldots,n_{g-1})} \;= \; [n_g,\ldots,n_1] \; =\; \frac{n}{k'}.
\end{equation*}

The proof is complete.
\end{proof}

\begin{lemma}\label{le.push}
For all $x \in E^g$, $\pi_*T_x^*\cL$ is indecomposable of slope $>2$.
\end{lemma}
\begin{proof}
Since $\pi_*T_x^*\cL=T_{x_{g}}^{*}\pi_{*}\cL$, we can assume $x=0$. The slope of $\pi_*\cL$ is $n/k'$ by \cref{prop.grauert}.
Using the fact that all $n_i$'s are $\ge 3$, an induction argument on $g$ shows that $n/k'>2$.

We prove that $\pi_*\cL$ is indecomposable by induction on $g$.
The case $g=1$ is trivial so  we  assume that $g\ge 2$ and that the result is true for $g-1$. The induction hypothesis will
be applied to the left-hand factor in $E^{g-1} \times E$ and the sheaf $\cL_{[n_1,\ldots,n_{g-1}+1]}$.

In \cite[\S 3.1.3]{CKS2}, we observed that $D=D_{n/k}$ is linearly equivalent to
\begin{equation*}
	D'_{[n_{1},\ldots,n_{g}]}\;=\;\sum_{i=1}^{g}E^{i-1}\times\fd'_{i}\times E^{g-i}-\sum_{j=1}^{g-1}\Delta'_{j,j+1}
\end{equation*}
where $\fd'_{i}=(n_{i}+2-\delta_{i,1}-\delta_{i,g})(0)$ and
\begin{equation*}
	\Delta'_{j,j+1}\;=\;\{(z_{1},\ldots,z_{g})\in E^{g}\;|\;z_{j}+z_{j+1}=0\}.
\end{equation*}
Thus $\cL\cong\cO_{E^{g}}(D'_{[n_{1},\ldots,n_{g}]})$. The rest of the proof uses the divisor $D'_{[n_{1},\ldots,n_{g}]}$ rather than $D_{n/k}$. Let 
\begin{equation*}
	D''\;=\;D'_{[n_{1},\ldots,n_{g}]}-E^{g-1}\times\fd'_{g}.
\end{equation*}

Corresponding to the factorization $E^g=E^{g-1} \times E$ and the decomposition $D=D'' +(E^{g-1}\times\fd'_{g})$, 
 $$
 \cL \; \cong \; \cM \otimes \pi^*\cF
 $$ 
 where $\cM:=\cO_{E^g}(D'')$ and $\cF :=\cO_E(\fd'_{g})$.  
By the projection formula,
\begin{equation}\label{eq:lmf}
  \pi_*\cL \; \cong \; \pi_*\cM\otimes \cF. 
\end{equation}
Since $\cF$ is an invertible $\cO_E$-module, $\pi_*\cL$ is indecomposable if and only if $\pi_*\cM$ is. We will
show that $\pi_*\cM$ is indecomposable.

By Grauert's Theorem \cite[Cor.~III.12.9]{Hart}, $\pi_*\cM$ is the locally free $\cO_E$-module whose fiber at $z\in E$ is 
\begin{equation*}
(\pi_*\cM)_z \; \cong \;   H^0\big(E^{g-1},\cO\big(D'_{L}- E^{g-2}\times\{-z\}\big)\big),
\end{equation*}
where $D'_{L}:=D'_{[n_{1},\ldots,n_{g-1}+1]}$ and $E^{g-1}$ is identified with $E^{g-1}\times\{z\}\subseteq E^g$.
The inclusion $  \cO_{E^{g-1}}(D'_{L}-E^{g-2}\times\{-z\})\subseteq \cO_{E^{g-1}}(D'_{L})$ gives rise to an inclusion 
\begin{equation*}
(\pi_*\cM)_z \; \cong \;  H^0\big(E^{g-1},\cO\big(D'_{L}-E^{g-2}\times\{-z\}\big)\big) \; \subseteq \; H^0(E^{g-1},\cO(D'_{L}))
\end{equation*}
at the level of global sections.  
This inclusion is the $z$-fiber of a monomorphism
\begin{equation}\label{eq:emb} 
  \pi_*\cM    \; \longrightarrow \; H^0(E^{g-1},\cO(D'_{L}))\otimes \cO_{E}
\end{equation}
between locally free $\cO_E$-modules. 
The cokernel of \Cref{eq:emb} is $\theta^*\rho_*\cO_{E^{g-1}}(D'_{L})$, where $\theta$ is the involution $z\mapsto -z$ on $E$.
After tracing through the identification above, the resulting map
\begin{align*}
  H^0(E^{g-1},\cO(D'_{L}))\otimes \cO_{E} 
  &\; \cong \; H^0(E,\rho_*\cO(D'_{L}))\otimes \cO_{E}\\
  &\; \cong \; H^0(E,\theta^*\rho_*\cO(D'_{L}))\otimes \cO_{E}\\
  &\;  \to \; \theta^*\rho_*\cO(D'_{L})
\end{align*}
is the canonical surjection exhibiting the right hand side as a locally free sheaf that is generated by its global sections.

This means that $\pi_*\cM$ can be identified with the shift $T_\cO\left(\theta^*\rho_*\cO(D'_{L})\right)[-1]$, where $T_\cO$ is the autoequivalence of the bounded derived category ${\sf D}^b(\coh(E))$ from \Cref{re.to}. By the induction hypothesis,
 $\theta^*\rho_*\cO(D'_{L})\cong\theta^*\rho_*\cL_{[n_{1},\ldots,n_{g-1}+1]}$ is indecomposable, so  $\pi_*\cM$ is also indecomposable. 
\end{proof}

\subsubsection{Remark}
	The exact sequence
	\begin{equation*}
		0\longrightarrow\pi_{*}\cM\longrightarrow H^{0}(E^{g-1},\cO(D'_{L}))\otimes\cO_{E}\longrightarrow\theta^{*}\rho_{*}(\cO(D'_{L}))\longrightarrow 0
	\end{equation*}
	that appeared in the proof of \cref{le.push} can  be obtained in another way.
	First, for simplicity, write  $\Delta'$  for $\Delta'_{g-1,g}$, and let $j\colon\Delta'\to X$ be the inclusion. 
	After applying the functor $\cO_{X}(D'_{L}\times E)\otimes -$ to
	\begin{equation*}
		0\to\cO_{X}(-\Delta')\to\cO_{X}\to j_{*}\cO_{\Delta'}\to 0,
	\end{equation*}
	we obtain
	\begin{equation}\label{eq.rem.cm}
		0\to\cM\to\cO_{X}(D'_{L}\times E)\to (j_{*}\cO_{\Delta'})(D'_{L}\times E)\to 0.
	\end{equation}
	Here $\cO_{X}(D'_{L}\times E)=\cO_{E^{g-1}}(D'_{L})\boxtimes\cO_{E}$ and
	\begin{align*}
		(j_{*}\cO_{\Delta'})(D'_{L}\times E) &\; \cong \;   (j_{*}\cO_{\Delta'})\otimes\cO_{X}(D'_{L}\times E)
		\\
		&\; \cong \;  j_{*}(\cO_{\Delta'}\otimes j^{*}\cO_{X}(D'_{L}\times E))
		\\
		&\; \cong \;   j_{*}j^{*}\pi'^{*}(\cO_{E^{g-1}}(D'_{L}))
	\end{align*}
	by the Projection Formula. Grauert's theorem shows that the fiber of $R^{1}\pi_{*}(\cM)$ at $z\in E$ is isomorphic to
	\begin{equation*}
		H^{1}(X_{z},\cM_{z})\cong H^{1}(E^{g-1},\cO_{E^{g-1}}(D'_{L}-E^{g-2}\times\{-z\})),
	\end{equation*}
	which is zero by \cref{2-char-var.prop}\cref{item:5}. Thus applying $\pi_{*}$ to \cref{eq.rem.cm} produces the exact sequence
	\begin{equation*}
		0\to\pi_{*}\cM\to H^{0}(E^{g-1},\cO_{E^{g-1}}(D'_{L}))\otimes\cO_{E}\to \pi_{*}j_{*}j^{*}\pi'^{*}(\cO_{E^{g-1}}(D'_{L}))\to 0.
	\end{equation*}
	Since $\pi_{*}j_{*}j^{*}\pi'^{*}=\theta^{*}\rho_{*}$, this is the desired exact sequence.

\begin{corollary}\label{cor.push}
For all $x \in E^g$, $\pi_*T_x^*\cL$  is stable of slope $>2$. 
\end{corollary}
\begin{proof}
By \Cref{le.push}, $\pi_*T_x^*\cL$ is indecomposable of slope $>2$. 
Stability now follows from \Cref{le.indec}, together with the observation made in the course of the proof of \cref{prop.grauert}
that the degree and rank of $\pi_*\cL$ are  $n$ and $k'$, respectively, which are coprime.  
\end{proof}

To prove \Cref{pr.l'l''} we must deal with tensor products of translates of $\cL$.
A first step is the  following observation.

\begin{lemma}\label{le.add}
If $\cL_1,\ldots,\cL_m$ are  translates  of $\cL$, then
 $   \mu\left(\pi_*\left(\cL_1\otimes\cdots\otimes \cL_m\right)\right) = m\mu\left(\pi_*\cL\right)$. 
\end{lemma}
\begin{proof}
In order to lighten the notation we only address the case $\cL_i=\cL$ for all $i$. The general case is analogous.

Since  $\cL=\cO_{E^g}(D)$,  
\begin{equation*}
\deg\big(\pi_*\cL^{\otimes m}\big) \;=\;   h^0(\pi_*\cL^{\otimes m}) \;=\; h^0(\cL^{\otimes m}) \;=\; \frac {(mD)^g}{g!} \;= \; m^g \deg(\pi_*\cL), 
\end{equation*}
where the third equality follows from the Riemann-Roch theorem for abelian varieties \cite[III.16]{Mum08}.  

On the other hand, a similar argument to the proof of \Cref{le.push} shows that $\pi_*\cL^{\otimes m}$ is a locally free $\cO_E$-module whose fiber over $z\in E$ is
\begin{equation}\label{eq:h0}
  H^0\left(E^{g-1},\cO_{E^{g-1}}\left(m\sum_{i=1}^{g-1}E^{i-1}\times D_{i}\times E^{g-i-1}+mE^{g-2}\times\{z\}+m\sum_{j=1}^{g-2}\Delta_{j,j+1}\right)\right).
\end{equation}
Because $E^{g-2}\times D_{g-1}$ is the pullback of a degree-$(n_{g-1}-2+\delta_{g-1,1})$ divisor on the right-most factor $E$ of $E^{g-1}$, the parenthetic divisor in \Cref{eq:h0} is $mD''$ for a standard divisor $D''$ of type $(n_{1},\ldots,n_{g-1})$. By another application of the Riemann-Roch theorem, the dimension of the 
vector space in \Cref{eq:h0} is
\begin{equation*}
  \frac{(mD'')^{g-1}}{(g-1)!} \;=\; m^{g-1}d(n_1,\cdots,n_{g-1}) \;=\; m^{g-1}\mathrm{rank}(\pi_*\cL). 
\end{equation*}
In conclusion, lifting $\cL$ to the $m^{th}$ tensor power scales the degree of $\pi_*\cL$ by $m^g$ and its rank by $m^{g-1}$. In conclusion the slope scales by $m$, as claimed. 
\end{proof}

\subsubsection{Remark}
\Cref{le.add} says that the map $\mu(\pi_*-)$ is a morphism from the sub-semigroup of $\NS(E^g)$ generated by $\cL$ to $\QQ_{>0}$.

\subsection{Proof of  \Cref{pr.l'l''}}
We will  prove the following more precise version of  \Cref{pr.l'l''}. 

\begin{lemma}\label{le.pm}
If   $\cL_0, \ldots,\cL_m$ are translates of $\cL$, then each $\pi_*(\cL_0\otimes\cdots\otimes \cL_i)$ is semistable and the map
\begin{equation*}
  H^0(E^g,\cL_0)\otimes H^0(E^g,\cL_1\otimes\cdots\otimes \cL_m) \; \longrightarrow \;  H^0(E^g,\cL_0\otimes\cdots\otimes \cL_m)
\end{equation*}
is onto.   
\end{lemma}
\begin{proof}
We denote the statement in the lemma by ${\bf P}_{g,m}$ since it depends on both indices. 
We will argue by induction on $g+m$. To make sure that the induction works, we prove ${\bf P}_{g,m}$ when $\cL$ is an arbitrary standard divisor of type $[n_{1},\ldots,n_{g}]$. However, for simplicity, we write the proof only for $\cL=\cL_{n/k}$.

The case $m=0$ is straightforward and the case $g=1$ follows from \cref{thm.good.U.V}. 
From now on we assume that $g\ge 2$ and $m\ge 1$. 

We assume without loss of generality, as will be clear from the proof,  that all $\cL_i$ are isomorphic to $\cL$. 

Identifying, for the locally free sheaves in question, $H^0(E^g,-)$ with $H^0(E,\pi_*(-))$, the surjectivity claim decomposes into the two separate demands that
\begin{equation}
  \label{eq:pre}
  H^0(E,\pi_*\cL)\otimes H^0(E,\pi_*\cL^{\otimes m})\to H^0(E,\pi_*\cL\otimes \pi_*\cL^{\otimes m})
\end{equation}
and
\begin{equation}
  \label{eq:post}
  H^0(E,\pi_*\cL\otimes \pi_*\cL^{\otimes m}) \to H^0(E,\pi_*\cL^{\otimes (m+1)})
\end{equation}
both be onto. 

The semistability claim in ${\bf P}_{g,m-1}$ implies that the two locally free sheaves appearing in the domain of \Cref{eq:pre} are semistable. 
By \Cref{le.push,le.add}, $\pi_*\cL$ and $\pi_*\cL^{\otimes m}$ have slope $>2$, so \Cref{thm.good.U.V} tells us that  \Cref{eq:pre} 
is surjective. 

It remains to show that \Cref{eq:post} is onto. As before, let $z \in E$ and define $X_z=\pi^{-1}(z) = E^{g-1}\times\{z\}\subseteq E^g$. 
By Grauert's theorem, for $z\in E$ the $z$-fiber of the map
\begin{equation}\label{eq:llmllm1}
  \pi_*\cL\otimes \pi_*\cL^{\otimes m}  \, \longrightarrow \, \pi_*\cL^{\otimes(m+1)}
\end{equation}
is
\begin{equation}\label{eq:xz}
  H^0(X_z,\cL|_{X_z})\otimes H^0(X_z,\cL^{\otimes m}|_{X_z}) \, \longrightarrow \, H^0(X_z,\cL^{\otimes(m+1)}|_{X_z}).
\end{equation}
The locally free sheaf $\cL|_{X_z}$ on $X_z\cong E^{g-1}$ is a standard divisor of type $(n_{1},\ldots,n_{g-1})$. 
The induction hypothesis ${\bf P}_{g-1,m}$ shows that \Cref{eq:xz} is onto, and hence \Cref{eq:llmllm1} is an epimorphism. 
The domain of \Cref{eq:llmllm1} is semistable by ${\bf P}_{g,m-1}$ and \Cref{le.tens-ss}, and the slopes of   $\pi_*\cL\otimes \pi_*\cL^{\otimes m}$ and 
$\pi_*\cL^{\otimes(m+1)}$ are the same by \Cref{le.add}. It follows that the three terms 
in the exact sequence 
\begin{equation}\label{eq:klll}
  0\to \cK\to \pi_*\cL\otimes \pi_*\cL^{\otimes m} \to \pi_*\cL^{\otimes(m+1)}\to 0
\end{equation}
are all semistable and of equal (positive) slopes. Since $\deg(\cK)>0$, $H^1(E,\cK)=0$, so the long exact cohomology sequence for \Cref{eq:klll} produces the desired surjection \Cref{eq:post}.
\end{proof}

\begin{proof}[Proof of \Cref{pr.l'l''}]
 Let $\ldots, \cL_{-1}, \cL_0, \cL_1, \ldots$ be translates of $\cL$ and write
  \begin{align*}
    \cL'    & \;=\;   \cdots \cL_{-1}\otimes \cL_0\\
    \cL''   &\;=\;  \cL_1\otimes\cL_2\cdots.
  \end{align*}
Now  apply \Cref{le.pm} repeatedly to conclude that the map
  \begin{equation*}
    \cdots \otimes  H^0(E^g,\cL_{-1})\otimes H^0(E^g,\cL_0)\otimes H^0(E^g,\cL'') \; \longrightarrow \;  H^0(E^g,\cL'\otimes \cL'')
  \end{equation*}
  is onto, and note that the latter map factors through \Cref{eq:l'l''}. 
\end{proof}

\section{Relations for $B(E^g,\sigma,\cL)$}
\label{se.rel.ep}

In this section,  $\s$ is an arbitrary translation automorphism of $E^g$ and we assume that $\cL:=\cL_{n/k}$ is very ample or, 
equivalently, that $n_i \ge 3$ for all $i$. We will show that the relations for $B(E^g,\s,\cL)$ are generated in degree $\le 3$.

\subsection{Notation}
\label{sect9.notation}
Throughout \S\ref{se.rel.ep}, $\pi:E^g \to E$ denotes the morphism $\pi(x_1,\ldots,x_g)=x_g$.
We will often write $X$ for $E^g$ and use the following notation for various $\cO_{E^g}$-modules:
\begin{align*}
\cM  \;=\; \cM_m & \; =\;   \sigma^*\cL\otimes\cdots\otimes (\sigma^*)^m \cL,
\\
\cN   \;=\; \cN_m& \; =\;  (\sigma^{m+1})^*\cL,
\\
\cK \;=\; \cK_m  & \; =\;  \ker \big(   H^0(E^g,\cN_m) \otimes \cO_X \twoheadrightarrow   \cN_m\big),
\\
\cG \;=\; \cG_m & \;=\;  \cM \otimes \cK\; =\;  \ker\big(\cM \otimes   H^0(E^g,\cN)   \twoheadrightarrow  \cM \otimes \cN\big).
\end{align*}

Since $\cL$ is very ample so is $\cN_m$ for all $m \ge 0$. Hence both are generated by their global sections. 
By \cref{le.push} and \cref{le.semist.pos.deg}\cref{semist.pos.deg.slope}, $\pi_*\cL$ and $\pi_*\cN_m$ are also generated by their global sections.

By \cref{lem.SS92}, the relations for $B(E^g,\s,\cL)$ are generated in degree $\le 3$ if and only if the  map 
 \begin{equation}
\label{onto?}
H^0(E^g,\cL) \otimes H^0(E^g,\cG_m)\, \longrightarrow \, H^0(E^g,\cL \otimes \cG_m)
  \end{equation}
is onto for all $m \ge 2$. 
As in previous sections, we use $\pi_*$ to reduce the question of whether \cref{onto?} is onto to similar questions on $E$.

\subsection{The surjectivity and kernel of the map $Q_{n,k}(E,\tau) \to B(E^g,\s,\cL_{n/k})$ when $X_{n/k}=E^g$}
For the rest of this section we fix an integer $m\geq 2$.

The map in \cref{onto?} factors as
\begin{equation*}
\xymatrix{ 
H^0(E^g,\cL)\otimes H^0(E^g,\cG_m) \ar@{=}[d]  \ar[rr] && H^{0}(\cL\otimes\cG_m)  \ar@{=}[d] 
\\
H^{0}(E,\pi_{*}\cL)\otimes H^{0}(E,\pi_{*}\cG_m) \ar[r]^<<<<<{\alpha'} & H^{0}(E,\pi_{*}\cL\otimes\pi_{*}\cG_m)  \ar[r]^<<<<<{\beta'} & 
H^{0}(E^g,\pi_{*}(\cL\otimes\cG_m))\rlap{.}
}
\end{equation*}

\begin{lemma}\label{7892349023}
	Let
	\begin{align*}
		&\alpha\colon\pi_{*}\cM\otimes H^{0}(E^g,\cN)  \; \longrightarrow \;  \pi_{*}\cM\otimes\pi_{*}\cN,
		\\
		&\beta\colon\pi_{*}\cM\otimes\pi_{*}\cN  \; \longrightarrow \;  \pi_{*}(\cM\otimes\cN)
	\end{align*}
	be the canonical morphisms. There is an exact sequence
	\begin{equation}\label{0924908234}
		0 \; \longrightarrow \;   \ker(\alpha)   \; \longrightarrow \;   \pi_{*}\cG  \; \longrightarrow \;   \ker(\beta)  \; \longrightarrow \;  0.
	\end{equation}
\end{lemma}
\begin{proof}
By its very definition, $\cG$ fits into an exact sequence 
 $
  0 \to \cG \to  \cM  \otimes H^0(E^g,\cN) \to  \cM \otimes \cN  \to 0.
$ 
There is an associated   exact sequence $0 \to \pi_*\cG \to \pi_* \cM  \otimes H^0(E^g,\cN) \to  \pi_*(\cM \otimes \cN)$ in which the right-most map factors as 
  $$
 \pi_* \cM  \otimes H^0(E^g,\cN) \, \stackrel{\alpha}{\longrightarrow} \, 
 \pi_* \cM  \otimes \pi_*\cN\, \stackrel{\beta}{\longrightarrow} \,    \pi_*(\cM \otimes \cN).
$$
Thus $\pi_*\cG = \a^{-1}(\ker(\b))$.
Since $\pi_*\cN$ is generated by its global sections, $H^0(E,\pi_*\cN) \otimes \cO_E \to \pi_*\cN$ is epic. Hence
$\a$ is also epic. Its restriction $\a^{-1}(\ker(\b))\to \ker(\b)$ is therefore epic too.  
\end{proof}

\begin{lemma}\label{3102948120}
	The sequence
	\begin{equation}
	\label{12345}
		0 \; \longrightarrow \;  \pi_{*}\cK \; \longrightarrow \;  H^{0}(E,\pi_{*}\cN)\otimes\cO_{E} \; \longrightarrow \;  \pi_{*}\cN \; \longrightarrow \;  0
	\end{equation}
	is exact, and $\pi_{*}\cK$ is indecomposable.
\end{lemma}
\begin{proof}
Applying \cite[Prop.~III.9.3]{Hart} to the projections from $E^{g}=E^{g-1}\times E$ to its two factors we see that 
$\pi_*\cO_{E^g}=\cO_E$. 
	Also, $H^{0}(E^g,\cN)=H^{0}(E,\pi_{*}\cN)$.
	
	Since $\cN$ is generated by its global sections, the sequence
	\begin{equation}\label{0312490582}
		0\; \longrightarrow \; \cK \; \longrightarrow \; H^{0}(E^g,\cN)\otimes\cO_{E^g}\; \longrightarrow \; \cN \; \longrightarrow \;  0
	\end{equation}
	is exact.
	By \cref{le.push}, $\pi_{*}\cN$ is indecomposable of slope $>2$ and therefore generated by global sections 
	by \cref{le.semist.pos.deg}\cref{semist.pos.deg.slope}. Thus, applying $\pi_{*}$ to \cref{0312490582} produces the exact sequence \cref{12345}.
	
	The indecomposability of $\pi_{*}\cK$ follows in the same way as \cref{le.otimes} or \cref{re.to}.
\end{proof}

\begin{lemma}\label{8430925089}\label{213798437}
Let $\a$ and $\b$ be the maps in \cref{7892349023}.
Both  $\ker(\alpha)$ and $\ker(\b)$ are semistable locally free $\cO_E$-modules of slope $>2$.
\end{lemma}
\begin{proof}
	(1)
	First we deal with $\ker(\a)$. 
	
	By \cref{le.pm},  $\pi_{*}\cM$ is semistable. By \cref{le.add} and \cref{prop.grauert}\cref{item.grauert.slope}, $\mu(\pi_{*}\cM)=m\mu(\pi_{*}\cL) = m\cdot\frac{n}{k'}$.

	Since $H^{0}(E^g,\cN)=H^{0}(E,\pi_{*}\cN)$, 
	\begin{equation*}
		\ker(\alpha) \; = \; \pi_{*}\cM\otimes\ker\big(H^{0}(E^g,\cN)\otimes\cO_{E}\to\pi_{*}(\cN)\big) \;=\;  \pi_{*}\cM\otimes\pi_{*}\cK
	\end{equation*}
	by \cref{3102948120}. 
	Because \cref{12345} is  exact, 
	the degree and rank of $\pi_*\cK$ are
	\begin{align*}
		\deg(\pi_{*}\cK)&\; =\; -\deg(\pi_{*}\cN)\;=\; - n,\\
		\rank(\pi_{*}\cK)&\;=\; \dim H^{0}(\pi_{*}\cN)-\rank(\pi_{*}\cN)\;=\;
		n-k'.
	\end{align*}
	by \cref{prop.grauert}\cref{item.grauert.loc.free}.
	Recall that $\frac{n}{k'}=[n_g,\ldots,n_1]$. Because all $n_i$'s are $\ge 3$, $n/k'>2$. 
	Hence $n-k'>k'$, and
	\begin{align*}
		\mu(\ker\alpha)& \; =\; \mu(\pi_{*}\cM)+\mu(\pi_{*}\cK)\; =\; m\cdot\frac{n}{k'}-\frac{n}{n-k'}\; =\; n\bigg(\frac{m}{k'}-\frac{1}{n-k'}\bigg)\\
		&>\; n\bigg(\frac{m}{k'}-\frac{1}{k'}\bigg) \; =\; (m-1)\frac{n}{k'}
	\end{align*}
	which is $>2$.

	By \cref{3102948120}, $\pi_{*}\cK$ is indecomposable, and hence $\ker(\alpha)=\pi_{*}\cM\otimes\pi_{*}\cK$ is semistable.
	
	(2)
	Since $\pi_{*}\cM\otimes\pi_{*}\cN$ and $\pi_{*}(\cM\otimes\cN)$ are semistable of slope $(m+1)\mu(\pi_{*}\cL)$, the kernel of $\beta$ has these properties too.
\end{proof}

\begin{lemma}\label{0912398344}
Let $\cG$ be as in \cref{sect9.notation}.
Every indecomposable summand of $\pi_{*}\cG$ has slope $>2$.
\end{lemma}
\begin{proof}	
	By \cref{7892349023}, there is an exact sequence $0\to\ker(\alpha)\to\pi_{*}\cG\to\ker(\beta)\to 0$. 
	Since $\ker(\alpha)$ and $\ker(\beta)$ are semistable of slope $>2$, the result follows 
	from \cref{lem.slope.summands}.
\end{proof}

\begin{lemma}\label{8902340983}
	The multiplication map $\alpha'\colon H^{0}(E,\pi_{*}\cL)\otimes H^{0}(E,\pi_{*}\cG) \to  H^{0}(E,\pi_{*}\cL \otimes  \pi_{*}\cG)$ is onto.
\end{lemma}
\begin{proof}
	We use \cref{thm.good.U.V}, the assumptions of which can be weakened as follows: all indecomposable summands of $\cU$ and $\cV$ have slopes $>2$. The result therefore follows from \cref{0912398344,le.push}.
\end{proof}

Our goal in this section is to show that the map \cref{onto?} is onto. We will prove this by induction on $g$.

\begin{lemma}\label{0912309823}
	Suppose that the map \cref{onto?} is onto when $g$ is replaced by $g-1$. The canonical map $\varphi:\pi_{*}\cL\otimes\pi_{*}\cG\to\pi_{*}(\cL\otimes\cG)$ is an epimorphism.
\end{lemma}
\begin{proof}
	The claim is trivial when $g=1$ so we assume $g\geq 2$.
	
	Let $z \in E$ and let $X_z=\pi^{-1}({z} )\cong E^{g-1}$. By Grauert's theorem, the fiber of $\varphi$ 
	at $z$ is the map
	\begin{equation}\label{0983240894}
		H^{0}(X_z,\cL|_{X_{z}})\otimes H^{0}(X_z,\cG|_{X_{z}})  \, \longrightarrow \,  H^{0}(X_z,\cL|_{X_{z}}\otimes\cG|_{X_{z}}).
	\end{equation}
	It suffices to show that this map is surjective for all $z$. We will do that.
	
	We will argue by induction on $g$. By replacing $E^g$ by $X_z$, $\cM$ by $\cM|_{X_{z}}$, and $\cN$ by $\cN|_{X_{z}}$ in the definitions of $\cK$ and $\cG$ we obtain sheaves 
	on $X_z\cong E^{g-1}$ that we denote by $\cK'$ and $\cG'$, respectively. 
	Since $\cN$ is locally free, the restriction of \cref{0312490582}  to $X_z$ is still exact. 
	Thus we obtain a commutative diagram with exact rows,
	\begin{equation*}
		\begin{tikzcd}
			0\ar[r] & \cK|_{X_{z}}\ar[d]\ar[r] & H^{0}(E^g,\cN)\otimes\cO_{X_{z}}\ar[d]\ar[r] & \cN|_{X_{z}}\ar[d,equal]\ar[r] & 0 \\
			0\ar[r] & \cK'\ar[r] & H^{0}(X_z,\cN|_{X_{z}})\otimes\cO_{X_{z}}\ar[r] & \cN|_{X_{z}}\ar[r] & 0\rlap{ .}
		\end{tikzcd}
	\end{equation*}
	The canonical map $H^{0}(E^g,\cN)\to H^{0}(X_z,\cN|_{X_{z}})$ is onto by \cite[Lem.~4.14]{CKS2}.
	Denote its kernel by $V$. 
	By the Snake Lemma there is  an exact sequence
	\begin{equation*}
		0\to V\otimes\cO_{X_{z}}\to\cK|_{X_{z}}\to\cK'\to 0.
	\end{equation*}
	Tensoring this with the locally free sheaf $\cM|_{X_{z}}$ yields the exact sequence
	\begin{equation*}
		0\to V\otimes\cM|_{X_{z}}\to\cG|_{X_{z}}\to\cG'\to 0.
	\end{equation*}
	Since $\cL|_{X_{z}}$ is a standard divisor of type $(n_{1},\ldots,n_{g-1})$, it is ample, and so are $\cM|_{X_{z}}$ and $\cL|_{X_{z}}\otimes\cM|_{X_{z}}$. The argument in \cite[Cor.~3.4]{CKS2} shows $H^{1}(\cM|_{X_{z}})=0$ and $H^{1}(\cL|_{X_{z}}\otimes\cM|_{X_{z}})=0$. Thus we obtain a commutative diagram
	\begin{equation*}
		\begin{tikzcd}
			0\ar[r] & H^{0}(\cL|_{X_{z}})\otimes H^{0}(V\otimes\cM|_{X_{z}})\ar[d]\ar[r] & H^{0}(\cL|_{X_{z}})\otimes H^{0}(\cG|_{X_{z}})\ar[d]\ar[r] & H^{0}(\cL|_{X_{z}})\otimes H^{0}(\cG')\ar[d]\ar[r] & 0 \\
			0\ar[r] & H^{0}(\cL|_{X_{z}}\otimes V\otimes\cM|_{X_{z}})\ar[r] & H^{0}(\cL|_{X_{z}}\otimes\cG|_{X_{z}})\ar[r] & H^{0}(\cL|_{X_{z}}\otimes\cG')\ar[r] & 0
		\end{tikzcd}
	\end{equation*}
	with exact rows. Since the surjectivity of the left (resp.\ right) vertical map 
	is reduced to the case $X=E^{g-1}$, the hypothesis for $g-1$ (resp.\ \cref{le.pm}) ensures the surjectivity. Therefore the vertical map in the middle is also surjective, and this completes the proof.
\end{proof}

We use the following notation in the next proof:  If $\cU_1$ and $\cU_2$ are $\cO_{E^g}$-modules we write 
$
\cK(\cU_1,\cU_2)
$  
for the kernel of the canonical morphism
$\pi_{*}\cU_1\otimes\pi_{*}\cU_2 \, \longrightarrow \, \pi_{*}(\cU_1\otimes\cU_2)$.
If $\cU_1$ and $\cU_2$ are locally free $\cO_{E^g}$-modules, then $\cK(\cU_1,\cU_2)$ is obviously a locally free $\cO_E$-module.

\begin{theorem}\label{0892340894}
Let $\s$ be  a translation automorphism of $E^g$. 
If $\cL_{n/k}$ is very ample, then 
the ideal of relations  for $B(E^{g},\s,\cL_{n/k})$ is generated by elements of degree $\leq 3$.
\end{theorem}
\begin{proof}
As explained at the beginning of this section, it suffices to show that the map \cref{onto?} is onto. We will do that, using induction on $g$ so that we can use the conclusion of \cref{0912309823}.
First, we will show that $H^{1}(E,\cK(\cL,\cG))=0$.

The exact sequence \cref{0924908234} is
	\begin{equation}\label{9408934245}
		0\to\pi_{*}\cM\otimes\pi_{*}\cK\to\pi_{*}\cG\to\cK(\cM,\cN)\to 0.
	\end{equation}
	Replacing $\cM$ by $\cL\otimes\cM$, the same argument produces the exact sequence
	\begin{equation}\label{0910975433}
		0\to\pi_{*}(\cL\otimes\cM)\otimes\pi_{*}\cK\to\pi_{*}(\cL\otimes\cG)\to\cK(\cL\otimes\cM,\cN)\to 0
	\end{equation}
	since $\cG$ is replaced by $\cL\otimes\cG$. 
	The sequences  $(\pi_{*}\cL) \otimes \,$\cref{9408934245} and \cref{0910975433} fit into the commutative diagram
	\begin{equation*}
		\begin{tikzcd}
			& 0\ar[d] & 0\ar[d] & 0\ar[d] & \\
			0\ar[r] & \cK(\cL,\cM)\otimes\pi_{*}\cK\ar[d]\ar[r] & \cK(\cL,\cG)\ar[d]\ar[r] & \cC\ar[d]\ar[r] & 0 \\
			0\ar[r] & \pi_{*}\cL\otimes\pi_{*}\cM\otimes\pi_{*}\cK\ar[d]\ar[r] & \pi_{*}\cL\otimes\pi_{*}\cG\ar[d]\ar[r] & \pi_{*}\cL\otimes\cK(\cM,\cN)\ar[d]\ar[r] & 0 \\
			0\ar[r] & \pi_{*}(\cL\otimes\cM)\otimes\pi_{*}\cK\ar[d]\ar[r] & \pi_{*}(\cL\otimes\cG)\ar[d]\ar[r] & \cK(\cL\otimes\cM,\cN)\ar[d]\ar[r] & 0 \\
			& 0 & 0 & 0 &
		\end{tikzcd}
	\end{equation*}
	where $\cC$ is a locally free $\cO_{E}$-module, and the exactness of the vertical sequences follows from \cref{0912309823} and the fact that the canonical morphism $\pi_{*}\cL\otimes\pi_{*}\cM\to\pi_{*}(\cL\otimes\cM)$ is an epimorphism. 
Thus, to show $H^1(E,\cK(\cL,\cG))=0$ it suffices to show that $H^1(E,\cK(\cL,\cM)\otimes\pi_{*}\cK)=H^1(E,\cC)=0$.
	
As we showed in \cref{8430925089}, $\pi_{*}\cM\otimes\pi_{*}\cK$ is semistable and has slope $>2$. Thus $\pi_{*}\cL\otimes\pi_{*}\cM\otimes\pi_{*}\cK$ is also semistable and has positive slope. Since $\pi_{*}(\cL\otimes\cM)\otimes\pi_{*}\cK$ has the same property, so does $\cK(\cL,\cM)\otimes\pi_{*}\cK$. On the other hand, the exact sequences
	\begin{equation*}
          0\to\pi_{*}\cL\otimes\cK(\cM,\cN)\to\pi_{*}\cL\otimes\pi_{*}\cM\otimes\pi_{*}\cN\to\pi_{*}\cL\otimes\pi_{*}(\cM\otimes\cN)\to 0
	\end{equation*}
	and
	\begin{equation*}
          0\to\cK(\cL\otimes\cM,\cN)\to\pi_{*}(\cL\otimes\cM)\otimes\pi_{*}\cN\to\pi_{*}(\cL\otimes\cM\otimes\cN)\to 0
	\end{equation*}
	imply $\pi_{*}\cL\otimes\cK(\cM,\cN)$ and $\cK(\cL\otimes\cM,\cN)$ are semistable and have the same positive slope, whence $\cC$ has the same property. It follows from \cref{le.semist.pos.deg} that $H^1(E,\cK(\cL,\cM)\otimes\pi_{*}\cK)=H^1(E,\cC)=0$.
	
	Since $0\to\cK(\cL,\cG)\to\pi_{*}\cL\otimes\pi_{*}\cG\to\pi_{*}(\cL\otimes\cG)\to 0$ is exact and $H^1(E,\cK(\cL,\cG))=0$, the map $H^{0}(E,\pi_{*}\cL\otimes\pi_{*}\cG) \to H^{0}(E,\pi_{*}(\cL\otimes\cG))=H^{0}(E,\cL\otimes\cG)$ is onto. It follows from \cref{8902340983} that the map \cref{onto?} is onto. The proof is complete.
\end{proof}


\bibliography{biblio4}

\def\cprime{$'$}
\providecommand{\bysame}{\leavevmode\hbox to3em{\hrulefill}\thinspace}
\providecommand{\MR}{\relax\ifhmode\unskip\space\fi MR }
\providecommand{\MRhref}[2]{%
  \href{http://www.ams.org/mathscinet-getitem?mr=#1}{#2}
}
\providecommand{\href}[2]{#2}
\begin{thebibliography}{ATVdB91}

\bibitem[AS87]{AS87}
M.~Artin and W.~F. Schelter, \emph{Graded algebras of global dimension {$3$}},
  Adv. in Math. \textbf{66} (1987), no.~2, 171--216. \MR{917738 (88k:16003)}

\bibitem[Ati57]{Atiyah}
M.~F. Atiyah, \emph{Vector bundles over an elliptic curve}, Proc. London Math.
  Soc. (3) \textbf{7} (1957), 414--452. \MR{0131423}

\bibitem[ATVdB90]{ATV1}
M.~Artin, J.~Tate, and M.~Van~den Bergh, \emph{Some algebras associated to
  automorphisms of elliptic curves}, The {G}rothendieck {F}estschrift, {V}ol.\
  {I}, Progr. Math., vol.~86, Birkh\"auser Boston, Boston, MA, 1990,
  pp.~33--85. \MR{1086882 (92e:14002)}

\bibitem[ATVdB91]{ATV2}
\bysame, \emph{Modules over regular algebras of dimension {$3$}}, Invent. Math.
  \textbf{106} (1991), no.~2, 335--388. \MR{1128218 (93e:16055)}

\bibitem[AVdB90]{AV90}
M.~Artin and M.~Van~den Bergh, \emph{Twisted homogeneous coordinate rings}, J.
  Algebra \textbf{133} (1990), no.~2, 249--271. \MR{1067406 (91k:14003)}

\bibitem[AZ94]{AZ94}
M.~Artin and J.~J. Zhang, \emph{Noncommutative projective schemes}, Adv. Math.
  \textbf{109} (1994), no.~2, 228--287. \MR{1304753 (96a:14004)}

\bibitem[Bel80]{bel}
A.~A. Belavin, \emph{Discrete groups and integrability of quantum systems},
  Funktsional. Anal. i Prilozhen. \textbf{14} (1980), no.~4, 18--26, 95.
  \MR{595725}

\bibitem[But94]{Butler94}
D.~C. Butler, \emph{Normal generation of vector bundles over a curve}, J.
  Differential Geom. \textbf{39} (1994), no.~1, 1--34. \MR{1258911}

\bibitem[CC81]{ch2}
D.~V. Chudnovsky and G.~V. Chudnovsky, \emph{Completely {$X$}-symmetric
  {$S$}-matrices corresponding to theta functions}, Phys. Lett. A \textbf{81}
  (1981), no.~2-3, 105--110. \MR{597095}

\bibitem[CC93]{CaCi93}
F.~Catanese and C.~Ciliberto, \emph{Symmetric products of elliptic curves and
  surfaces of general type with {$p_g=q=1$}}, J. Algebraic Geom. \textbf{2}
  (1993), no.~3, 389--411. \MR{1211993}

\bibitem[Che82]{cher}
I.~V. Cherednik, \emph{On the properties of factorized {$S$}\ matrices in
  elliptic functions}, Yadernaya Fiz. \textbf{36} (1982), no.~2, 549--557
  (1983). \MR{700961}

\bibitem[Che86]{Ch86-R-matrix}
\bysame, \emph{On {$R$}-matrix quantization of formal loop groups}, Group
  theoretical methods in physics, {V}ol.\ {II} ({Y}urmala, 1985), VNU Sci.
  Press, Utrecht, 1986, pp.~161--180. \MR{919789}

\bibitem[CKS18]{CKS1}
A.~{Chirvasitu}, R.~{Kanda}, and S.~P. {Smith}, \emph{{Feigin and Odesskii's
  elliptic algebras}}, arXiv:1812.09550v3.

\bibitem[CKS19a]{CKS5}
\bysame, \emph{{Finite quotients of powers of an elliptic curve}},
  arXiv:1905.06710v3.

\bibitem[CKS19b]{CKS2}
\bysame, \emph{{The characteristic variety for Feigin and Odesskii's elliptic
  algebras}}, arXiv:1903.11798v4.

\bibitem[CKS20]{CKS4}
\bysame, \emph{{Elliptic R-matrices and Feigin and Odesskii's elliptic
  algebras}}, arXiv:2006.12283v1.

\bibitem[CSW18]{CSW}
A.~Chirvasitu, S.~P. Smith, and L.~Z. Wong, \emph{Noncommutative geometry of
  homogenized quantum {$\frak{sl}(2,\Bbb C)$}}, Pacific J. Math. \textbf{292}
  (2018), no.~2, 305--354. \MR{3733976}

\bibitem[FO89]{FO-Kiev}
B.~L. Feigin and A.~V. Odesskii, \emph{Sklyanin algebras associated with an
  elliptic curve}, Preprint deposited with Institute of Theoretical Physics of
  the Academy of Sciences of the Ukrainian SSR (1989), 33 pages.

\bibitem[Gus90]{Gushel}
N.~P. Gushel\cprime, \emph{Very ample divisors on projective bundles over an
  elliptic curve}, Mat. Zametki \textbf{47} (1990), no.~6, 15--22, 158.
  \MR{1074523}

\bibitem[Har77]{Hart}
R.~Hartshorne, \emph{Algebraic geometry}, Springer-Verlag, New York-Heidelberg,
  1977, Graduate Texts in Mathematics, No. 52. \MR{0463157 (57 \#3116)}

\bibitem[Hul86]{Hulek86}
K.~Hulek, \emph{Projective geometry of elliptic curves}, Ast\'erisque (1986),
  no.~137, 143. \MR{845383}

\bibitem[Kee00]{Keeler-Criteria}
D.~S. Keeler, \emph{Criteria for {$\sigma$}-ampleness}, J. Amer. Math. Soc.
  \textbf{13} (2000), no.~3, 517--532. \MR{1758752}

\bibitem[Kem91]{Kempf-book}
G.~R. Kempf, \emph{Complex abelian varieties and theta functions},
  Universitext, Springer-Verlag, Berlin, 1991. \MR{1109495}

\bibitem[Koi76]{Koi76}
S.~Koizumi, \emph{Theta relations and projective normality of {A}belian
  varieties}, Amer. J. Math. \textbf{98} (1976), no.~4, 865--889. \MR{480543}

\bibitem[Lev92]{Lev92}
T.~Levasseur, \emph{Some properties of noncommutative regular graded rings},
  Glasgow Math. J. \textbf{34} (1992), no.~3, 277--300. \MR{1181768
  (93k:16045)}

\bibitem[LS93]{LS93}
T.~Levasseur and S.~P. Smith, \emph{Modules over the {$4$}-dimensional
  {S}klyanin algebra}, Bull. Soc. Math. France \textbf{121} (1993), no.~1,
  35--90. \MR{1207244 (94f:16054)}

\bibitem[Mar81]{Mar81}
M.~Maruyama, \emph{The theorem of {G}rauert-{M}\"ulich-{S}pindler}, Math. Ann.
  \textbf{255} (1981), no.~3, 317--333. \MR{615853}

\bibitem[MFK94]{mumf-git}
D.~Mumford, J.~Fogarty, and F.~Kirwan, \emph{Geometric invariant theory}, third
  ed., Ergebnisse der Mathematik und ihrer Grenzgebiete (2) [Results in
  Mathematics and Related Areas (2)], vol.~34, Springer-Verlag, Berlin, 1994.
  \MR{1304906}

\bibitem[Mum70]{mum-q}
D.~Mumford, \emph{Varieties defined by quadratic equations}, Questions on
  {A}lgebraic {V}arieties ({C}.{I}.{M}.{E}., {III} {C}iclo, {V}arenna, 1969),
  Edizioni Cremonese, Rome, 1970, pp.~29--100. \MR{0282975}

\bibitem[Mum08]{Mum08}
\bysame, \emph{Abelian varieties}, Tata Institute of Fundamental Research
  Studies in Mathematics, vol.~5, Published for the Tata Institute of
  Fundamental Research, Bombay; by Hindustan Book Agency, New Delhi, 2008, With
  appendices by C. P. Ramanujam and Yuri Manin, Corrected reprint of the second
  (1974) edition. \MR{2514037 (2010e:14040)}

\bibitem[Mur64]{mur}
J.~P. Murre, \emph{On contravariant functors from the category of pre-schemes
  over a field into the category of abelian groups (with an application to the
  {P}icard functor)}, Inst. Hautes \'{E}tudes Sci. Publ. Math. (1964), no.~23,
  5--43. \MR{0206011}

\bibitem[Ode02]{Od-survey}
A.~V. Odesskii, \emph{Elliptic algebras}, Uspekhi Mat. Nauk \textbf{57} (2002),
  no.~6(348), 87--122. \MR{1991863}

\bibitem[OF89]{FO89}
A.~V. Odesskii and B.~L. Feigin, \emph{Sklyanin elliptic algebras},
  Funktsional. Anal. i Prilozhen. \textbf{23} (1989), no.~3, 45--54, 96.
  \MR{1026987 (91e:16037)}

\bibitem[Pol03]{Polishchuk-book}
A.~Polishchuk, \emph{Abelian varieties, theta functions and the {F}ourier
  transform}, Cambridge Tracts in Mathematics, vol. 153, Cambridge University
  Press, Cambridge, 2003. \MR{1987784}

\bibitem[Pol05]{Polizzi}
F.~Polizzi, \emph{On surfaces of general type with {$p_g=q=1,\ K^2=3$}},
  Collect. Math. \textbf{56} (2005), no.~2, 181--234. \MR{2154303}

\bibitem[PP04]{PP04}
G.~Pareschi and M.~Popa, \emph{Regularity on abelian varieties. {II}. {B}asic
  results on linear series and defining equations}, J. Algebraic Geom.
  \textbf{13} (2004), no.~1, 167--193. \MR{2008719}

\bibitem[Skl82]{Skl82}
E.~K. Sklyanin, \emph{Some algebraic structures connected with the
  {Y}ang-{B}axter equation}, Funktsional. Anal. i Prilozhen. \textbf{16}
  (1982), no.~4, 27--34, 96. \MR{684124 (84c:82004)}

\bibitem[Skl83]{Skl83}
\bysame, \emph{Some algebraic structures connected with the {Y}ang-{B}axter
  equation. {R}epresentations of a quantum algebra}, Funktsional. Anal. i
  Prilozhen. \textbf{17} (1983), no.~4, 34--48. \MR{725414 (85k:82011)}

\bibitem[Smi04]{SPS15}
S.~P. Smith, \emph{Maps between non-commutative spaces}, Trans. Amer. Math.
  Soc. \textbf{356} (2004), no.~7, 2927--2944. \MR{2052602}

\bibitem[{Smi}16]{SPS15-corrigendum}
S.~P. {Smith}, \emph{{Corrigendum to ``Maps between non-commutative spaces''
  [Trans. Amer. Math. Soc., 356(7) (2004) 2927-2944]}}, Trans. Amer. Math. Soc.
  \textbf{368} (2016), no.~7, 8295--8302 (electronic). \MR{2052602
  (2005f:14004)}

\bibitem[SS92]{SS92}
S.~P. Smith and J.~T. Stafford, \emph{Regularity of the four-dimensional
  {S}klyanin algebra}, Compositio Math. \textbf{83} (1992), no.~3, 259--289.
  \MR{1175941 (93h:16037)}

\bibitem[SS15]{sam-snow}
S.~V. Sam and A.~Snowden, \emph{Stability patterns in representation theory},
  Forum Math. Sigma \textbf{3} (2015), e11, 108. \MR{3376738}

\bibitem[ST94]{ST94}
S.~P. Smith and J.~T. Tate, \emph{The center of the {$3$}-dimensional and
  {$4$}-dimensional {S}klyanin algebras}, Proceedings of {C}onference on
  {A}lgebraic {G}eometry and {R}ing {T}heory in honor of {M}ichael {A}rtin,
  {P}art {I} ({A}ntwerp, 1992), vol.~8, 1994, pp.~19--63. \MR{1273835}

\bibitem[ST01]{st}
P.~Seidel and R.~Thomas, \emph{Braid group actions on derived categories of
  coherent sheaves}, Duke Math. J. \textbf{108} (2001), no.~1, 37--108.
  \MR{1831820}

\bibitem[SVdB13]{SVdB-NCQ}
S.~P. Smith and M.~Van~den Bergh, \emph{Noncommutative quadric surfaces}, J.
  Noncommut. Geom. \textbf{7} (2013), no.~3, 817--856. \MR{3108697}

\bibitem[Swe69]{Swe69}
M.~E. Sweedler, \emph{Hopf algebras}, Mathematics Lecture Note Series, W. A.
  Benjamin, Inc., New York, 1969. \MR{0252485 (40 \#5705)}

\bibitem[{The}18]{stacks-project}
{The Stacks Project Authors}, \emph{\textit{Stacks Project}},
  \url{https://stacks.math.columbia.edu}, 2018.

\bibitem[Tra85]{tra}
C.~A. Tracy, \emph{Embedded elliptic curves and the {Y}ang-{B}axter equations},
  Phys. D \textbf{16} (1985), no.~2, 203--220. \MR{796270}

\bibitem[Tu93]{tu}
L.~W. Tu, \emph{Semistable bundles over an elliptic curve}, Adv. Math.
  \textbf{98} (1993), no.~1, 1--26. \MR{1212625}

\bibitem[TVdB96]{TvdB96}
J.~T. Tate and M.~Van~den Bergh, \emph{Homological properties of {S}klyanin
  algebras}, Invent. Math. \textbf{124} (1996), no.~1-3, 619--647. \MR{1369430
  (98c:16057)}

\bibitem[VdB01]{VdB-blowup}
M.~Van~den Bergh, \emph{Blowing up of non-commutative smooth surfaces}, Mem.
  Amer. Math. Soc. \textbf{154} (2001), no.~734, x+140. \MR{1846352
  (2002k:16057)}

\end{thebibliography}
\bibliographystyle{customamsalpha}

\end{document}